\documentclass[12pt]{article}
\usepackage{mlmodern}
\usepackage{pdfrender,xcolor}
\usepackage{graphicx,color}
\usepackage{amsmath,amsfonts,amssymb,amsthm,MnSymbol}
\usepackage[hidelinks,colorlinks=true,linkcolor=blue,citecolor=blue]{hyperref}
\usepackage{enumerate}
\usepackage{algorithm}
\usepackage{algpseudocode}
\usepackage{rsfso}
\usepackage{bbm}
\usepackage{cancel}

\usepackage{graphicx}
\usepackage{subcaption}
\usepackage{wrapfig}
\usepackage{tikz} %

\usepackage{natbib}
\usepackage{ulem}

\definecolor{darkgreen}{rgb}{0.06, 0.56, 0.2}

\newcommand{\blue}[1]{{\textcolor{blue}{#1}}}

\newcommand{\E}{\mathbb{E}}
\newcommand{\F}{\mathbb{F}}
\newcommand{\R}{\mathbb{R}}
\newcommand{\Prob}{\mathbb{P}}
\newcommand{\Z}{\mathbb{Z}}

\newcommand{\bS}{\mathbb{S}}

\def\cala{{\mathcal A}}

\def\calc{{\mathcal C}}
\def\cale{{\mathcal E}}
\def\calf{{\mathcal F}}

\def\calh{{\mathcal H}}
\def\calj{{\mathcal J}}
\def\calk{{\mathcal K}}
\def\call{{\mathcal L}}
\def\cals{{\mathcal S}}

\DeclareMathOperator{\diag}{diag}

\DeclareMathOperator{\Var}{Var}

\newcommand{\pend}{\hfill \thicklines \framebox(6.6,6.6)[l]{}}

\newenvironment{proof*}[1]{\noindent {\sc  #1} \rm}{\pend}

\usepackage{mathtools}
\DeclarePairedDelimiter{\dotp}{\langle}{\rangle}

\newtheorem{theorem}{Theorem}[section]
\newtheorem{lemma}{Lemma}[section]
\newtheorem{assumption}{Assumption}[section]
\newtheorem{proposition}{Proposition}[section]

\newtheorem{remark}{Remark}[section]

\newcommand{\setsection}[2] {
	\setcounter{section}{#1}
	\setcounter{subsection}{0}
	\setcounter{equation}{0}
	\setcounter{conjecture}{0}
	\setcounter{assumption}{0}
	\setcounter{question}{0}
	\setcounter{definition}{0}
	\setcounter{theorem}{0}
	\setcounter{corollary}{0}
	\setcounter{lemma}{0}
	\setcounter{proposition}{0}
	\setcounter{remark}{0}
	\setcounter{appen}{0}
	\setsection*{\large \bf \thesection. #2}}

\allowdisplaybreaks

\begin{document}
	\title{\bf \Large  Asymptotic Product-form Steady-state for Multiclass Queueing Networks with SBP Service Policies in Multi-scale Heavy Traffic}
	
	\author{J.G. Dai\\Cornell University\\ \and Dongyan (Lucy) Huo\\Cornell University\\}
	\date{\today \medskip\\ }
	\normalsize 
	
	
	\maketitle
	
	\begin{abstract}
		In this work, we study the stationary distribution of the scaled queue length vector process in multiclass queueing networks operating under static buffer priority service policies. We establish that when subjected to a multi-scale heavy traffic condition, the stationary distribution converges to a product-form limit, with each component in the product form following an exponential distribution. 
		A major assumption in proving the desired product-form limit is the uniform moment bound for scaled queue lengths.
		We prove this assumption holds if
		the unscaled high-priority queue lengths have uniform moment bound and
		a certain reflection matrix is a $\mathcal{P}$-matrix.

	\end{abstract}
	
	\begin{quotation}
		\noindent {\bf Keywords}: multiclass queueing networks, product-form stationary distribution, heavy traffic approximation, performance analysis
		
		\medskip
		
		\noindent {\bf Mathematics Subject Classification}: 60K25, 60J27, 60K37
	\end{quotation}

	\section{Introduction}
	\label{sec:intro}
	
	This paper studies multiclass queueing networks (MCNs) with general inter-arrival and service time distributions in a multi-scale heavy traffic. We prove that, in this multi-scale heavy-traffic regime, the steady state of the scaled queue length process converges to a unique product-form limit, where each component is independently distributed and follows an exponential distribution. We provide the first closed-form formula for general MCNs with general inter-arrival and service time distributions, enabling quick performance approximations without the need for complex and time-consuming simulations.
	
	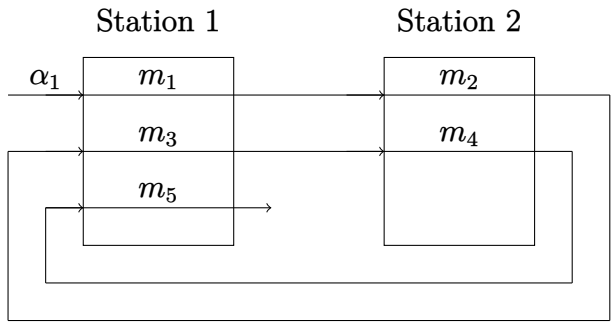
\begin{figure}[htbp]
		\centering
		\begin{tikzpicture}
			\draw (0, -0.5) rectangle (2, 2);
			\draw (4, -0.5) rectangle (6, 2);
			\draw[-] (-1, 1.5) -- (7,1.5);
			\draw[-] (-1, 0.75) -- (6.5,0.75);
			\draw[-] (7, 1.5) -- (7,-1.5);
			\draw[-] (6.5, 0.75) -- (6.5,-1);
			\draw[-] (6.5, -1) -- (-0.5,-1);
			\draw[-] (7, -1.5) -- (-1,-1.5);
			\draw[-] (-1, -1.5) -- (-1,0.75);
			\draw[-] (-0.5, -1) -- (-0.5,0);
			\draw[->] (-0.5, 0) -- (2.5,0);
			\draw[->] (-0.5, 0) -- (0,0);
			\draw[->] (-0.5, 1.5) -- (0,1.5);
			\draw[->] (-0.5, 0.75) -- (0,0.75);
			
			\draw[->] (3.5, 1.5) -- (4,1.5);
			\draw[->] (3.5, 0.75) -- (4,0.75);
			
			\draw (-0.5, 1.7) node {$\alpha_1$};
			\draw (1, 1.7) node {$m_1$};
			\draw (1, 0.95) node {$m_3$};
			\draw (1, 0.2) node {$m_5$};
			
			\draw (5, 1.7) node {$m_2$};
			\draw (5, 0.95) node {$m_4$};
			
			\draw (1, 2.5) node {Station 1};
			\draw (5, 2.5) node {Station 2};
			
		\end{tikzpicture}
		\caption{Two-Station Five-Class Re-entrant Line}
		\label{fig:2s5c-fig}
	\end{figure}
	
	To illustrate the general results, we consider a simple two-station five-class re-entrant line as shown in Figure~\ref{fig:2s5c-fig}. 
	Jobs arrive externally following a renewal process, with the inter-arrival times $\{T_{e,1}(i)/\alpha_1, i\geq1\}$ assumed to be independently identically distributed (i.i.d.) with $\E[T_{e,1}(1)]=1$ and $\Var[T_{e,1}(1)]=c_{e,1}^2$.
	Each job needs to go through five different steps in the system before exiting, following the flow depicted in the figure. When a job finishes its processing in the current step, it will move to the next buffer to wait for the start of its next step. We assume that each buffer has infinite capacity.
	Following the framework in~\cite{Harr1988}, we adopt the notion of job classes, and we denote a job to be of class $k$ if it is currently being processed in step $k$ or waiting in buffer $k$ for its turn.
	In the remainder of this paper, we will use the terms ``class'' and ``buffer'' interchangeably.
	The processing times for each class are assumed to be i.i.d.\ with mean $m_k$, and we denote the sequence of processing times of class $k$ jobs as $\{m_kT_{s,k}(i),i\geq1\}$ with $\E[T_{s,k}(1)]=1$ and $\Var[T_{s,k}(1)]=c_{s,k}^2$ for $k=1,\ldots,5$. 
	In this work, we assume that the stations are unitary, meaning that only one job can be processed at a time.
	Upon a job completion, the server must decide on the next job to work on, and this decision gives rise to a service policy. 
	We focus on the static buffer priority (SBP) policies, in which classes at the same station are assigned a strict priority ranking and higher-priority classes are served first. 
	As an example, we study the following SBP service discipline, $\{(5,3,1),(2,4)\}.$
	For jobs at station 1, we assign the highest priority to class $5$, the next priority to class $3$, and the lowest priority to class $1$. For jobs at station 2, we have class $2$ jobs of the highest priority and class $4$ jobs of the lowest. 
	
	Let $\rho_1$ and $\rho_2$ denote the respective traffic intensity at Station 1 and 2, computed as 
	\begin{equation*}
		\rho_1=\alpha_1(m_1+m_3+m_5)<1,\quad\text{and}\quad
		\rho_2=\alpha_1(m_2+m_4)<1.
	\end{equation*}
	Under the specified SBP policy, it is important to place an additional assumption that $\rho_v=\alpha_1(m_2+m_5)<1$, as discussed in \citet{DaiVand2000}, to ensure that the network is stable. With this, the Markov process associated with the system is positive Harris recurrent, under some minor assumptions on inter-arrival and service time distributions. 
	
	Let $Z_k$ denote the steady-state queue length of Class $k$. To understand $Z_k$, we study its heavy traffic limit by considering a sequence of networks in which the traffic intensities at both stations converge to 1. Specifically, we study the following ``multi-scale'' heavy traffic regime: for a sequence of networks indexed by $r$ as $r\downarrow0$, we have 
	\begin{equation*}
		1-\rho_1^{(r)}=rb_1,\quad 1-\rho_2^{(r)}=r^2b_2,\quad\text{and}\quad\lim_{r\to0}\rho_v^{(r)}<1,
	\end{equation*}
	with $b_1,b_2>0$. 
	This separation in the rate at which the traffic intensities of each station converge to 1 is the rationale for terming it the ``multi-scale'' heavy traffic, which reflects the unbalanced workload distribution commonly observed in real-world systems.

	We now are ready to state the result.
	\begin{proposition}
		\label{prop:2s5c-result}
		Assume that there exists a small $\delta_0>0$ such that $\E[T_{e,1}^{3+\delta_0}]<\infty$, and $\E[T_{s,k}^{3+\delta_0}]<\infty$, $k\in\{1,\ldots,5\}.$
		Under the multi-scale heavy-traffic regime, there exists a random vector $(Z_1^\ast, Z_4^\ast)\in\R_+^2$ such that
		\begin{equation*}
			\left((1-\rho_1^{(r)})\begin{pmatrix}
				Z_1^{(r)}\\Z_3^{(r)}\\Z_5^{(r)}
			\end{pmatrix}, (1-\rho_2^{(r)})\begin{pmatrix}
				Z_2^{(r)}\\Z_4^{(r)}
			\end{pmatrix}\right)\Rightarrow \left(\begin{pmatrix}
				Z_1^*\\0\\0
			\end{pmatrix}, \begin{pmatrix}
				0\\Z_4^*
			\end{pmatrix}\right),\quad\text{as }r\to0,
		\end{equation*}
		where ``$\Rightarrow$'' denotes convergence in distribution.
		Moreover, $Z_1^\ast$ and $Z_4^\ast$ are independent and each follows an exponential distribution.
	\end{proposition}
	
	We first note that the high-priority classes vanish in the limit (a form of state space collapse), leaving only the lowest-priority classes at each station. More importantly, these remaining classes exhibit a product-form structure, with each marginal distribution independently following an exponential distribution. This simple closed-form limit enables straightforward computation and approximation of system performance.
	
	To illustrate our result and explore its practical applications, we conduct a set of numerical experiments with $\rho_1=96\%$ and $\rho_2=99\%$, with mean and variance parameters for inter-arrival and service time distributions specified in Section~\ref{sec:experiments}.
	In particular, we compare the exponential distribution to the empirical distribution obtained from the simulation. 
	The histogram in Figure~\ref{fig:z1-intro} demonstrates a close distributional match, especially in the tail, which is valuable for assessing tail latency -- a key performance metric. Additionally, the closed-form formula also enables performance approximation across various policies, allowing us to identify the best-performing SBP policy and optimize system performance, with little computational effort compared to lengthy simulation runs, especially in heavy traffic. For instance, focusing on the average time-in-system, or cycletime, our result suggests that $\{(5,3,1),(4,2)\}$ is the best performing SBP policy, a result subsequently confirmed by simulation. Details on this numerical example and further results are discussed in Section~\ref{sec:experiments}.
	
	\begin{figure}[htbp]
		\centering
		\begin{minipage}[b]{0.33\textwidth} 
			\centering
			\includegraphics[width=\textwidth]{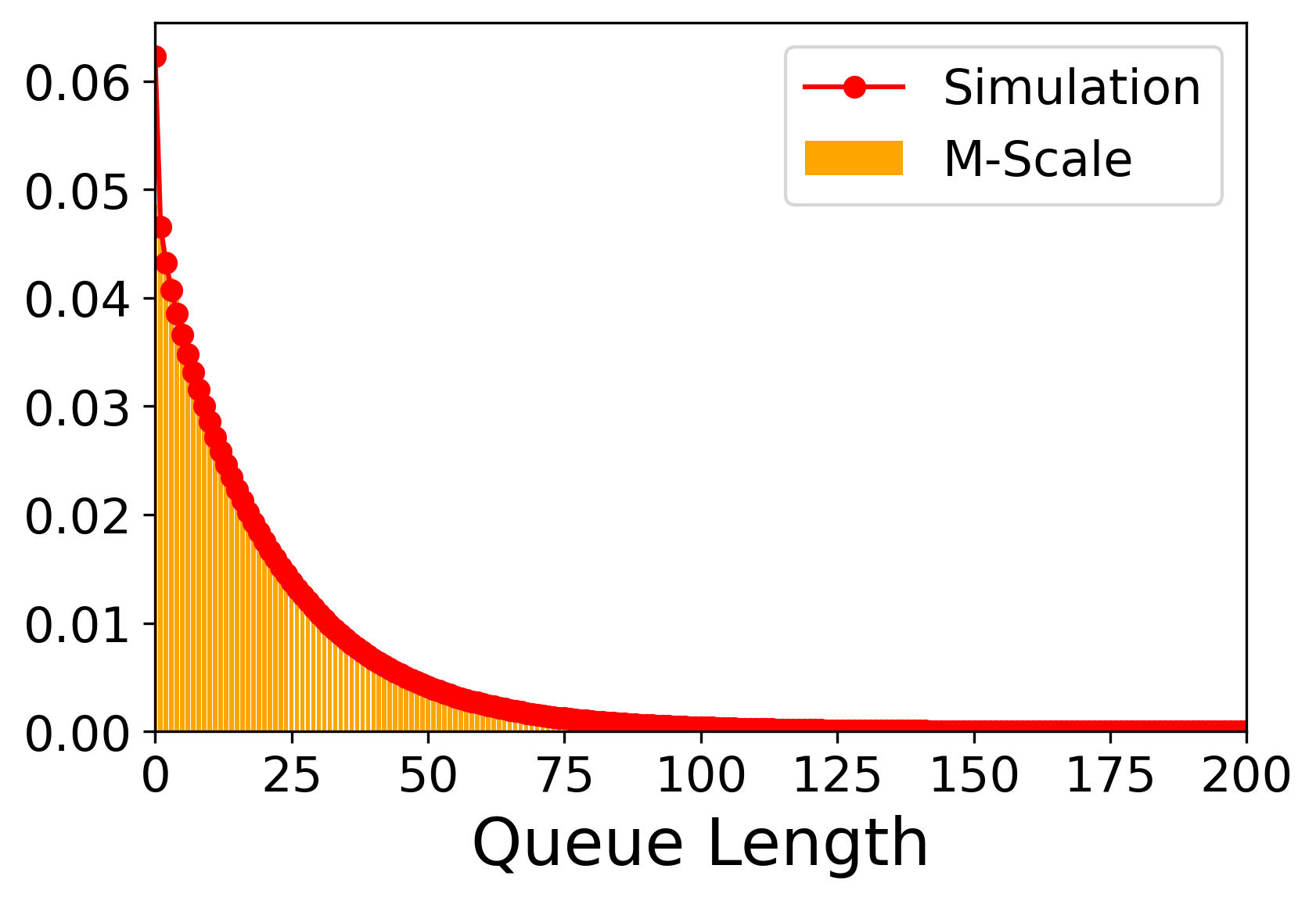} 
			\caption{Histogram of $Z_1$}
			\label{fig:z1-intro}
		\end{minipage}%
		\hspace{1em} 
		\begin{minipage}[b]{0.63\textwidth} 
			\centering
			\begin{tabular}{ccc}
				\hline  
				Priority    &  M-Scale &   Simulation\\\hline 
				$\{(5,3,1),(4,2)\}$    &    \blue{ 134.862 }&        $137.759\pm0.001$\\\hline  
				$\{(3,1,5),(4,2)\}$&     157.891&        $157.102\pm0.001$   \\\hline     
				$\{(5,1,3),(4,2)\}$   &     180.920 &        $194.142\pm0.001$  \\\hline  
				$\{(5,3,1),(2,4)\}$ &     187.828 &        $185.250\pm0.001$ \\\hline  
				$\{(3,1,5),(2,4)\}$  &     217.947 &        $216.077\pm0.001$\\\hline  
				$\{(5,1,3),(2,4)\}$  &     217.947 &        $226.220\pm0.001$\\\hline  
			\end{tabular}
			\captionof{table}{Cycletime Approximation} 
		\end{minipage}
	\end{figure}
	
	The general multiclass network setup is described in Section~\ref{sec:multiclass}, with a general theorem (Theorem~\ref{thm:main}) stated in Section~\ref{sec:main} and proved in Section~\ref{sec:proof}. 
	Even though Proposition~\ref{prop:2s5c-result} is a special case of Theorem~\ref{thm:main}, the full proof of Proposition~\ref{prop:2s5c-result} can be found in the companion paper~\citet{DaiHuo2024a}. We recommend reading the companion paper prior to engaging with the current paper, as the companion paper uses the 2-station 5-class network in Figure~\ref{fig:2s5c-fig} to illustrate our proof technique without the extensive notation required for tracking a multiclass network. 
	
	\citet{DaiGlynXu2023} coined the term ``multi-scale heavy
	traffic''. They study steady-state convergence in generalized Jackson
	networks (GJNs), single-class networks with general inter-arrival and
	service time distributions. They demonstrate that for GJNs, under
	multi-scale heavy traffic, the appropriately scaled steady-state
	queue length converges to a product-form limit.
	Our work establishes that MCNs under SBP service policies also exhibit a product-form steady-state limit in this multi-scale heavy traffic regime, which suggests that the asymptotic product-form steady-state phenomenon observed in the multi-scale heavy traffic regime may hold more broadly across stochastic processing networks.

	To prove the steady-state convergence, we employ the basic adjoint relationship (BAR) approach 
	for continuous-time, discrete-event stochastic systems with general distributions of event times. 
	BAR uses the fundamental theorem of calculus to directly analyze the balance equation at steady state, breaking it down into model primitives such as event time distributions.
	The term ``BAR'' was first introduced in \citet{HarrWill1987} to characterize the stationary distribution of a semimartingale reflecting Brownian motion (SRBM). The BAR approach is formally developed in recent works to prove the steady-state convergence of GJNs \citep{BravDaiMiya2017, Miya2018}.
	Subsequent work by \citet{BravDaiMiya2023} develops the Palm version of the BAR to study the steady-state convergence of MCNs under SBP and demonstrates that the BAR is a natural approach to prove the steady-state convergence.
	Our work employs the BAR analysis in \citet{BravDaiMiya2023}.  The BAR approach has been used in~\citet{DaiGlynXu2023} and~\citet{GuanXuDai2024}.
	
	In addition to proving the steady-state convergence, we also employ BAR to analyze the steady-state moment of scaled queue length processes and establish a uniform moment bound, which is a critical step in establishing the product-form limit.
	A similar uniform moment bound of scaled queue length processes, proven in~\citet{GuanChenDai2023}, is also essential in \citet{DaiGlynXu2023} for proving the asymptotic product-form of GJNs in multi-scale heavy traffic. For MCNs under SBP, the high-priority class moment bound has been previously studied in~\cite{CaoDaiZhan2022}, where the authors provided a sufficient condition for the high-priority class moment collapse. Our work, on the other hand, is the first to establish such a moment bound for scaled queue-length vectors for low-priority classes in general MCNs. To establish the desired uniform moment bound for low-priority classes, we need to carefully utilize the properties of the uniform moment bound for high-priority classes in~\cite{CaoDaiZhan2022} and additionally assume that certain reflection matrix is a $\mathcal{P}$-matrix.

	\subsection{Literature Review}

	``Product-form'' stationary distribution, which implies independence of each component, has received much attention in the literature on stochastic processing networks, due to its mathematical tractability, analytical simplicity, and practical applicability.
	The classical work~\citep{Jack1957} establishes that the queue length vector process of open Jackson networks, a single-class queueing network with exponential inter-arrival and service time distributions, admits a ``product-form'' stationary distribution. 
	This ``product-form'' phenomenon has also been identified in several classical MCNs, such as the BCMP network in \citet{BaskChanMuntPala1975} and Kelly's network in \citet{Kell1975}.
	For a comprehensive summary of product-form distributions up to the 1990s, one can refer to the textbook by~\citet{Serf1999}.
	
	However, many networks do not admit product-form stationary distribution. 
	Even in Jackson networks, when the inter-arrival and service time are generally distributed, known as general Jackson networks (GJNs), the product-form phenomenon no longer holds, which often makes performance evaluation difficult. 
	To evaluate the steady-state performance, a prominent alternative to analyzing the exact system is to study the heavy-traffic limit, where systems operating near full utilization are often more accessible~\citep{King1962}.
	
	The conventional heavy-traffic regime assumes that the traffic
	intensities of all stations in the network converge at the \emph{same
		rate}~\citep{IgleWhit1970, Reim1984}.
	\citet{John1983} and~\citet{Reim1984} analyze single-class GJNs and establish
	that the appropriately scaled queue length process converges weakly to
	their corresponding SRBM limits in conventional heavy traffic.
	Since Harrison first introduced Brownian system models in
	\citet{Harr1988}, there have been many papers establishing (conventional)
	heavy traffic limit (process-level) convergence for MCNs under various
	service policies \citep{Will1998a, Bram1998c, ChenZhan2000,
		ChenZhan2000b, BramDai2001}. 
	
	Process-level convergence does not imply steady-state convergence.
	Before the appearance of \citet{BravDaiMiya2017, BravDaiMiya2023},
	to justify the steady-state convergence, the classical technique that has dominated the literature for the past two decades is to validate the ``interchange of limits.'' 
	This method is first argued in the pioneering paper by
	\citet{GamaZeev2006}, which studies the steady-state convergence in
	GJNs.  Subsequent work in \citet{BudhLee2009} improves the results in
	\citet{GamaZeev2006} by relaxing the assumptions to more standard
	conditions on the network primitives.  Extending the ``interchange of
	limits'' technique to MCNs, the paper by~\citet{Gurv2014a} proves the
	steady-state convergence for MCNs under queue-ratio service
	disciplines, which include SBP service policy as a special
	case. Subsequent works in \citet{YeYao2016, YeYao2018} prove the
	steady-state convergence covering a wider class of service disciplines
	of MCNs. 
	Despite these steady-state convergence results, the
	stationary distributions of these SRBMs are typically not of
	product form, and there have been ongoing effort to design scalable
	algorithms to compute these stationary distributions~\citep{DaiHarr1992, ShenChenDaiDai2002, BlanChenSiGlyn2021}. Our product-form steady-state requires little computational
	effort for carrying out numerical work.

	\section{Multiclass Queueing Networks}
	\label{sec:multiclass}
	
	In this section, we formally introduce multiclass queueing networks (MCNs) and static buffer priority (SBP) service policies. We follow closely the terminology and notation in \cite{BravDaiMiya2023}.
	
	In an MCN, there are $J$ service stations that process $K$ classes of jobs, where $J$ and $K$ are positive integers. We let $\calj=\{1,2,\ldots, J\}$ denote the set of stations and $\calk=\{1,2,\ldots, K\}$ denote the set of classes. Each station is assumed to have a single server and unlimited waiting space. When an external job arrives at the network, the job will be sequentially processed by a finite number of stations and eventually leave the network. At any given time during its lifespan in the network, the job belongs to exactly one of the job classes. As the job moves through the network, the class of the job will be updated at each service completion. The ordered sequence of classes that a job takes is called its route.
	We let $\calc(j)$ denote the set of classes that belong to station $j$ and $s(k)$ denote the station to which class $k$ belongs. 
	
	For each class $k$, we associate it with two non-negative numbers $\alpha_k\geq0$ and $m_k>0$, two independent and identically distributed (i.i.d.) sequences of random variables, $T_{e,k}(\cdot)=\{T_{e,k}(i), i\geq 1\}$ and $T_{s,k}(\cdot)=\{T_{s,k}(i),i\geq1\}$, and a third sequence of $\R^K$-valued i.i.d.\ random vectors $\xi^{(k)}(\cdot)=\{\xi^{(k)}(i), i\geq 1\}$. We assume that the $3K$ sequences
	\begin{equation*}
		T_{e,1}(\cdot),\ldots, T_{e,K}(\cdot),\quad T_{s,1}(\cdot), \ldots, T_{s,K}(\cdot),\quad\text{and}\quad\xi^{(1)}(\cdot),\ldots, \xi^{(K)}(\cdot)
	\end{equation*}
	are defined on a common probability space $(\Omega, \calf, \Prob)$ and mutually independent. Let $T_{e,k}$, $T_{s,k}$ and $\xi^{(k)}$ denote random generic element from respective sequences. Following the convention in~\cite{BravDaiMiya2023}, $T_{e,k}$ and $T_{s,k}$ are unitizied, i.e., $\E[T_{e,k}]=1$ and $\E[T_{s,k}]=1$. For $\xi^{(k)}$, it takes values from the set $\{e^{(0)}, e^{(k)}, k\in\calk\}$, where $e^{(k)}$ is a $K$-dimensional vector with all entries being $0$, except for its $k$-th element being $1$ and $e^{(0)}$ is a $K$-dimensional vector with all entries being $0$.
	
	The value $\alpha_k$ denotes the external arrival rate to class $k$. When a class $k$ does not have external arrivals, we simply set $\alpha_k=0$. We denote the set of classes with external arrivals by
	\begin{equation*}
		\cale=\{k\in\calk, \alpha_k> 0\},\quad E=|\cale|,
	\end{equation*}
	where $|\cals|$ denotes the cardinality of set $\cals$.
	For $k\in\cale$ and $i\geq1$, $T_{e,k}(i)/\alpha_k$ denotes the inter-arrival time between the $(i-1)$th and the $i$th externally arriving job of class $k$.
	For $k\in\calk$, $m_k$ is the mean service time for class $k$ jobs, and $m_kT_{s,k}(i)$ denotes the service time for the $i$th class $k$ job.
	We place the following moment assumption on $T_{e,k}$ and $T_{s,k}$.
	\begin{assumption}
		\label{assumption:t-moment}
		There exist a small $\delta_0>0$ such that
		\begin{equation}
			\label{eq:t-moment}
			\E\left[ (T_{e,k}(1))^{J+1+\delta_0}\right] < \infty,\quad\text{for } k\in\cale\quad \text{and}\quad \E\left[ (T_{s,k}(1))^{J+1+\delta_0}\right] < \infty,\quad\text{for } k\in\calk.
		\end{equation}
	\end{assumption}
	We then set
	\begin{align}
		c_{e,k}^2&=\Var(T_{e,k})/\E[T_{e,k}]^2=\Var(T_{e,k})\quad\text{for } k\in\cale,\label{eq:ce-def}\\
		\text{and}\quad c_{s,k}^2&=\Var(T_{s,k})/\E[T_{s,k}]^2=\Var(T_{s,k})\quad\text{for } k\in\calk,\label{eq:cs-def}
	\end{align}
	which are the squared coefficients of variation for inter-arrival and service times.
	
	The random variable $\xi^{(k)}(i)$ characterize the $i$-th departure from class $k$. Setting $\xi^{(k)}(i)=e^{(l)}$ means that the $i$-th service completion at class $k$ will become a class $l$ job upon service completion at class $k$ for $l\in\calk$ or will leave the network if $l=0$. We introduce the $K\times K$ routing matrix $P=(P_{kl})$,
	\begin{equation*}
		P_{kl}=\Prob(\xi^{(k)}=e^{(l)}),\quad l\in\calk,\quad\text{and}\quad
		P_{k0}=\Prob(\xi^{(k)}=e^{(0)})=1-\sum_{l\in\calk}P_{kl}.
	\end{equation*}
	We focus on open networks in this paper, that is $I-P$ is invertible, which ensures that each job will eventually exit the network.
	
	\paragraph*{SBP Service Discipline.} 
	A service discipline dictates the sequence of jobs to be served at each station. We assume that the service discipline is non-idling, which means that the server is always active as long as there are jobs waiting at the station for service. In this paper, we focus on the static buffer priority (SBP) service discipline, which specifies the service order of jobs based on their classes. At each station, classes at the station are assigned a fixed ranking from the highest priority to the lowest priority. To avoid ambiguity, the ranking is strict and does not contain any ties. The server will always begin processing jobs from the non-empty class of the highest priority, and will only move on to jobs of lower priority classes when no more jobs of strictly higher priority are remaining at the station. 
	For jobs within the same priority class, the server will follow a first-come first-served fashion and serve the leading job or the longest waiting one. 
	Moreover, we assume that the SBP service discipline is preemptive-resume, where the arrival of a job of a higher-priority class leads to the interruption of the current job being processed, with the server switching to the newly arrived, higher-priority job. 
	After completing all jobs of higher-priority classes, the server will resume the previously interrupted job, picking up from where it left off.

	\paragraph*{Traffic Equation.}
	When studying open queueing networks, we make use of the solution denoted by $\lambda_l$ for $l\in\calk$ to the system of traffic equations, 
	\begin{equation*}
		\lambda_k=\alpha_k + \sum_{l\in\calk}\lambda_lP_{lk},\quad k\in\calk,
	\end{equation*}
	or written compactly in matrix-vector format, $\lambda = \alpha + P^\top\lambda$, where the superscript $\top$ denotes the transpose operator. In this paper, all vectors are defaulted to column vectors, unless explicitly stated otherwise. The solution $\lambda$ has expression $\lambda = (I-P^\top)^{-1}\alpha$ and is unique, for $I-P$ is invertible under the open network assumption. $\lambda_k$ denotes the nominal total arrival rate at class $k$, including both external arrivals and internal arrivals from other classes.
	
	Let $\rho_j$ denote the traffic intensity for the $j$th station and $\rho_j$ is computed as 
	\begin{equation}
		\label{eq:rho-j}
		\rho_j=\sum_{k\in\calc(j)}\lambda_km_k.
	\end{equation}
	Rewriting \eqref{eq:rho-j} in matrix-vector format, we have $\rho=CM\lambda$, where $M=\diag(m)$ and $C$ is the $J\times K$ constituency matrix such that for each class $k$, $C_{jk}=1$ if station $j=s(k)$, and 0 otherwise.
	When $\rho_j< 1$, $\rho_j$ often can be interpreted as the fraction of the time that the station $j$ is busy, but this interpretation is not always guaranteed in MCNs. 
	In this paper, we are particularly interested in networks in which $\rho_j<1$ but is close to $1$. Such networks are called ``heavily loaded.'' Moreover, we consider a multi-scale heavy traffic situation. The precise setup will be formally defined in Section~\ref{sec:main}.
	
	\paragraph*{Markov Process.}
	At the time $t\geq0$, for $k\in\calk$, let $Z_k(t)$ denote the number of class $k$ jobs including possibly the one in service, and let $R_{s,k}(t)$ be the remaining service time of a class $k$ job at time $t$. When $Z_k(t)=0$, then set $R_{s,k}(t)$ as the service time of the next class $k$ job. For $k\in\cale$, let $R_{e,k}(t)$ denote the remaining time til the next external arrival of a class $k$ job. Consolidating into vector forms, we have $Z(t), R_e(t), R_s(t)$ be the random vectors with $k$-th entries being $Z_k(t), R_{e,k}(t), R_{s,k}(t)$ respectively. Define
	\begin{equation*}
		X(t)\equiv \big(Z(t), R_e(t), R_s(t)\big),\quad t\geq0,
	\end{equation*}
	on state space $\bS\equiv Z_+^K\times R_+^E\times R_+^K$. It is easy to see that $X(\cdot)\equiv \{X(t), t\geq0\}$ is a Markov process with respect to the filtration $\F^X\equiv\{\calf_t^X, t\geq0\}$, where $\calf_t^X=\sigma(\{X(u), 0\leq u\leq t\})$. We additionally assume that $X(\cdot)$ is right continuous on $[0,\infty)$ and has left limit in $(0,\infty)$.
	When the inter-arrival and service time are exponentially distributed, thanks to the memoryless property of exponential distribution, $\{Z(t), t\geq0\}$ itself is a continuous-time Markov chain on state space $Z_+^K$.

	\section{Multi-scale Heavy Traffic Assumption and the Main Result}
	\label{sec:main}
	\setcounter{equation}{0}
	
	In this section, we first introduce the assumptions that will be used in our main theorem and follow by stating the main result in Theorem~\ref{thm:main}.
	
	We consider a family of MCNs indexed by $r$, with $r$ tending towards $0$ through a strictly decreasing family of values in $(0,1)$. To keep the presentation clean, for each $r\in(0,1)$, we have the mean service time as the only network parameter to depend on $r$ and denote $m_k^{(r)}$ as the mean service time for class $k$ in the $r$-th network. Under this setup, the external arrival rate $\alpha$, unitized inter-arrival and service times $T_{e,k}$ and $T_{s,k}$, and routing matrix $P$ do not depend on $r$, and subsequently the nominal total arrival rate vector $\lambda$ does not depend on $r$. The SBP service policies for the sequence of networks are the same, i.e., the priority order in SBP policies is fixed and does not depend on $r$.
	
	\begin{assumption}[Multi-scale heavy traffic]
		\label{assumption:multi-scale-heavy-traffic}
		We assume that there are $\R^K$-valued vectors $\alpha>0$ with $\alpha_k=0$ for $k\notin\cale$, $m>0$, and $m^\ast$ such that
		\begin{align}
			&m_k^{(r)}=m_k+r^{s(k)}m^\ast_k, \text{for }k\in\calk,\label{eq:mu-heavy}\\
			&M^{(r)}=\diag(m_k^{(r)}),\quad M=\diag(m_k),\quad\text{and}\quad M^\ast=\diag(m_k^\ast),\nonumber\\
			&\rho = CM\lambda=e,\quad b\equiv -CM^\ast\lambda>0,\label{eq:rho-heavy}
		\end{align}
		where $e$ is the $J$-vector of all $1$'s and $\lambda=(I-P^\top)^{-1}\alpha$.
	\end{assumption}
	Equations \eqref{eq:mu-heavy} and \eqref{eq:rho-heavy} imply
	\begin{equation*}
		\rho_j^{(r)}=1-r^jb_j,\quad j\in\calj,\quad\text{and}\quad
		\mu_k^{(r)}=\mu_k+r^{s(k)}\mu_k^\ast+o(r^{s(k)})\quad\text{as }r\to0,
	\end{equation*}
	where $\mu_k^{(r)}=1/m_k^{(r)}$, $\mu_k=1/m_k$ and $\mu_k^\ast=-m_k^\ast/(m_k^2)$. 
	The traffic intensities at stations, $\rho_j^{(r)}$, approach the critical load 1 at separated orders of r. This contrasts with conventional heavy-traffic conditions, where the traffic intensities of all stations converge to the critical load at the same rate. Therefore, we term this assumption as the ``multi-scale heavy traffic condition.''
	
	Throughout this paper, we adopt the following notations. $f(r)=o(g(r))$ means $f(r)/g(r)\to0$ as $r\to0,$
	and $f(r)=O(g(r))$ means $|f(r)| \leq c|g(r)|$ as $r\to0$ for some $c>0$.
	
	Next, we denote the Markov process describing the multiclass network with index $r\in(0,1)$ by $X^{(r)}(\cdot)=\{X^{(r)}(t),t\geq0\}$, where
	\begin{equation}
		\label{eq:x-state}
		X^{(r)}(t)=\Big(Z^{(r)}(t), R_e^{(r)}(t), R_s^{(r)}(t)\Big)\in\bS,\quad t\geq0.
	\end{equation}
	It is well known that simply assuming the nominal traffic intensity $\rho_j^{(r)}<1$ in MCN with SBP service policy will not always lead to a stable network. Therefore, to study the stationary distribution of the Markov process $X^{(r)}(\cdot)$ in this work, we place the following assumption to ensure stability. 
	\begin{assumption}
		\label{assumption:stability}
		For each index $r\in(0,1)$, $X^{(r)}(\cdot)$ has a unique stationary distribution.
	\end{assumption}

	In this paper, we assume the following uniform moment bound for the unscaled high-priority queue length random variables.
	\begin{assumption}
		\label{assumption:high-moment} 
		For each $k\in\calh$, there exists $r_0\in(0,1)$ and small $\epsilon_0>0$ such that
		\begin{equation*}
			\sup_{r\in(0,r_0)}\E\Big[\Big(Z_k^{(r)}\Big)^{J+\epsilon_0}\Big]<\infty.
		\end{equation*}
	\end{assumption}
	
	Under Assumption~\ref{assumption:stability}, Assumption~\ref{assumption:high-moment} intuitively makes sense, as each server would still have excess capacity after serving all high-priority jobs, 
	and hence the queue lengths of high-priority classes will not blow up even when the network becomes critically loaded. It has been proven in~\cite{CaoDaiZhan2022} that Assumption~\ref{assumption:high-moment} holds for re-entrant line under first-buffer-first-serve (FBFS) or last-buffer-first-serve (LBFS) policies. It is still unclear whether Assumption~\ref{assumption:high-moment} would hold for general MCNs, and the work in \cite{CaoDaiZhan2022} provides a sufficient condition for verifying Assumption~\ref{assumption:high-moment} in general MCNs.
	
	Next, we impose the following structural assumption on a reflection matrix. The detailed definition of the reflection matrix is postponed to Section~\ref{sec:postpone-param}, at the end of this section.

	\begin{assumption}
		\label{assumption:q-m-matrix} 
		A certain reflection matrix $R$ is well defined and is a $\mathcal{P}$-matrix.
	\end{assumption}
	The matrix $R$ is often associated with the reflection matrix in the SRBM that serves the diffusion approximation of a multiclass network; see, for example,~\citet{ChenZhan2000b,BravDaiMiya2023}.
	$\mathcal{P}$-matrix is a type of matrix where all its principal minors are positive. This special matrix structure has been studied extensively across many fields, including the stability analysis of SRBM~\citep{KharTahaYaac2002, BramDaiHarr2010}.
	It is a special case of completely-$\mathcal{S}$ matrix, a property that has been assumed in~\citet{BravDaiMiya2023}.
	The $\mathcal{P}$-matrix assumption has been proven to hold under both FBFS and LBFS policies \cite{DaiYehZhou1997}. We provide a more detailed discussion of the reflection matrix $R$ in Appendix~\ref{sec:equivalent-reflection}.
	
	Now, we are ready to state the main result.
	\begin{theorem}
		\label{thm:main}
		Suppose that Assumption~\ref{assumption:t-moment}--\ref{assumption:q-m-matrix} hold. Then,
		\begin{equation}
			\label{eq:main-convergence}
			\Big(rZ_1^{(r)}, r^2Z_2^{(r)},\ldots, r^JZ_J^{(r)}, rZ_H^{(r)}\Big)\Rightarrow\Big(Z_1^\ast, Z_2^\ast,\ldots, Z_J^\ast, 0_H\Big),\quad\text{as } r\to0,
		\end{equation}
		where ``$\Rightarrow$'' stands for convergence in distribution. 
		Moreover, $Z_1^\ast, Z_2^\ast,\ldots, Z_J^\ast$ are independent, and for each class $k\in\call$, $Z_k^\ast$ is an exponential random variable with mean $d_k$ given by
		\begin{equation}
			\label{eq:sig-def}
			d_k=\frac{\sigma_k^2}{2(1-w_{kk})\mu_kb_k},\quad\text{and}\quad
			\sigma_k^2=2q^\ast(u^{(k)}),
		\end{equation}
		with quadratic function $q^\ast$ and vector $u^{(k)}$ carefully defined in Section~\ref{sec:postpone-param}.
	\end{theorem}
	\begin{remark}
		Note that we have abused the notation in \eqref{eq:main-convergence} and adopted the convention that $z=(z_1,\ldots, z_J, z_H)$. While $z_J$ denotes the element corresponding to class $J$, $z_H=(z_k, k\in\calh)\in\R^\calh$ is a vector of elements corresponding to all high-priority classes.
	\end{remark}

	In \eqref{eq:main-convergence}, for high-priority classes, $rZ_H^{(r)}\Rightarrow0$, which suggests that the load from high-priority jobs is absent from the heavy traffic limit. This SSC phenomenon naturally follows from the SBP service discipline under Assumption~\ref{assumption:stability} and~\ref{assumption:high-moment}. 
	For low priority classes, i.e., $j\in\call$, it is not surprising that appropriately scaled queue length processes $r^jZ_j^{(r)}\Rightarrow Z_j^\ast$, but what is striking is the independence result of $(Z_1^\ast, Z_2^\ast,\ldots, Z_J^\ast)$, which gives rise to a product-form stationary distribution. We emphasize that the pre-limit stationary distribution is not of product form, and neither is the conventional single-scale heavy traffic limit, which has been discussed thoroughly in \cite{BravDaiMiya2023}.

	An important intermediate step in proving the product-form limit is the establishment of the uniform moment bound for scaled low-priority queue lengths in Theorem~\ref{assumption:low-moment}.
	The uniform moment bound for scaled low-priority queue lengths can be proved under Assumption~\ref{assumption:t-moment} and Assumptions~\ref{assumption:high-moment} and \ref{assumption:q-m-matrix}. A uniform moment bound for appropriately scaled queue lengths has been similarly established for general Jackson networks in~\citet{GuanChenDai2023}.
	
	\begin{theorem}
		\label{assumption:low-moment}
		Under Assumptions~\ref{assumption:t-moment},~\ref{assumption:multi-scale-heavy-traffic}--\ref{assumption:q-m-matrix},
		for each $k\in\call$, there exists $r_0\in(0,1)$ and small $\epsilon_0>0$ such that
		\begin{equation}
			\label{eq:low-moment}\sup_{r\in(0,r_0)}\E\Big[\Big(r^kZ_k^{(r)}\Big)^{J+\epsilon_0}\Big]<\infty.
		\end{equation}
	\end{theorem}
	A detailed proof of Theorem~\ref{assumption:low-moment} is postponed to Appendix~\ref{sec:moment-bound-proof}.

	\subsection{Postponed Parameter Definition}
	\label{sec:postpone-param}
	Here, we provide rigorous definitions for several quantities that are mentioned earlier but not formally defined, namely the reflection matrix $R$ in Assumption~\ref{assumption:q-m-matrix} and the variance parameter $\sigma_k^2$ in Theorem~\ref{thm:main}.
	
	\paragraph*{Reflection Matrix.}
	To define the reflection matrix, we first introduce the following helper quantities. We start with the following matrix $A$ defined as
	\begin{equation*}
		A=(I-P^\top)\diag(\mu)(I-B),
	\end{equation*}
	where $B$ is a $K\times K$ matrix defined as
	\begin{equation*}
		B_{kl}=\begin{cases}
			1&\text{if }l=k+,\\
			0&\text{otherwise},
		\end{cases}
	\end{equation*}
	where $k+$ denotes the class which is at the same station as Class $k$ and with the priority immediately above Class $k$. In the two-station-five-class re-entrant line example discussed in the introduction  Section~\ref{sec:intro}, setting $k=3$ gives $k+=5$.

	Writing $A$ in block form, we have
	\begin{equation*}
		A=\begin{pmatrix}
			A_L & A_{LH}\\ A_{HL} & A_H
		\end{pmatrix},
	\end{equation*}
	where $A_L$ and $A_H$ are the principal submatrices corresponding the index set $\call$ and $\calh$ respectively, and $A_{LH}$ and $A_{HL}$ are the remaining corresponding blocks of $A$. Following~\cite{BravDaiMiya2023}, we adopt the same invertibility assumption for $A_H$ to ensure that the reflection matrix $R$ is well-defined.
	
	\begin{assumption}
		\label{assumption:ah-invertible} 
		The matrix $A_H$ is invertible.
	\end{assumption}
	
	Next, we construct the following matrix $Q\in\R^{J\times J}$, with entries defined as
	\begin{equation}
		\label{eq:q-def}
		Q_{ij}=P_{ij}-P_{i,H}A_H^{-\top}A_{j,H}^\top,\quad \forall\,i,j\in\call.
	\end{equation}
	Writing as a product of matrices, we have $Q=P_{L}-P_{LH}A_H^{-\top}A_{LH}^\top$. To clarify the notation, following the submatrix notation introduced above, $P_L$ denotes the principal submatrix corresponding to the index set $\call$, and $A_{LH}^\top=(A_{LH})^\top$.
	
	Given $Q$, we define the reflection matrix as
	\begin{equation}
		R=I-Q^\top,\label{eq:r-def}
	\end{equation}
	and impose the $\mathcal{P}$-matrix structural assumption stated in Assumption~\ref{assumption:q-m-matrix}.
	
	We will mainly use the following two properties of a $\mathcal{P}$-matrix.
	Firstly, all diagonal elements of a $\mathcal{P}$-matrix are positive and all leading principal matrices are invertible. Secondly, the inverse of a $\mathcal{P}$-matrix is also a $\mathcal{P}$-matrix, and hence all diagonal elements of the inverse of $R$ are positive. 
	
	\paragraph*{Variance Parameter.}
	To define the variance $\sigma_k^2=2q^\ast(u^{(k)})$ in Theorem~\ref{thm:main}, we need to define the quadratic function $q^\ast$ and the series of vector $u^{(k)}$ indexed by $k\in\call$.
	
	Starting with the function $q^\ast$, we have
	\begin{equation}
		\label{eq:q-star-def}
		\begin{aligned}
			q^\ast(\theta)&=\frac{1}{2}\sum_{l\in\cale}\alpha_lc_{e,l}^2\theta_l^2
			+\frac{1}{2}\sum_{l\in\calk}\lambda_l\Bigg(\sum_{l'\in\calk}P_{ll'}\theta_{l'}^2-\Big(\sum_{l'\in\calk}P_{ll'}\theta_{l'}\Big)^2+c_{s,l}^2\Big(-\theta_l+\sum_{l'\in\calk}P_{ll'}\theta_{l'}\Big)^2\Bigg),
		\end{aligned}
	\end{equation}
	where $c_{e,l}^2$ and $c_{s,l}^2$ are squared coefficients of variation defined in~\eqref{eq:ce-def} and~\eqref{eq:cs-def}.
	
	To define the vector $u^{(k)}$, we need to introduce the following $J\times J$ matrix $w=(w_{ij})$, which is well-defined under Assumptions~\ref{assumption:q-m-matrix} and~\ref{assumption:ah-invertible}.
	With $Q_{ij}$ defined in~\eqref{eq:q-def}, for each $k\in\call$, each column of $w$ matrix is defined as
	\begin{align}
		(w_{1,k},\ldots, w_{k-1,k})^\top&=\Big(I-Q_{1:k-1,1:k-1}\Big)^{-1}Q_{1:k-1,k},\label{eq:w-def-1}\\
		(w_{k,k},\ldots, w_{L,k})^\top&=Q_{k:L,k}
		+Q_{k:L,1:k-1}\Big(I-Q_{1:k-1,1:k-1}\Big)^{-1}Q_{1:k-1,k}.\label{eq:w-def-2}
	\end{align}
	
	Utilizing block matrix inversion properties, we recover the following relationship that 
	\begin{equation*}
		(1-w_{kk})^{-1} = (I-Q_{1:k, 1:k})^{-1}_{k,k}>0, \quad\forall k\in\call.
	\end{equation*}
	By the second property of $\mathcal{P}$-matrix mentioned above, it follows that
	$1-w_{kk}>0$. 
	
	We briefly explain the above notation convention $P_{S_1, S_2}$ with $S_1, S_2\subseteq \calk$. $P_{S_1, S_2}$ represents the submatrix of $P$ whose entries are $(P_{kl})$ with $k\in S_1$ and $l\in S_2$ following the conventions below, that
	\begin{align*}
		&[l:k] = \{l, l+1,\ldots, k\},\quad \text{when }l,k\in\call,\text{ and } l\leq k,\\
		&P_{S_1,l}=P_{S_1,\{l\}},\quad P_{k,S_2}=P_{\{k\}, S_2},
	\end{align*}
	and expression $P_{S_1, S_2}$ is interpreted to be 0 when either $S_1$ or $S_2$ is an empty set.
	
	Given these helper quantities, we define the series of vectors indexed by $k\in\call$ as follows
	\begin{equation}
		\label{eq:u-vec-def}
		u^{(k)} = \begin{pmatrix}
			u_L^{(k)} \\ u_H^{(k)}
		\end{pmatrix},\qquad\text{with}\quad u_L^{(k)}=\begin{pmatrix}
			w_{1:k-1,k}\\1\\0_{k+1:L}
		\end{pmatrix}\quad\text{and}\quad u_H^{(k)}=-A_H^{-\top}A_{LH}^\top u_L^{(k)}.
	\end{equation}
	
	As such, we have completed the definition of $\sigma_k^2=2q^\ast(u^{(k)})$ in Theorem~\ref{thm:main}.

	\section{Numerical Experiments}
	\label{sec:experiments}
	
	In this section, we present a series of numerical experiments to demonstrate our theory. We use the two-station, five-class re-entrant line in Figure~\ref{fig:2s5c-fig} discussed in the introduction as the case study. Following the discussion in Section~\ref{sec:intro}, here we deviate from the convention $\ell(j)=j$ for $j \in \calj$, and instead, we index the classes according to the traffic flow in Figure~\ref{fig:2s5c-fig}. We employ the static buffer priority policy $\{(5,3,1), (2,4)\}$, implying that the first station follows a last-buffer-first-serve policy, while the second station uses a first-buffer-first-serve policy.

	We examine how well our heavy-traffic formula approximates a heavily loaded pre-limit system.
	The external arrival rate is fixed at $\alpha_1=1$ and the limiting mean service times are 
	\begin{equation}
		\label{eq:limit-m}
		\Big(m_1,m_3,m_5\Big)=\rho_1\Big(\frac{1}{2},\frac{1}{4},\frac{1}{4}\Big)\quad\text{and}\quad \Big(m_2,m_4\Big)=\rho_2\Big(\frac{1}{3},\frac{2}{3}\Big).
	\end{equation}
	We model inter-arrival and service times using gamma distributions $\Gamma(a,b)$ where the parameter $a$ remains fixed regardless of $\rho_1$ and $\rho_2$, as specified in Table~\ref{tab:gamma-dis}. By properties of the gamma distribution, the squared coefficient of variation for these inter-arrival and service time distributions is computed as $c^2=1/a$.
	\begin{table}[htbp]
		\centering
		\begin{tabular}{ccccccc}
			\hline
			&Arrival& Class 1 & Class 2 & Class 3 & Class 4 & Class 5 \\
			\hline
			a&0.75   & 0.95 & 0.6& 0.95 & 0.6& 0.95 \\
			\hline
		\end{tabular}
		\caption{Parameter $a$ in Gamma Distribution}
		\label{tab:gamma-dis}
	\end{table}
	
	We conduct three sets of experiments where the traffic intensity at station 2 is fixed at $\rho_2=0.99$, while varying the traffic intensity at station 1 with $\rho_1 \in \{0.9, 0.96, 0.99\}$. For each choice of $\rho_1$, we approximate the mean queue length as
	\begin{equation*}
		\E\Big[Z_1\Big]\approx d_1/(1-\rho_1)\quad\text{and}\quad \E\Big[Z_4\Big]\approx d_4/(1-\rho_2),
	\end{equation*}
	where $d_1$ and $d_4$ are computed as follows using our closed-form formulas with limiting service rate in~\eqref{eq:limit-m} and the squared coefficient of variation deduced from Table~\ref{tab:gamma-dis} 
	\begin{align*}
		d_1&=\frac{\alpha_1}{2\big(m_1+m_3-m_5\frac{m_2}{m_4}\big)}\Big(\big(m_1+m_3-m_5\frac{m_2}{m_4}\big)^2c_{e,1}^2+m_1^2c_{s,1}^2+m_3^2c_{s,3}^2+m_5^2c_{s,5}^2\\
		&\qquad\qquad\qquad\qquad\qquad\quad+\big(\frac{m_5}{m_4}\big)^2\big(m_2^2c_{s,2}^2+m_4^2c_{s,4}^2\big)\Big),\\
		d_4&=\frac{\alpha_1}{2m_4}\Big((m_2+m_4)^2c_{e,1}^2+m_2^2c_{s,2}^2+m_4^2c_{s,4}^2\Big).
	\end{align*}

	We simulate the network for a total of $10^8$ arrivals, recording the mean queue length. Each setup is repeated 20 times to compute the confidence interval. We also examine the independence between the two low-priority classes. To examine the independence between two random variables, we use the Information Quality Ratio (IQR). IQR quantifies the information content of one variable relative $X$ to another $Y$ against total uncertainty and is calculated as
	\begin{equation*}
		\text{IQR}(X,Y)=\frac{I (X;Y)}{H(X,Y)}=\frac {\sum _{x\in X}\sum _{y\in Y}p(x,y)\log {p(x)p(y)}}{\sum _{x\in X}\sum _{y\in Y}p(x,y)\log {p(x,y)}}-1.
	\end{equation*}
	Despite a non-conventional metric for independence, IQR equals zero if and only if the two random variables $X$ and $Y$ are independent~\citep[Chapter 2]{CoveThoma2005}. The closer the IQR is to zero, the less dependent the two variables are. 
	
	\begin{table}[htbp]
		\centering
		\begin{tabular}{rrrrrr}
			\hline
			& \multicolumn{2}{c}{Class 1} & \multicolumn{2}{c}{Class 4}  &\\
			\cline{2-5} 
			$\rho_1$& \multicolumn{1}{c}{Sim} &   \multicolumn{1}{c}{M-Scale} & \multicolumn{1}{c}{Sim} & \multicolumn{1}{c}{M-Scale} &IQR$(10^{-4})$ \\
			\hline  
			$90\%$& $7.17\pm 0.02 $ & 7.53 & $163.23\pm 4.48 $ &167.75  & $3.88\pm 0.25$\\	
			\hline
			$96\%$& $19.71 \pm 0.11 $& 20.08 & $ 166.52\pm4.25 $ &167.75 & $9.04\pm0.63$ \\ 
			\hline
			$99\%$ & $82.80\pm 2.04$ & 82.83 & $159.44 \pm 3.12$ & 167.75 & $26.83\pm 2.94$\\
			\hline
		\end{tabular}
		\caption{Levels of load separation.}
		\label{tab:queue-table}
	\end{table}

	The approximated values are given in the respective M-Scale columns in Table~\ref{tab:queue-table}.
	Compared to the simulation results, our closed-form formula provides decent approximations. The near-zero IQR values also suggest approximate independence of low-priority queue lengths in heavy traffic. Interestingly, even when both stations are operating at a load of 0.99, i.e., $\rho_1=\rho_2=0.99$, with no load separation, our approximation still performs well. This result is intriguing and may require further investigation.
	
	Additionally, we explore how well the exponential distribution theory approximates the actual distribution by considering the case where $\rho_1=0.96$ and $\rho_2=0.99$. We plot histograms of $Z_1$ and $Z_4$ in Figure~\ref{fig:z1-plot} and~\ref{fig:z4-plot} respectively, with the $x$-axis representing queue length and the $y$-axis representing probability. The orange bars represent the geometric distribution derived from the exponential distribution in Theorem~\ref{thm:main}, while the red line shows the empirical distribution. We observe that the geometric approximation aligns well with the empirical data, with minor deviations at the beginning of the distribution but generally capturing the tail distribution accurately.
	\begin{figure}[htbp]
		\centering
		\begin{minipage}{0.3\textwidth}
			\centering
			\includegraphics[width=\textwidth]{Plots/Q1-histogram.png}
			\subcaption{Overall Histogram}\label{fig:plot1-z1}
		\end{minipage}%
		\hfill
		\begin{minipage}{0.3\textwidth}
			\centering
			\includegraphics[width=\textwidth]{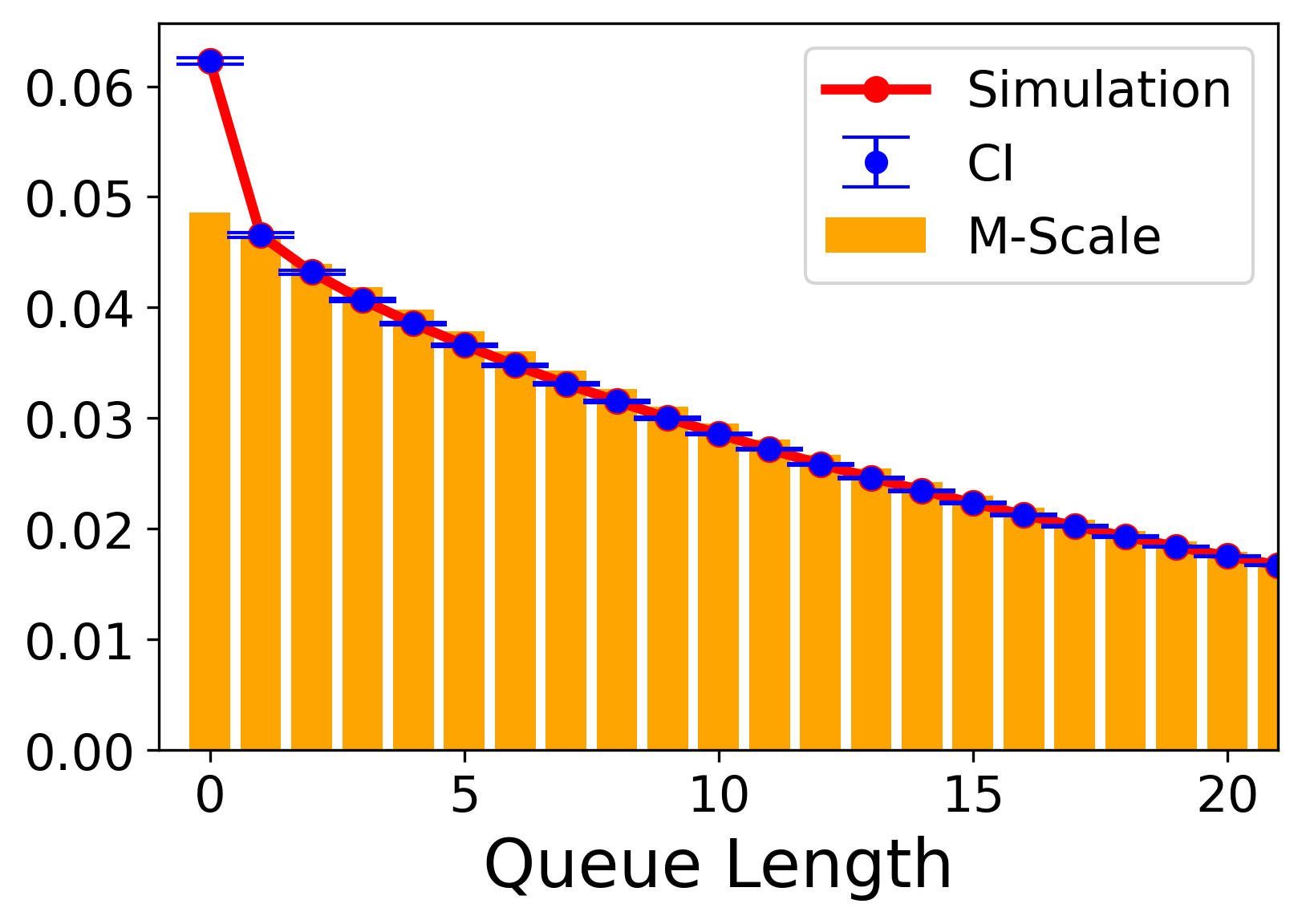}
			\subcaption{Head-Zoomed Histogram}\label{fig:plot2-z1}
		\end{minipage}%
		\hfill
		\begin{minipage}{0.3\textwidth}
			\centering
			\includegraphics[width=\textwidth]{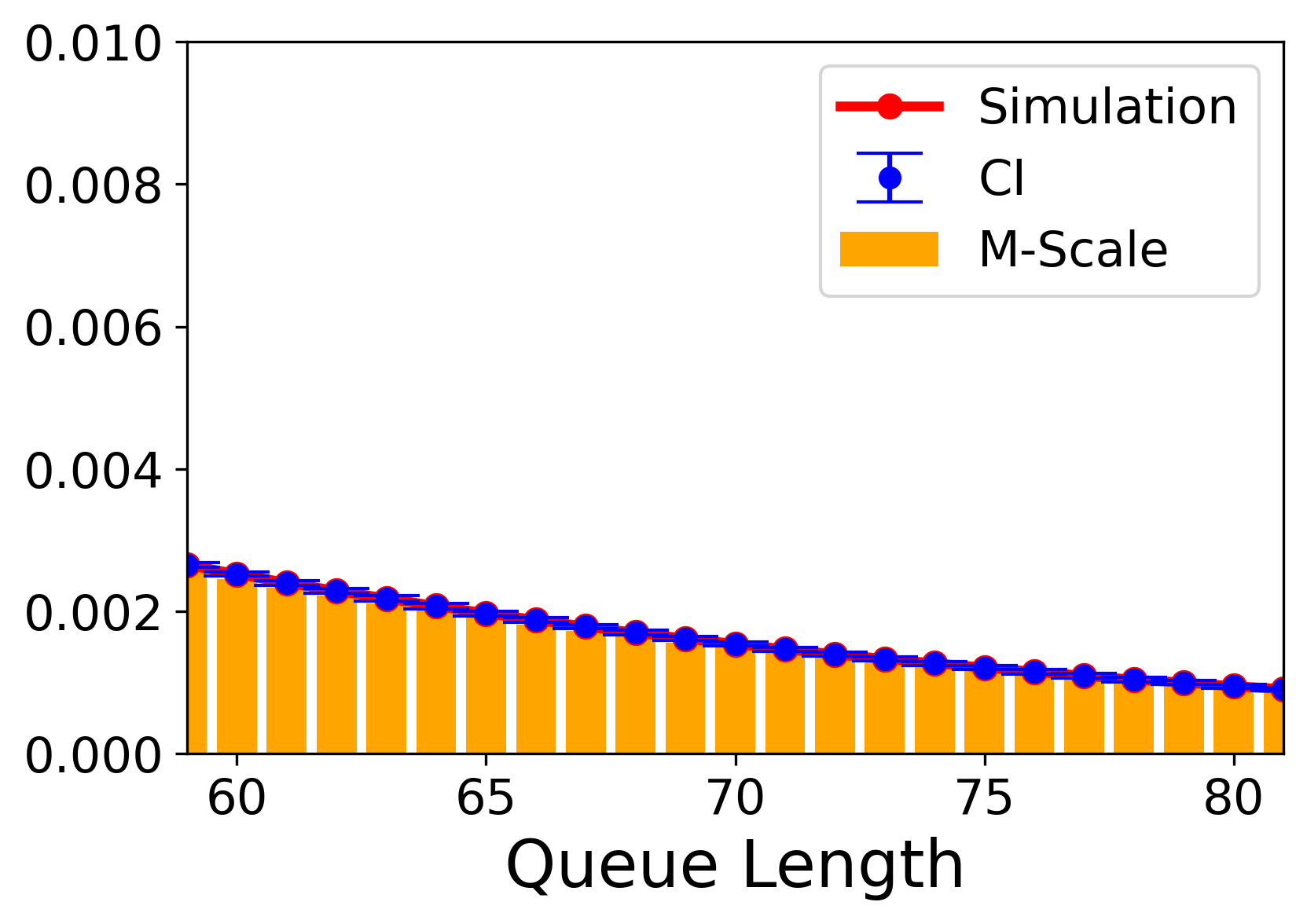}
			\subcaption{Tail-Zoomed Histogram}\label{fig:plot3-z1}
		\end{minipage}
		\caption{Histogram of $Z_1$}
		\label{fig:z1-plot}
	\end{figure}
	
	\begin{figure}[htbp]
		\centering
		\begin{minipage}{0.3\textwidth}
			\centering
			\includegraphics[width=\textwidth]{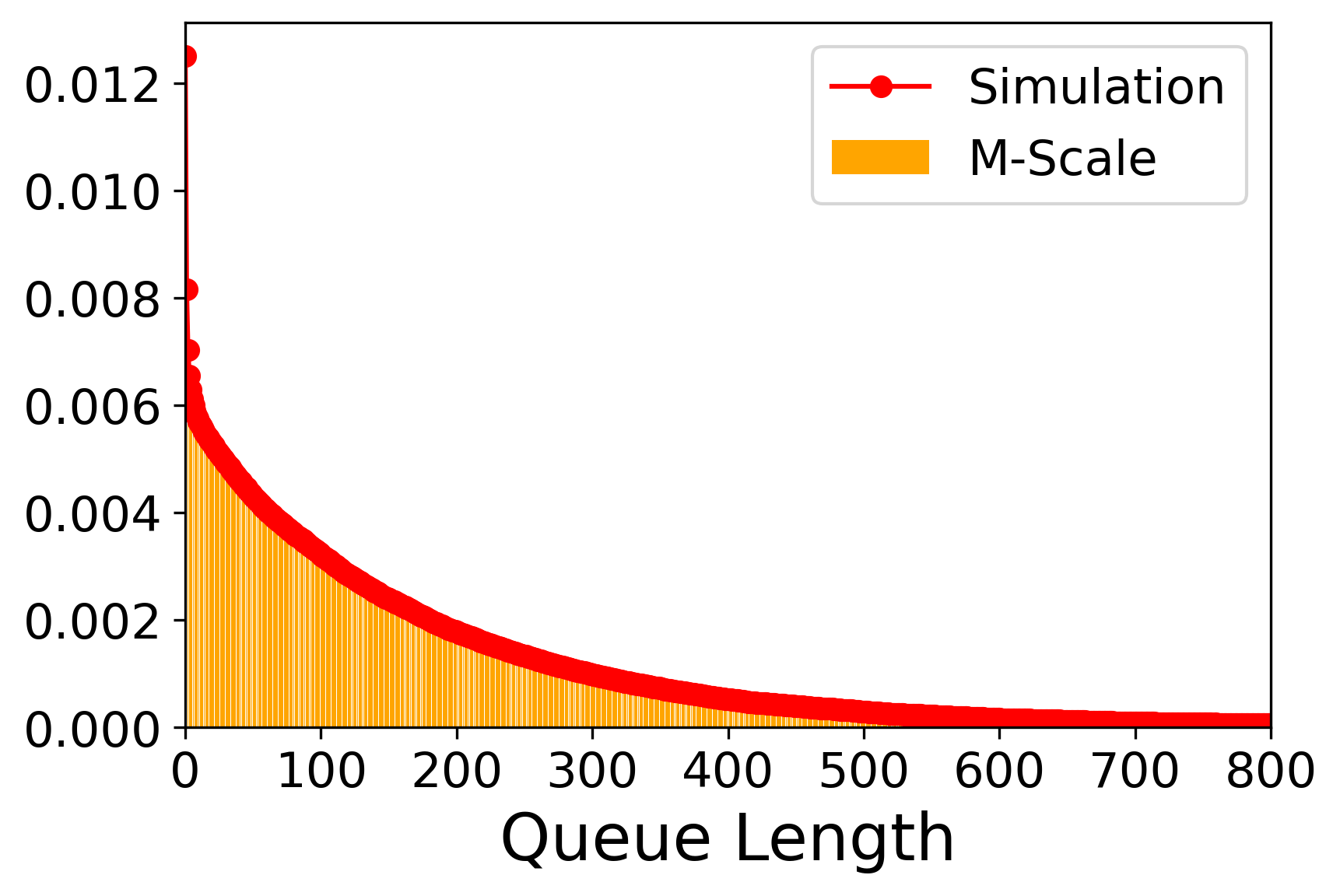}
			\subcaption{Overall Histogram}\label{fig:plot1-z4}
		\end{minipage}%
		\hfill
		\begin{minipage}{0.3\textwidth}
			\centering
			\includegraphics[width=\textwidth]{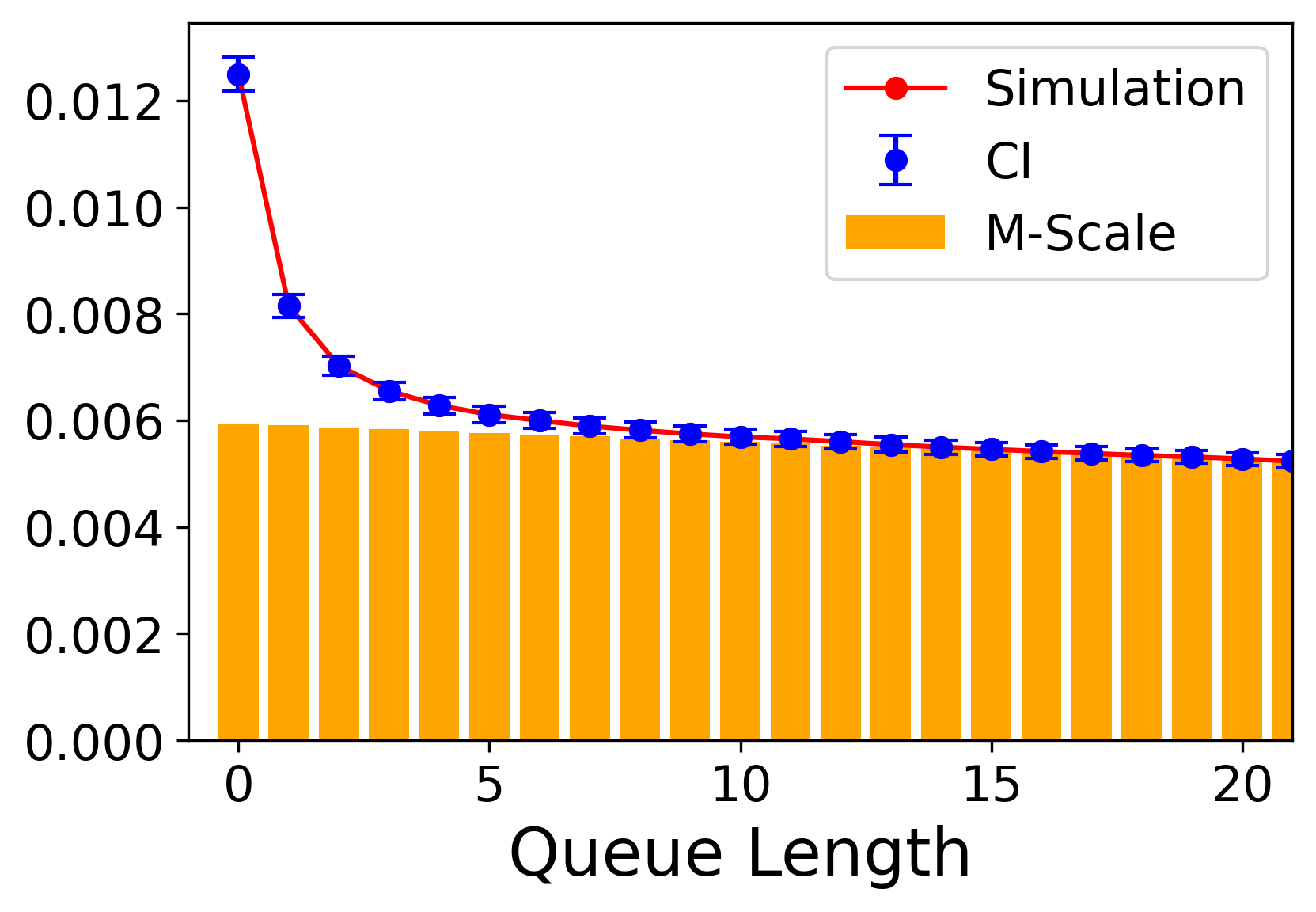}
			\subcaption{Head-Zoomed Histogram}\label{fig:plot2-z4}
		\end{minipage}%
		\hfill
		\begin{minipage}{0.3\textwidth}
			\centering
			\includegraphics[width=\textwidth]{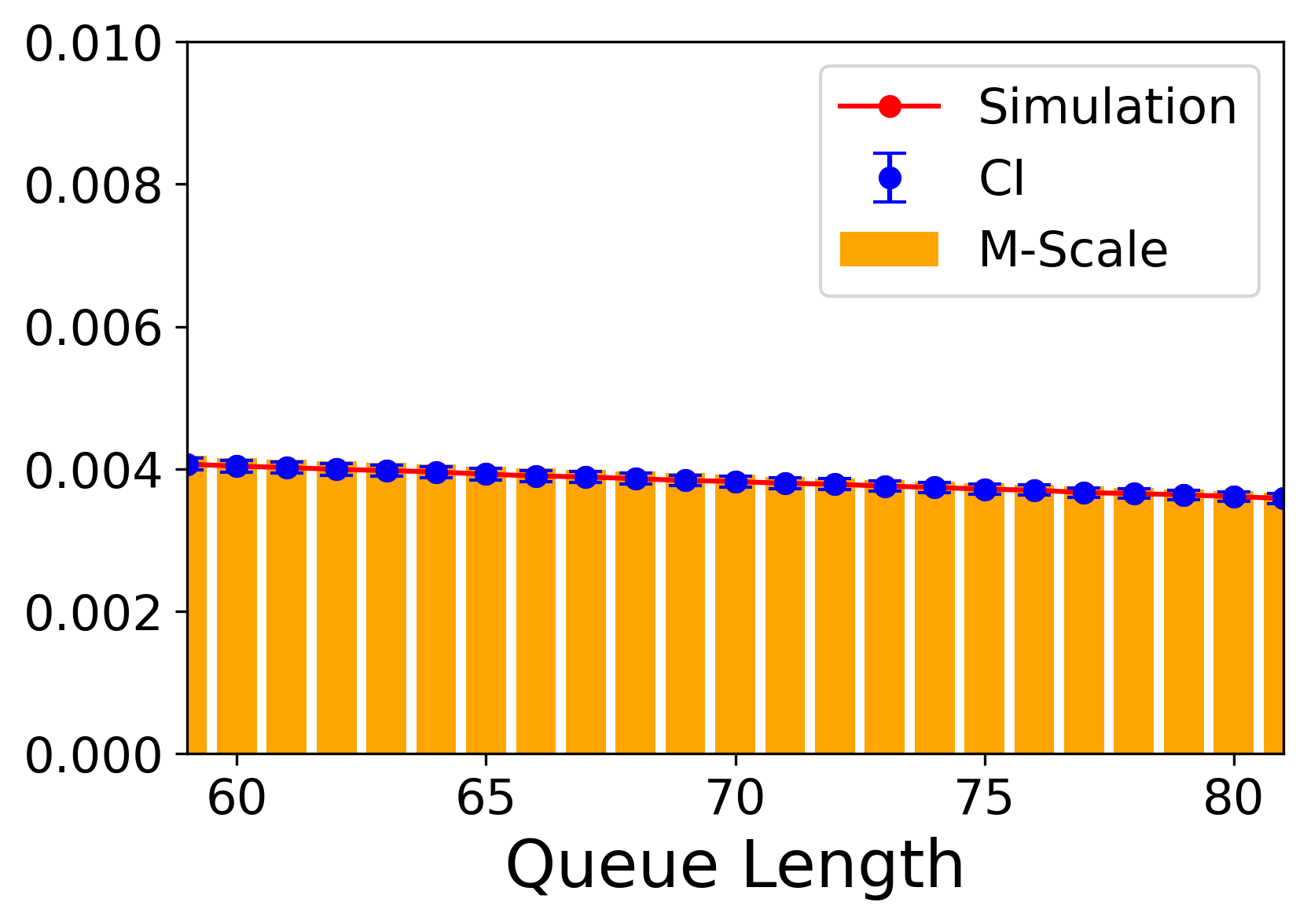}
			\subcaption{Tail-Zoomed Histogram}\label{fig:plot3-z4}
		\end{minipage}
		\caption{Histogram of $Z_4$}
		\label{fig:z4-plot}
	\end{figure}

	\subsection{Static Policy Optimization}

	Building on the closed-form formula, we extend its use beyond performance evaluation to evaluate all policies and perform static policy optimization, identifying the optimal static buffer priority policy. In this experiment, we focus on optimizing the average time-in-system, or cycle time, of jobs, which can be derived from queue lengths using Little’s law. By approximating the low-priority class queue lengths with our formula, and ignoring the high-priority queue lengths, we apply Little’s law to estimate the corresponding cycle times. The results of these approximations are shown in the ``M-Scale'' column of Table~\ref{tab:cycle-time}. The cycle times obtained from simulations, averaged over 2 billion observations, exhibit negligible confidence intervals.
	
	Based on the approximation, the SBP policies $\{(5,3,1),(4,2)\}$ corresponding to the last-buffer-first-serve policy, and $\{(3,5,1),(4,2)\}$ emerge as the best static buffer priority policies. This finding is subsequently confirmed by our simulation. 
	
	Additionally, we also observe that swapping the priority ranking of high-priority classes does not impact the limiting performance prediction. However, changing the priority ranking will affect performance in pre-limit systems. This discrepancy arises because high-priority queue lengths, which are assumed to be negligible in the limit, are nonzero in pre-limit systems.
	
	Nonetheless, our formula still offers a useful first step for performance evaluation and static policy optimization. Overall, these numerical experiments collectively demonstrate that our formula provides a solid foundation for performance evaluation. We leave the exploration of additional ways to leverage our theory for more refined performance evaluations to future research.
	
	\begin{table}[htbp]
		\centering
		\begin{tabular}{ccc}
			\hline  
			Priority    &  M-Scale &   Simulation\\\hline 
			$\{(5,3,1),(4,2)\}$    &    \blue{ 134.862 }&        $137.759\pm0.001$\\\hline  
			$\{(3,5,1),(4,2)\}$    &     134.862&        $137.117\pm0.001$ \\\hline  
			$\{(3,1,5),(4,2)\}$&     157.891&        $157.102\pm0.001$   \\\hline     
			$\{(1,3,5),(4,2)\}$   &     157.891  &        $167.010\pm0.001$ \\\hline  
			$\{(5,1,3),(4,2)\}$   &     180.920 &        $194.142\pm0.001$  \\\hline  
			$\{(1,5,3),(4,2)\}$   &     180.920 &        $195.931\pm0.001$\\\hline 
			$\{(5,3,1),(2,4)\}$ &     187.828 &        $185.250\pm0.001$ \\\hline  
			$\{(3,5,1),(2,4)\}$  &     187.828 &        $184.437\pm0.001$ \\\hline 
			$\{(3,1,5),(2,4)\}$  &     217.947 &        $216.077\pm0.001$\\\hline  
			$\{(1,3,5),(2,4)\}$  &     217.947 &        $217.676\pm0.001$\\\hline  
			$\{(5,1,3),(2,4)\}$  &     217.947 &        $226.220\pm0.001$\\\hline  
			$\{(1,5,3),(2,4)\}$  &     217.947 &        $226.364\pm0.001$ \\\hline  
			
		\end{tabular}
		\caption{Cycletime Approximation}
		\label{tab:cycle-time}
	\end{table}

	\section{Proof of Main Result}
	\label{sec:proof}
	
	In this section, we formally present the proof of our main result. 
	We commence by explaining the key components of our proof, which relies on the basic adjoint relationship (BAR) established for multiclass networks in \cite{BravDaiMiya2023}. To prevent redundancy, we skip the detailed derivation of BAR, and we refer interested readers to \cite{BravDaiMiya2023} for more details.

	\subsection{Moment Generating Functions for \texorpdfstring{$Z^{(r)}$}{Z(r)}}
	
	We begin by discussing moment generating functions (MGFs), which serve as a bridge between the BAR technique and the desired weak convergence result.
	
	For $\theta\in\R_-^K$, where $R_-^K=\{x\in\R^K:x_k\leq0, k=1,\ldots, K\}$, and we define the function $g_\theta(z)$ as
	\begin{equation}
		\label{eq:orig-g-test}
		g_\theta(z)=\exp\Big(\langle\theta, z\rangle\Big),\quad\forall z\in\R_+^K,
	\end{equation}
	where $\dotp{\theta,z}$ denotes the inner product and is computed as $\dotp{\theta,z}=\sum_{k=1}^K\theta_kz_k$, for $\theta,z\in\R^K$.
	
	We define the corresponding MGFs as
	\begin{equation}
		\label{eq:mgf-exp}
		\phi^{(r)}(\theta)=\E\Big[g_\theta(Z^{(r)})\Big],\quad\text{and}\quad 
		\phi_l^{(r)}(\theta)=\E\Big[g_\theta(Z^{(r)})\mid Z_{H(l)}^{(r)}=0\Big],\quad\forall l\in\calk,
	\end{equation}
	$H(k)$ denotes the set of classes at the same station as Class $k$ and with priorities at least as high as Class $k$. In the two-station-five-class re-entrant line example discussed in the introduction  Section~\ref{sec:intro}, setting $k=3$ gives $H(k)=\{3,5\}$.
	For fixed $\eta\in\R_-^K$, we define $\theta\in\R_-^K$ as
	$\theta\equiv\theta(\eta,r)=\big(r\eta_1,r^2\eta_2,\ldots,r^J\eta_J,r\eta_H).$
	When substituting this $\theta$ into \eqref{eq:mgf-exp}, we obtain
	\begin{equation*}
		\phi^{(r)}(\theta)=\E\Big[g_\theta(Z^{(r)})\Big]
		=\E\Big[\exp\Big(\sum_{l\in\call}\eta_lr^lZ_l^{(r)} + \sum_{l\in\calh}\eta_lrZ_l^{(r)}\Big)\Big],
	\end{equation*}
	which corresponds to the MGF of the appropriately scaled random vector 
	\begin{equation*}
		\big(rZ_1^{(r)},\ldots,r^JZ_J^{(r)},rZ_H^{(r)}\big).
	\end{equation*}
	\begin{remark}
		Similar to Theorem~\ref{thm:main}, we have abused the notation of $\eta_H$ to let it denote a vector of elements $\{\eta_k, k\in\calh\}$. For the remainder of this paper, we will continue to adopt this convention of notational simplification.
	\end{remark}

	Thus, the weak convergence in \eqref{eq:main-convergence} can be equivalently established by proving the following convergence in MGFs, for any $\eta=(\eta_1,\ldots, \eta_J, \eta_H^\top)^\top\in\R_-^K$,
	\begin{equation*}
		\lim_{r\to0}\phi^{(r)}(r\eta_1, r^2\eta_2,\ldots, r^J\eta_J, r\eta_H)=\prod_{k=1}^J\frac{1}{1-d_k\eta_k}.
	\end{equation*}
	
	Moreover, under Assumption~\ref{assumption:high-moment}, we can immediately obtain the following lemma, which gives a version of a transform state space collapse (SSC) result. The following lemma generalizes Lemma 7.8 from \cite{BravDaiMiya2023}. For completeness, we provide the full proof in Appendix~\ref{sec:ssc1-proof}.
	\begin{lemma}[Lemma 7.8 of \cite{BravDaiMiya2023}] 
		\label{lem:transform-ssc1}
		Under Assumption~\ref{assumption:high-moment}, for fixed $\eta= 
		(\eta_1,\ldots, \eta_J, \eta_H^\top)^\top\in\R_-^K$,
		\begin{align*}
			&\phi^{(r)}(r\eta_1,\ldots, r^J\eta_J, 0_H) - \phi^{(r)}(r\eta_1,\ldots, r^J\eta_J, r\eta_H) =o(1),\\
			&\phi_k^{(r)}(r\eta_1,\ldots, r^J\eta_J, 0_H) - \phi_k^{(r)}(r\eta_1,\ldots, r^J\eta_J, r\eta_H) =o(1),\quad\forall k\in\calk.
		\end{align*}
	\end{lemma}
	\begin{remark}
		By this SSC lemma, we note that $\eta_H$ is not in the heavy traffic limit. Hence, in the remainder of this paper, we will overload the notation $\eta$ and set $\eta=(\eta_1,\ldots, \eta_J)^\top\in\R_-^L$.
	\end{remark}
	
	Consequently, the lemma above enables us to focus on low-priority classes only. We prove our main results in Theorem~\ref{thm:main}, if we could show the following limit,
	\begin{equation*}
		\lim_{r\to0}\phi^{(r)}(r\eta_1, r^2\eta_2,\ldots, r^J\eta_J, 0_H)=\prod_{k=1}^J\frac{1}{1-d_k\eta_k}, \quad \eta = (\eta_1,\ldots, \eta_J)^\top\in\R_-^J.
	\end{equation*}
	
	\subsection{BAR}
	
	To prove the desired steady-state convergence, we adopt the basic adjoint relationship (BAR) technique, which is an alternative to the ``interchange of limits'' approach that has dominated the literature for the past two decades.
	The BAR approach directly analyzes the stationary distribution of the Markov process and uses the fundamental theorem of calculus to break down the steady-state balance equation into model primitives such as event time distributions.
	The BAR for MCNs with general inter-arrival and service time distributions is fully derived in \cite{BravDaiMiya2023}. To avoid repetition, we only provide a brief overview of BAR and omit the detailed derivation. We direct interested readers to~\cite{BravDaiMiya2023} for further details. 
	
	With inter-arrival and service times now generally distributed, we turn our attention to the Markov process $\{X(t),t\geq0\}$ as defined in~\eqref{eq:x-state}. Assuming that the process begins in its stationary distribution, the stationary distribution $\pi$ satisfies the following balance equation
	\begin{equation*}
		\E_\pi[f(X(1))-f(X(0))]=0
	\end{equation*}
	for function $f:\mathbb{S}\to\R$ satisfying the following conditions: for any $z\in\Z_+^J$, the function $f(z,\cdot,\cdot):\R_+^E\times\R_+^J\to\R$ is continuously differentiable at all but finitely many points and the derivatives for $f(z,\cdot,\cdot)$ in each dimension are uniformly bounded over $z$. By applying the fundamental theorem of calculus to expand $f(X(1))-f(X(0))$ and making use of the concept of the Palm probability measures, we obtain the BAR for the MCN~\citep[(6.16)]{BravDaiMiya2023}
	\begin{equation}
		\label{eq:bar-start}
		-\E_\pi[\cala f(X)] = \sum_{k\in\cale}\alpha_k\E_{e,k}[f(X_+^{(e,k)})-f(X_-^{(e,k)})]+\sum_{k\in\calk}\lambda_k\E_{s,k}[f(X_+^{(s,k)})-f(X_-^{(s,k)})].
	\end{equation}
	
	We break down the three terms in~\eqref{eq:bar-start}. In the first term, the operator $\cala$ is the ``interior operator'' defined as
	\begin{equation*}
		\cala f(x)=-\sum_{l\in\cale}\frac{\partial f}{\partial r_{e,l}}(x)-\sum_{l\in\calk}\frac{\partial f}{\partial r_{s,l}}(x)\mathbbm{1}\{z_l>0, z_{H_+(l)}=0\},\quad x=(z,r_e,r_s)\in\mathbb{S}.
	\end{equation*}
	In the absence of external arrival or service completion events, the first term in~\eqref{eq:bar-start} reflects how the residual times for external arrivals decrease, while the second term shows how the residual service times change at a rate of $-1$ when jobs are present at the class buffer and all buffers of higher priorities are empty, or remain unchanged otherwise. 
	Thus, the operator $\cala f$ in~\eqref{eq:bar-start} corresponds to the evolution of the MCN between jumps.
	
	The remaining terms describe changes in the state due to external arrivals and service completions, utilizing Palm probability measures. A Palm probability measure captures the conditional distribution of a point process, given that an event has occurred at a specified location; see \cite{BaccBrem2003} for further details. In the second term of~\eqref{eq:bar-start}, $(X_-^{(e,k)}, X_+^{(e,k)})\equiv((Z_-^{(e,k)}, R_{e,-}^{(e,k)}, R_{s,-}^{(e,k)}), (Z_+^{(e,k)}, R_{e,+}^{(e,k)}, R_{s,+}^{(e,k)}))$ represents the pre-jump and post-jump states conditioned on an external arrival at class $k$ occurring at time $0.$
	Given the occurrence of an external arrival at class $k$, $R_{e,-}^{(e,k)}$ decreases to $0$, $R_{e,+}^{(e;k)}$ is reset to a new independent inter-arrival time, and the queue length at class $k$ increases by one, i.e., $Z_{k,+}^{(e,k)}=Z_{k,-}^{(e,k)}+1$, with all other states unchanged.
	
	Similarly, the third term deals with service completions. $(X_-^{(s,k)}, X_+^{(s,k)})$ represents the pre-jump and post-jump states conditioned on service completion at class $k$ occurring at time $0.$ In this case, $R_{s,-}^{(s,k)}$ decreases to $0$, $R_{s,+}^{(s;k)}$ is reset to a new independent inter-arrival time, and the queue length at class $k$ decreases by one. The queue length might increase again at the same class or at one of the other classes, depending on routing. 
	
	The precise definitions of the Palm measures are provided in (6.8) and (6.9) of~\cite{BravDaiMiya2023}. We include the lemma~\citep[Lemma 6.3]{BravDaiMiya2023} that describes the properties of the pre-jump and post-jump states for completeness.
	
	\begin{lemma}[\cite{BravDaiMiya2023}]
		The pre-jump state and the post-jump state have the following representation,
		\begin{align*}
			&X_+^{(e,k)}=X_-^{(e,k)}+\Big(e^{(k)},e^{(k)}T_{e,k}/\alpha_k,0\Big)\quad\text{under }\Prob_{e,k}, k\in\cale\\
			\text{and}\quad&X_+^{(s,k)}=X_-^{(s,k)}+\Big(-e^{(k)}+\xi^{(k)},0,e^{(k)}T_{s,k}/\mu_k\Big)\quad\text{under }\Prob_{s,k}, k\in\calk,
		\end{align*} 
		where $\Prob_{e,k}$ and $\Prob_{s,k}$ denote the Palm measures corresponding to external arrivals and service completions respectively.
		
	\end{lemma}
	To simplify the notation, we omit the superscripts ``$^{(e,k)}$'' or ``$^{(s,k)}$'' from $X^{(e,k)}_-$ and $X^{(s,k)}_-$ for $k\in\calk$. Moreover, we use $\Delta$ to represent the increments corresponding to different jump events, i.e., external job arrival or service completion,
	\begin{align*}
		&\Delta_{e,k}\equiv\Big(e^{(k)},e^{(k)}T_{e,k}/\alpha_k,0\Big),\quad k\in\cale,\\
		\text{and}\quad&\Delta_{s,k}\equiv\Big(-e^{(k)}+\xi^{(k)},0,e^{(k)}T_{s,k}/\mu_k\Big),\quad k\in\calk.
	\end{align*} 
	Together, these random vectors satisfy the following BAR 
	\begin{equation}
		\label{eq:bar-general}
		-\E_\pi[\cala f(X)] = \sum_{k\in\cale}\alpha_k\E_{e,k}[f(X_-+\Delta_{e,k})-f(X_-)]+\sum_{k\in\calk}\lambda_k\E_{s,k}[f(X_-+\Delta_{s,k})-f(X_-)],
	\end{equation}
	where $\E_{e,k}[\cdot]$ and $\E_{s,k}[\cdot]$ are the respective Palm expectations.

	BAR can be used to analyze all steady-state properties of the system by selecting appropriate test functions to evaluate in~\eqref{eq:bar-general};  it is not limited to proving the desired product-form limit. For instance, consider the test function
	\begin{equation*}
		f(x)=r_s,\quad x=(z,r_e,r_s)\in\mathbb{S},
	\end{equation*}
	which allows us to obtain the following characterization of idle probabilities. 
	
	\begin{lemma}
		\label{lem:idle-prob}
		Under Assumption~\ref{assumption:stability}, we have
		\begin{equation*}
			\Prob(Z_{H(k)}^{(r)}=0)=\beta_k^{(r)},\quad k\in\calk,\quad r\in(0,1),
		\end{equation*}
		where $\beta_k^{(r)}$ is defined as
		\begin{equation*}
			\beta_k^{(r)} = 1-\sum_{l\in H(k)}\lambda_lm_l^{(r)}, \quad\forall k\in\calk.
		\end{equation*}
	\end{lemma}
	Lemma~\ref{lem:idle-prob} has been carefully proved in~\citet{BravDaiMiya2023}, and hence the proof is omitted.
	By Lemma~\ref{lem:idle-prob} and Assumption~\ref{assumption:multi-scale-heavy-traffic}, we notice that
	as $r\to0$,
	\begin{align*}
		&\Prob(Z_{H(k)}^{(r)}=0)=1-\rho_k^{(r)} = r^kb_k\to0,\quad\text{for }k\in\call,\\
		\text{and}\quad
		&\Prob(Z_{H(k)}^{(r)}=0)\to\beta_k>0,\quad\text{for }k\in\calh.
	\end{align*}

	\subsection{Exponential Test Functions with Truncation}
	
	When using exponential test functions in~\eqref{eq:bar-general}, we can recover moment generating functions, whose convergence implies weak convergence. 
	To prove the weak convergence to the product-form limit, it is important to design the $\theta$ in the exponential test functions~\eqref{eq:orig-g-test} as a function of the fixed $\eta=(\eta_1,\ldots,\eta_J)^\top\in\R_-^J$ and $r\in(0,1)$.
	Fix $\eta\in\R_-^J$ and an index $k\in\call$. 
	Let $\theta$ be a function of $\eta$ and $r$ based on the following rules, where $c_{ll'}\in\R$ are constant coefficients to be specified later for $l,l'\in\calk$,
	\begin{align}
		&\theta_l=r^l\eta_l,
		&&\text{for } l\in\call \text{ and } l>k\label{eq:theta-rule-1}\\
		&\theta_k=r^k\eta_k+\sum_{l'=k+1}^L c_{kl'}(r^{l'}\eta_{l'})\\
		&\theta_l=\sum_{l'=k}^L c_{ll'}(r^{l'}\eta_{l'}),
		&&\text{for remaining } l\in\calk. \label{eq:theta-rule-3} 
	\end{align}
	We note that $\theta$ as a function of $\eta$ and $r$ is of order $|\theta|=O(r^k)$.

	An inherent challenge in this construction arises from the fact that even though $\eta\in\R_-^L$, there is no assurance that $\theta_l\leq0$ for $l\in\call$ and $l\leq k$ or for $l\in\calh$. Hence, the simple exponential test function $g_\theta$ as defined in \eqref{eq:orig-g-test} may not always be bounded in $z\in\R_+^K$ as desired. 
	To overcome this potential unboundedness, we employ the truncation technique devised in~\cite{BravDaiMiya2017, BravDaiMiya2023}, and we consider the following test function shown in~\eqref{eq:g-test}. For fixed $\eta\in\R_-^L$ and $k\in\call$, and $\theta\equiv\theta(k,r,\eta)$ following~\eqref{eq:theta-rule-1}--\eqref{eq:theta-rule-3}, we define
	\begin{equation}
		\label{eq:g-test}
		\begin{aligned}
			g_{\theta,k,r}(z)=\exp\Big(&\left\langle\theta_{1:k-1},z_{1:k-1}\wedge r^{-k}\right\rangle\\
			&+ \Big(\sum_{l=k+1}^Lc_{kl}r^l\eta_{l}\Big)\Big( z_k\wedge r^{-(k+1)}\Big)  +\sum_{l=k}^Lr^l\eta_lz_l +\left\langle\theta_H, z_H\wedge r^{-1}\right\rangle\Big),
		\end{aligned}
	\end{equation}
	where $u\wedge v$ denotes element-wise minimum of vectors $u$ and $v$.
	
	To work with general inter-arrival and service time distributions, we adopt the technique developed in~\citet{BravDaiMiya2017, BravDaiMiya2023} and design the following test function built upon~\eqref{eq:g-test} to incorporate the additional time-related random variables.
	As inter-arrival and service times may not be bounded, we utilize the truncation technique 
	on the remaining inter-arrival times $u$ and remaining service times $v$. Hence, for fixed $\eta\in\R_-^J$ and $k\in\call$ and $\theta$ constructed as~\eqref{eq:theta-rule-1}--\eqref{eq:theta-rule-3}, the test function is defined as
	\begin{equation*}
		f_{\theta,k,r, t}(z, u, v)=g_{\theta,k,r}(z)\cdot\exp\Big(-\Big(\sum_{l\in\cale}\gamma_l(\theta,t)\Big(\alpha_l u_l\wedge t^{-1}\Big) + \sum_{l'\in\calk}\zeta_{l'}(\theta,t)\Big(\mu_{l'}v_{l'}\wedge t^{-1}\Big)\Big)\Big),
	\end{equation*}
	for $g_{\theta,k,r}(z)$ defined in~\eqref{eq:g-test}, $\varepsilon_0\in(0,\delta_0/(J+\delta_0)]$ and functions $\gamma(\theta,t)$ and $\zeta(\theta,t)$ that satisfy
	\begin{align}
		&e^{\theta_l}\E\Big(e^{-\gamma_l(\theta_l,t)(T_{e,l}\wedge t^{-1})}\Big)=1,\quad l\in\cale\label{eq:gamma-def}\\
		&\Big(\sum_{l'\in\bar\calk}P_{ll'}e^{-\theta_l+\theta_{l'}} \Big)\E\Big(e^{-\zeta_l(\theta,t)(T_{s,l}\wedge t^{-1})}\Big)=1,\quad l\in\calk\label{eq:xi-def}.
	\end{align}
	
	Properties of $\gamma_l(\theta_l,t)$ and $\zeta_l(\theta, t)$ functions are thoroughly discussed in \cite{BravDaiMiya2017, BravDaiMiya2023} and \cite{DaiGlynXu2023}. 
	Functions $\gamma_l(\theta_l,t)$ and $\zeta_l(\theta, t)$ are uniquely determined by equations \eqref{eq:gamma-def} and \eqref{eq:xi-def} respectively, which has been shown in \cite{BravDaiMiya2017, BravDaiMiya2023}. 
	For generic $\theta\equiv\theta(r)\in\R^K$ such that $|\theta(r)|\leq cr$, with $\theta$ defined in \eqref{eq:theta-rule-1}--\eqref{eq:theta-rule-3} as an example, $\gamma_l(\theta_l,t)$ and $\zeta_l(\theta,t)$ admits Taylor expansion when $r\to0$ and $t\to0$. These expansions are demonstrated in Lemma 7.5 of \cite{BravDaiMiya2023} and Lemma 5.3 of \cite{DaiGlynXu2023}. We include Lemma 5.3 of \cite{DaiGlynXu2023} below for completeness, but omit the proofs.
	
	To facilitate stating the Taylor expansion result, we define the following quantities. 
	\begin{align*}
		\Bar{\gamma}_l(\theta_l)&=\theta_l,\quad \Tilde{\gamma}_l(\theta_l)=\frac{1}{2}c_{e,l}^2\theta_l^2,\quad 
		\gamma^\ast_l(\theta_l)=\Bar{\gamma}_l(\theta_l)+\Tilde{\gamma}_l(\theta_l),\quad l\in\cale\\
		\Bar{\zeta}_l(\theta)&=-\theta_l+\sum_{l'\in\calk}P_{ll'}\theta_{l'},\quad l\in\calk\\
		\Tilde{\zeta}_l(\theta)&=\frac{1}{2}\left(\sum_{l'\in\calk}P_{ll'}\theta_{l'}^2-\left(\sum_{l'\in\calk}P_{ll'}\theta_{l'}\right)^2+c_{s,l}^2\left(-\theta_l+\sum_{l'\in\calk}P_{ll'}\theta_{l'}\right)^2\right),\quad l\in\calk\\
		\zeta^\ast_l(\theta)&=\Bar{\zeta}_l(\theta)+\tilde\zeta_l(\theta),\quad l\in\calk.
	\end{align*}
	
	\begin{lemma}[Lemma 5.3 of \cite{DaiGlynXu2023}]
		\label{lem:xi-gamma-taylor}
		Let $\theta(r) \in \R^K$ be given for each $r \in (0, 1)$, satisfying $|\theta(r)|\leq cr$. Denoting $\theta = \theta(r)$ and setting $t=r^{1-\varepsilon_0}$, we have
		\begin{align*}
			\gamma_l(\theta_l, r^{1-\varepsilon_0})&=\gamma_l^\ast(\theta_l) + o(r^J\theta_l) + o(|\theta_l|^2)\quad l\in\cale,\\
			\zeta_l(\theta, r^{1-\varepsilon_0})&=\zeta_l^\ast(\theta)+ o(r^J|\theta|) + o(|\theta|^2),\quad l\in\calk.
		\end{align*}
	\end{lemma}
	
	\begin{remark}
		The above lemma works for general $\theta$ as a function of $r$. As an example, when substituting our specially constructed $\theta$ following \eqref{eq:theta-rule-1}--\eqref{eq:theta-rule-3} into the above lemma, we have
		\begin{equation*}
			\gamma_l({\theta}_l, r^{1-\varepsilon_0})=\gamma_l^\ast({\theta}_l) + o(|\theta_l|^2)\quad l\in\cale,\quad\text{and}\quad
			\zeta_l({\theta}, r^{1-\varepsilon_0})=\zeta_l^\ast({\theta}) + o(|\theta|^2)\quad l\in\calk,
		\end{equation*}
		for $|\theta|=O(r^k)$, $k\leq J$, and hence $o(|\theta|^2)$ is the dominating term. 
	\end{remark}
	
	We next define the MGFs
	\begin{equation*}
		\psi^{(r)}(\theta)=\E\Big[f_{\theta,k,r,r^{1-\varepsilon_0}}(X^{(r)})\Big],\quad\text{and}\quad 
		\psi_l^{(r)}(\theta)=\E\Big[f_{\theta,k,r,r^{1-\varepsilon_0}}(X^{(r)})\mid Z_{H(l)}^{(r)}=0\Big],\quad\forall l\in\calk.
	\end{equation*}
	
	Together with the Taylor expansion results in Lemma~\ref{lem:xi-gamma-taylor}, we have the following asymptotic BAR, which is a generalization of (7.7) from \cite{BravDaiMiya2023}.
	\begin{proposition}
		\label{prop:asymptotic-bar}
		Assume the assumptions in Theorem~\ref{thm:main}. Fix $\eta\in\R_-^L$ and index $k\in\call$. For $\theta\equiv\theta(k,r,\eta)$ defined in \eqref{eq:theta-rule-1}--\eqref{eq:theta-rule-3}, as $r\to0$, we have
		\begin{equation}
			\label{eq:bar-taylor}
			\begin{aligned}
				q^\ast(\theta)\psi^{(r)}(\theta)&-\sum_{l\in\call}\beta_l^{(r)}\mu_l^{(r)}\zeta_l^\ast(\theta)\Big(\psi_l^{(r)}(\theta)-\psi^{(r)}(\theta)\Big)\\
				&+\sum_{l\in\calh}\beta_l^{(r)}\Big(\mu_{l-}^{(r)}\zeta_{l-}^\ast(\theta) - \mu_l^{(r)}\zeta_l^\ast(\theta)\Big)\Big(\psi_l^{(r)}(\theta)-\psi^{(r)}(\theta)\Big) =  o(r^J|\theta|)+ o(|\theta|^2),
			\end{aligned}
		\end{equation}
		where $q^\ast(\theta)=\sum_{l\in\cale}\alpha_l\gamma_l^\ast(\theta_l)+\sum_{l\in\calk}\lambda_l\zeta_l^\ast(\theta)$ as previously defined in~\eqref{eq:q-star-def}.
	\end{proposition}

	Moreover, by setting up $\theta\equiv\theta(k,r,\eta)$ following \eqref{eq:theta-rule-1}--\eqref{eq:theta-rule-3}, we have the following state space collapse (SSC) results that relate $\psi^{(r)}$ to $\phi^{(r)}$. Please refer to Appendix~\ref{sec:ssc-proof} for the detailed proof.
	
	\begin{lemma}
		\label{lem:phi-ssc}
		Fix $\eta\in\R_-^J$ and $k\in\call$. For $\theta\equiv\theta(k,r,\eta)$ following \eqref{eq:theta-rule-1}--\eqref{eq:theta-rule-3}, we have
		\begin{align}
			&\phi^{(r)}(0_{1:k-1},r^k\eta_k, r^{k+1}\eta_{k+1},\ldots, r^J\eta_J,0_H)-\psi^{(r)}(\theta)=o(1)\label{eq:ssc-1}\\
			&\phi_k^{(r)}(0_{1:k}, r^{k+1}\eta_{k+1},\ldots, r^J\eta_J,0_H)-\psi_k^{(r)}(\theta)=o(1)\label{eq:ssc-2}.
		\end{align}
		For $l\in\calh$, we have
		\begin{equation}
			\phi_{l}^{(r)}(0_{1:k-1},r^k\eta_k, r^{k+1}\eta_{k+1},\ldots, r^J\eta_J,0_H)-\psi_{l}^{(r)}(\theta)=o(1)\label{eq:ssc-high-1}.
		\end{equation}
	\end{lemma}

	\subsection{Proof of Theorem~\ref{thm:main}}
	\label{sec:outline}
	
	Finally, we are ready to prove our main result Theorem~\ref{thm:main}.
	Proving this theorem relies on the analysis of BAR in~\cite{BravDaiMiya2023} and a smart application of the multi-scale heavy traffic assumption.
	
	We first present the following lemma, a simple extension of (7.30) of Lemma 7.9 in \cite{BravDaiMiya2023} to this multi-scale heavy traffic setup. 
	
	\begin{lemma}
		\label{lem:mgf-high}
		Fix $\eta\in\R_-^L$ and $k\in\call$, for every $l\in\calh$, we have
		\begin{equation*}
			\phi_{l}^{(r)}(0_{1:k-1},r^k\eta_k, r^{k+1}\eta_{k+1},\ldots, r^J\eta_J,0_H)-\phi^{(r)}(0_{1:k-1},r^k\eta_k, r^{k+1}\eta_{k+1},\ldots, r^J\eta_J,0_H)=o(1).
		\end{equation*}
		
	\end{lemma}
	
	This lemma intuitively makes sense. In heavy traffic, high-priority queues tend to have relatively small queue lengths. Conditioning on $Z_{H(l)}=0$ for $l\in\calh$, i.e., conditioning on no queue for high-priority classes will have a negligible impact on the limit. Subsequently, the conditional MGF is approximately the same as the unconditional MGF.
	The detailed proof of this lemma is included in Appendix~\ref{sec:mgf-high-proof}.

	Making use of Lemma~\ref{lem:idle-prob} and Lemma~\ref{lem:mgf-high}, we further simplify the asymptotic BAR in \eqref{eq:bar-taylor} to the following, for fixed $k\in\call$, and $\theta$ following~\eqref{eq:theta-rule-1}--\eqref{eq:theta-rule-3},
	\begin{equation}
		\label{eq:theta-bar-sim}
		\begin{aligned}
			\Tilde{q}(\theta)\psi^{(r)}(\theta)&-\sum_{l\in\call}\beta_l^{(r)}\mu_l^{(r)}\Bar{\zeta}_l(\theta)\Big(\psi_l^{(r)}(\theta)-\psi^{(r)}(\theta)\Big)\\
			&+\sum_{l\in\calh}\beta_l^{(r)}\Big(\mu_{l-}^{(r)}\bar{\zeta}_{l-}(\theta) - \mu_l^{(r)}\Bar{\zeta}_l(\theta)\Big)\Big(\psi_l^{(r)}(\theta)-\psi^{(r)}(\theta)\Big) 
			=o(r^J|\theta|)+ o(|\theta|^2).
		\end{aligned}
	\end{equation}
	
	Note that we have effectively simplified the coefficients of $\psi_l^{(r)}(\theta)-\psi^{(r)}(\theta)$ from the sum of the first two orders of terms in Taylor expansion in~\eqref{eq:bar-taylor} to only the first-order terms in~\eqref{eq:theta-bar-sim}. 
	
	Next, we make use of the previously designed structure of $\theta$ and carefully set the coefficients $c_{ll'}$ in~\eqref{eq:theta-rule-1}--\eqref{eq:theta-rule-3}, so as to make the best use of the multi-scale heavy traffic assumption.
	
	Fix $\eta=(\eta_1,\ldots,\eta_L)^\top\in\R_-^L$. For \textbf{each} $k\in\call$, we choose $\theta\equiv\theta(\eta,k,r)$ such that
	\begin{align}
		&\theta_{l}=w_{lk}\theta_k+\sum_{l'=k+1}^Lw_{ll'}\theta_{l'}, && \forall\,l<k, l\in\call\label{eq:theta-def-first}\\
		&\theta_k=r^k\eta_k+\sum_{l'=k+1}^Lw_{kl'}\theta_{l'}\\
		&\theta_{l}=r^l\eta_l,&&\forall\, l>k, l\in\call\\
		&\theta_H=-A_H^{-\top}A_{LH}^\top\theta_L,\label{eq:theta-def-last}
	\end{align}
	where $w_{ij}$ has been defined in \eqref{eq:w-def-1} and \eqref{eq:w-def-2}.

	Clearly, the special construction of $\theta$ satisfies the general rule in \eqref{eq:theta-rule-1}--\eqref{eq:theta-rule-3}.
	Given this choice of coefficients, we do not have much guarantee for the value of $w_{ll'}$ under Assumption~\ref{assumption:q-m-matrix}, except that $w_{kk}<1$. Therefore, given the fixed index $k$, for $l\in\call$ and $l<k$, 
	when $w_{lk}<0$, we may have $\theta_l>0$. 
	This explains the rationale behind the construction of test function $g$ in~\eqref{eq:g-test} with truncation on the low-priority classes.
	
	Nonetheless, all of the SSC results presented for general $\theta$ in Lemma~\ref{lem:phi-ssc} are applicable to the specially designed $\theta$ in \eqref{eq:theta-def-first}--\eqref{eq:theta-def-last}.
	Moreover, there are the additional properties brought along by the specially chosen coefficients $c_{ll'}$ of $\theta$ and we list them below.
	
	\begin{lemma}
		\label{lem:w-solution}
		Fix $k\in\call$. Under Assumption~\ref{assumption:q-m-matrix}, the following equations
		\begin{equation*}
			w_{ij}=Q_{ij} + \sum_{l<k}Q_{il}w_{lj},\quad i,j\in\call
		\end{equation*}
		has a unique solution $(w_{1k},\ldots, w_{Lk})$ that is given by \eqref{eq:w-def-1} and \eqref{eq:w-def-2}.
	\end{lemma}
	
	\begin{lemma}[Solution to a linear system]
		\label{lem:theta-solution}
		Fix $\eta\in\R_-^L$ and $k\in\call$. $\theta$ defined in \eqref{eq:theta-def-first}--\eqref{eq:theta-def-last} corresponds to the (unique) solution to the following system of equations,
		\begin{align*}
			&\Bar{\zeta}^{(r)}_l(\theta)=0,
			&&\forall\, l<k, l\in\call\\
			&\Bar{\zeta}^{(r)}_k(\theta)=-(1-w_{kk})\cdot r^k\eta_k,\\
			&\mu_{l-}^{(r)}\Bar{\zeta}_{l-}(\theta)-\mu_l^{(r)}\Bar{\zeta}_l(\theta)=0,
			&&\forall\, l\in\calh.
		\end{align*}
	\end{lemma}
	The proofs of Lemma~\ref{lem:w-solution} and \ref{lem:theta-solution} follow directly from the definitions of $w$ and $\theta$, and are therefore omitted.
	
	Hence, starting from the asymptotic BAR~\eqref{eq:theta-bar-sim}, we make use of the above-mentioned key properties of $\theta$ and obtain the relationship among moment generating functions in the following lemma, which plays a crucial role in constructing the inductive argument presented later in Proposition~\ref{prop:independent-prop}. The complete proof can be found in Appendix~\ref{sec:indep-lem-proof}.
	
	\begin{lemma}
		\label{lem:independence-lem}
		Fix $\eta\in\R_-^L$ and $k\in\call$. Let $\theta$ and be defined in \eqref{eq:theta-def-first}--\eqref{eq:theta-def-last}, and we have
		\begin{align}
			&\psi^{(r)}(\theta)-\frac{1}{1-d_k\eta_k}\psi_k^{(r)}(\theta)=o(1).\label{eq:mgf-eq1}
		\end{align}
	\end{lemma}

	Therefore, the following proposition, which is the last step in completing the proof for Theorem~\ref{thm:main}, naturally follows from the lemma above and the SSC results in Lemma~\ref{lem:phi-ssc}. 
	\begin{proposition}
		\label{prop:independent-prop}
		Fix $\eta\in\R_-^L$. For each $k\in\call$, as $r\downarrow0$, we have
		\begin{equation}
			\begin{aligned}
				&\phi^{(r)}(0,\ldots, 0, r^k\eta_k, \ldots, r^J\eta_J, 0_H)\\
				&\qquad-\frac{1}{1-d_k\eta_k}\phi^{(r)}(0,\ldots, 0, r^{k+1}\eta_{k+1},\ldots, r^J\eta_J, 0_H)=o(1).
			\end{aligned}
			\label{eq:independent-prop}
		\end{equation}
	\end{proposition}
	\begin{proof}
		$\eta\in\R_-^J$ is fixed throughout this proof. Following Lemma~\ref{lem:independence-lem} and Lemma~\ref{lem:phi-ssc}, we have
		\begin{align*}
			&\psi^{(r)}(\theta)-\frac{1}{1-d_k\eta_k}\psi_k^{(r)}(\theta)\\
			&=\phi^{(r)}(0,\ldots, 0, r^k\eta_k, \ldots, r^J\eta_J, 0_H)-\frac{1}{1-d_k\eta_k}\phi_k^{(r)}(0,\ldots, 0, r^{k+1}\eta_{k+1}, \ldots, r^J\eta_J, 0_H)+o(1)\\
			&=o(1).
		\end{align*}
		Hence, we first have
		\begin{equation}
			\label{eq:proof-step1}
			\phi^{(r)}(0,\ldots, 0, r^k\eta_k, \ldots, r^J\eta_J, 0_H)-\frac{1}{1-d_k\eta_k}\phi_k^{(r)}(0,\ldots, 0, r^{k+1}\eta_{k+1}, \ldots, r^J\eta_J, 0_H)=o(1).
		\end{equation}
		
		If we further set $\eta_k=0$ in~\eqref{eq:proof-step1}, then we obtain
		\begin{equation*}
			\phi^{(r)}(0,\ldots, 0, r^{k+1}\eta_{k+1}, \ldots, r^J\eta_J, 0_H)-\phi_k^{(r)}(0,\ldots, 0, r^{k+1}\eta_{k+1}, \ldots, r^J\eta_J, 0_H)=o(1).
		\end{equation*}
		
		Combining the pieces above, we have proved this proposition, that
		\begin{equation*}
			\phi^{(r)}(0,\ldots, 0, r^k\eta_k, \ldots, r^J\eta_J, 0_H)-\frac{1}{1-d_k\eta_k}\phi^{(r)}(0,\ldots, 0, r^{k+1}\eta_{k+1}, \ldots, r^J\eta_J, 0_H)=o(1).
		\end{equation*}
		
	\end{proof}
	
	We inductively apply Proposition~\ref{prop:independent-prop}, and as such, we prove Theorem~\ref{thm:main} and conclude that for fixed $\eta\in\R_-^J$, we have
	\begin{equation*}
		\lim_{r\to0}\phi^{(r)}(r\eta_1,\ldots, r^J\eta_J, 0_H)=\prod_{k=1}^J\frac{1}{1-d_k\eta_k}.
	\end{equation*}
	Hence, we conclude that the scaled queue length vector process converges to a product-form limit, with each component in the product form following an exponential distribution, when subject to the multi-scale heavy traffic condition.

	\section{Conclusion}
	\label{sec:conclude}

	In this work, we analyze the asymptotics of multiclass networks under static buffer priority policies in a multi-scale heavy traffic regime. Using the BAR approach, we establish uniform moment bounds for scaled low-priority queue lengths and prove that the stationary distribution of the scaled queue length vector process converges to a product-form limit, where each component follows an exponential distribution. 
	Our numerical experiments collectively demonstrate that these product-form limits provide a solid foundation for performance evaluation. We leave the exploration of additional ways to leverage our theory for more refined performance evaluations to future research.

	\newpage
	\appendix
	
	\section{Postponed Proofs of Main Result}
	\label{sec:main-result-proof}
	
	In this section, we present all delayed proofs for results and lemmas from Sections 3 and 4, leading to the establishment of our main result, Theorem~\ref{thm:main}.

	\subsection{Proof of Lemma~\ref{lem:transform-ssc1}}
	\label{sec:ssc1-proof}
	
	In this section, we prove Lemma~\ref{lem:transform-ssc1}, whose proof closely follows that of Lemma 7.8 of~\cite{BravDaiMiya2023}. For completeness, we provide the full details here.
	
	\begin{proof}
		Given $\eta\in\R_-^K$, we first recall that
		\begin{equation*}
			\phi^{(r)}(r\eta_1,\ldots,r^J\eta_J,r\eta_H)=\E\Big[\exp\Big(\sum_{l=1}^L\eta_lr^lZ_l^{(r)}+\sum_{k\in\calh}\eta_krZ_k^{(r)}\Big)\Big].
		\end{equation*}
		Hence, we have
		\begin{align*}
			&\phi^{(r)}(r\eta_1,\ldots,r^J\eta_J,0) - \phi^{(r)}(r\eta_1,\ldots,r^J\eta_J,r\eta_H)\\
			&=\E\Big[\exp\Big(\sum_{l=1}^L\eta_lr^lZ_l^{(r)}\Big)\Big(1-\exp\Big(\sum_{k\in\calh}\eta_krZ_k^{(r)}\Big)\Big)\Big]\\
			&\leq \E\Big[1-\exp\Big(\sum_{k\in\calh}\eta_krZ_k^{(r)}\Big)\Big]
			\leq\sum_{k\in\calh}|\eta_k|r\E[Z_k^{(r)}].
		\end{align*}
		By the uniform moment bound on the unscaled high-priority queue lengths shown in Assumption~\ref{assumption:high-moment}, we can show the desired SSC, as $r\to0$,
		\begin{equation*}
			\phi^{(r)}(r\eta_1,\ldots,r^J\eta_J,0) - \phi^{(r)}(r\eta_1,\ldots,r^J\eta_J,r\eta_H)\to0.
		\end{equation*}
		
		We prove the following conditional version of SSC, and the remaining conditional SSCs would follow a similar proof strategy.
		\begin{align*}
			&\phi_L^{(r)}(r\eta_1,\ldots,r^J\eta_J,0) - \phi_L^{(r)}(r\eta_1,\ldots,r^J\eta_J,r\eta_H)\\
			&=\E\Big[\exp\Big(\sum_{l=1}^L\eta_lr^lZ_l^{(r)}\Big)\Big(1-\exp\Big(\sum_{k\in\calh}\eta_krZ_k^{(r)}\Big)\Big)\mid Z_{H(L)}^{(r)}=0\Big]\\
			&\leq \E\Big[1-\exp\Big(\sum_{k\in\calh}\eta_krZ_k^{(r)}\Big)\mid Z_{H(L)}^{(r)}=0\Big]\\
			&\leq \frac{1}{\Prob(Z_{H(L)}^{(r)}=0)} \E\Big[\Big(\sum_{k\in\calh}|\eta_k|rZ_k^{(r)}\Big)\mathbbm1\{Z_{H(L)}^{(r)}=0\}\Big].
		\end{align*}
		Next, we use H\"{o}lder's inequality and obtain
		\begin{equation*}
			\E\Big[\Big(\sum_{k\in\calh}|\eta_k|rZ_k^{(r)}\Big)\mathbbm1\{Z_{H(L)}^{(r)}=0\}\Big]\leq \E\Big[\Big(\sum_{k\in\calh}|\eta_k|rZ_k^{(r)}\Big)^p\Big]^{1/p}\E\Big[\mathbbm1\{Z_{H(L)}^{(r)}=0\}\Big]^{1/q}.
		\end{equation*}
		Consider setting $p=L+\epsilon_0$, then following Assumption~\ref{assumption:high-moment}, we are able to conclude that
		\begin{align*}
			&\phi_L^{(r)}(r\eta_1,\ldots,r^J\eta_J,0) - \phi_L^{(r)}(r\eta_1,\ldots,r^J\eta_J,r\eta_H)\\
			&\qquad\leq r^{\epsilon_0/(L+\epsilon_0)}\E\Big[\Big(\sum_{k\in\calh}|\eta_k|Z_k^{(r)}\Big)^{L+\epsilon_0}\Big]^{1/(L+\epsilon_0)}
			\to0,\quad\text{as } r\to0.
		\end{align*}
		As such, we have proven the desired conditional SSC.
	\end{proof}

	\subsection{Proof of Lemma~\ref{lem:phi-ssc}}
	\label{sec:ssc-proof}

	In this section, we establish the SSC results in Lemma~\ref{lem:phi-ssc}. Rather than demonstrating a direct connection between $\phi$ and $\psi$, we introduce an intermediary term, 
	\begin{equation}
		\varphi(\theta)=\E[g_{\theta,k,r}(Z)],
	\end{equation}
	where $g_{\theta,k,r}$ is defined in~\eqref{eq:g-test}, to bridge the two MGFs of interest $\phi$ and $\psi$. We state the following two lemmas, each connecting $\varphi$ to $\phi$ and $\varphi$ to $\psi$, which collectively lead to Lemma~\ref{lem:phi-ssc}.
	
	\begin{lemma}
		\label{lem:phi-ssc-1}
		Fix $\eta\in\R_-^L$ and $k\in\call$. For $\theta\equiv\theta(k,r,\eta)$ following \eqref{eq:theta-rule-1}--\eqref{eq:theta-rule-3}, we have
		\begin{align}
			&\phi^{(r)}(0_{1:k-1},r^k\eta_k, r^{k+1}\eta_{k+1},\ldots, r^J\eta_J,0_H)-\varphi^{(r)}(\theta)=o(1)\label{eq:ssc-1-appendix}\\
			&\phi_k^{(r)}(0_{1:k}, r^{k+1}\eta_{k+1},\ldots, r^J\eta_J,0_H)-\varphi_k^{(r)}(\theta)=o(1)\label{eq:ssc-2-appendix}.
		\end{align}
		For $l\in\calh$, we have
		\begin{equation}
			\phi_{l}^{(r)}(0_{1:k-1},r^k\eta_k, r^{k+1}\eta_{k+1},\ldots, r^J\eta_J,0_H)-\varphi_{l}^{(r)}(\theta)=o(1)\label{eq:ssc-high-1-appendix}.
		\end{equation}
		
	\end{lemma}
	
	\begin{lemma}
		\label{lem:psi-ssc}
		Fix $\eta\in\R_-^L$ and $k\in\call$. For $\theta\equiv\theta(k,r,\eta)$ following \eqref{eq:theta-rule-1}--\eqref{eq:theta-rule-3}, we have
		\begin{equation}
			\varphi^{(r)}(\theta) - \psi^{(r)}(\theta)=o(1),\quad\text{and}\quad
			\varphi_k^{(r)}(\theta) - \psi_k^{(r)}(\theta)=o(1),\quad\forall k\in\calk.
			\label{eq:psi-phi-ssc}
		\end{equation}
		
	\end{lemma}
	
	Lemma~\ref{lem:psi-ssc} is a generalization of (7.28) of~\cite{BravDaiMiya2023}, and hence we omit the proof here. In the remainder of this section, we dedicate our efforts to proving the six SSC phenomena stated in Lemma~\ref{lem:phi-ssc-1}.

	We start with proving the three SSC results for $\theta$.
	\begin{proof}[Proof of \eqref{eq:ssc-1-appendix}]
		For fixed $\eta\in\R_-^L$, $k\in\call$, and $\theta\equiv\theta(k,\eta,r)$ as defined in \eqref{eq:theta-rule-1}--\eqref{eq:theta-rule-3}, we have
		\begin{align}
			&\Big\vert g_{(0_{1:k-1},r^k\eta_k,\ldots, r^L\theta_L,0_H)}(Z^{(r)}) - g_{\theta,k,r}(Z^{(r)})\Big\vert\nonumber\\
			&=\Big\vert\exp\Big(\underbrace{\sum_{l=k}^Lr^l\eta_lZ_l^{(r)}}_{\leq0} \Big)\nonumber\\
			&\qquad\qquad\Big(1 - \exp\Big(\langle\theta_{1:k-1},Z_{1:k-1}^{(r)}\wedge1/r^k\rangle + (\theta_k - r^k\eta_k)(Z_k^{(r)}\wedge1/r^{k+1})
			+ \langle\theta_H, Z_H^{(r)}\wedge 1/r\rangle\Big)\Big)\Big\vert\nonumber\\
			&\leq\Big\vert1 - \exp\Big(\langle\theta_{1:k-1},Z_{1:k-1}^{(r)}\wedge1/r^k\rangle 
			+ (\theta_k - r^k\eta_k)(Z_k^{(r)}\wedge1/r^{k+1})
			+ \langle\theta_H, Z_H^{(r)}\wedge 1/r\rangle\Big)\Big\vert\nonumber\\
			&\leq \Big\vert\langle\theta_{1:k-1},Z_{1:k-1}^{(r)}\wedge1/r^k\rangle 
			+ (\theta_k - r^k\eta_k)(Z_k^{(r)}\wedge1/r^{k+1}) 
			+ \langle\theta_H, Z_H^{(r)}\wedge 1/r\rangle\Big\vert\label{eq:ssc-1-term-1}\\
			&\qquad\qquad\times\exp\Big(\Big\vert\langle\theta_{1:k-1},Z_{1:k-1}^{(r)}\wedge1/r^k\rangle
			+ (\theta_k - r^k\eta_k)(Z_k^{(r)}\wedge1/r^{k+1})
			+ \langle\theta_H, Z_H^{(r)}\wedge 1/r\rangle\Big\vert\Big),\label{eq:ssc-1-term-2}
		\end{align}
		where we make use of the inequality $|e^x-e^y|\leq |x-y|e^{|x-y|}$ to obtain the last inequality.
		
		It is important to establish bounds for \eqref{eq:ssc-1-term-1} and \eqref{eq:ssc-1-term-2}. Starting with \eqref{eq:ssc-1-term-2}, we have
		\begin{align}
			\eqref{eq:ssc-1-term-2}
			&\leq \exp\Big(\sum_{l=1}^{k-1}\frac{|\theta_l|}{r^k} 
			+ \frac{|\theta_k-r^k\eta_k|}{r^{k+1}} + \sum_{l\in\calh}\frac{|\theta_{l}|}{r}\Big)\label{eq:bounded-comp},\\
			&\leq\exp\Big(\sum_{l=1}^{k-1}\Big|\sum_{l'=k}^Lc_{ll'}\eta_{l'}\Big|+\Big|\sum_{l'=k+1}^Lc_{kl'}\eta_{l'}\Big|+\Big|\sum_{l\in\calh}c_{ll'}\eta_{l'}\Big|\Big):=\Gamma<\infty.\nonumber
		\end{align}
		where $c_{ll'}$ is the constant coefficient defined in~\eqref{eq:theta-rule-1}--\eqref{eq:theta-rule-3}. Hence, we observe that $\eqref{eq:ssc-1-term-2}$ can be uniformly upper bounded by some constant independent of $r$.
		For \eqref{eq:ssc-1-term-1}, we note that
		\begin{align*}
			\eqref{eq:ssc-1-term-1}&\leq\sum_{l=1}^{k-1}|\theta_l\mid Z_l^{(r)} + |\theta_k-r^k\eta_k\mid Z_k^{(r)} 
			+ \sum_{l'\in\calh}|\theta_{l'}\mid Z_{l'}^{(r)}\\
			&=\sum_{l=1}^{k-1}\frac{|\theta_l|}{r^l}(r^lZ_l^{(r)}) + \frac{|\theta_k-r^k\eta_k|}{r^k}(r^kZ_k^{(r)}) + \sum_{l'\in\calh}|\theta_{l'}|Z_{l'}^{(r)}.
		\end{align*}
		
		Combining the analyses of both terms, we take expectations on both sides and obtain
		\begin{align*}
			&\Big\vert\phi^{(r)}(0_{1:k-1},r^k\eta_k, r^{k+1}\eta_{k+1},\ldots, r^J\eta_J,0_H)-\varphi^{(r)}(\theta)\Big\vert\\
			&\leq \Gamma\cdot \Big\vert \sum_{l=1}^{k-1}\underbrace{\frac{|\theta_l|}{r^l}}_{=O(r^{k-l})}\E[r^lZ_l^{(r)}] + \underbrace{\frac{|\theta_k-r^k\eta_k|}{r^k}}_{=O(r)}\E[r^kZ_k^{(r)}] + \sum_{l'\in\calh}\underbrace{|\theta_{l'}|}_{=O(r^k)}\E[Z_{l'}^{(r)}]\Big\vert,
		\end{align*}
		where the orders are due to the choice of $\theta$ under \eqref{eq:theta-rule-1}--\eqref{eq:theta-rule-3}. 
		Making use of Assumption~\ref{assumption:low-moment} and ~\ref{assumption:high-moment}, i.e., for some $r_0\in(0,1)$, 
		\begin{equation*}
			\sup_{r\in(0,r_0)}\E\left[r^lZ_l^{(r)}\right]<\infty\quad\forall\,l\in\call
			\quad\text{and}\quad
			\sup_{r\in(0,r_0)}\E\left[Z_{l}^{(r)}\right]<\infty\quad\forall\,l\in\calh,
		\end{equation*}
		we are able to prove \eqref{eq:ssc-1-appendix} that
		\begin{equation*}
			\phi^{(r)}(0_{1:k-1},r^k\eta_k, r^{k+1}\eta_{k+1},\ldots, r^J\eta_J,0_H)-\varphi^{(r)}(\theta)=o(1).
		\end{equation*}
	\end{proof}

	To prove~\eqref{eq:ssc-2-appendix}, we need the following technical lemma, whose proof is delayed to the next subsection Appendix~\ref{sec:cross-condition-proof}.
	\begin{lemma}
		\label{lemma:cross-condition} 
		Under the same condition in Theorem~\ref{assumption:low-moment}, for each $k\in\call$, 
		\begin{equation*}
			\E\Big[\sum_{l\in\call,l<k}r^{l+1}Z_{l}^{(r)}\mid Z_{H(k)}^{(r)}=0\Big]=o(1).
		\end{equation*}
	\end{lemma}

	\begin{proof}[Proof of \eqref{eq:ssc-2-appendix}]
		The conditional SSC can be proven following a similar approach as in the proof of \eqref{eq:ssc-1-appendix} shown above. We start by observing the following inequality,
		\begin{align}
			& \Big\vert \phi_k^{(r)}(0_{1:k-1},r^k\eta_k,\ldots, r^J\eta_J,0_H) - \varphi_k^{(r)}(\theta)\Big\vert\nonumber\\
			&\leq \Big(\sum_{i=1}^{k-1}|\theta_i|\E\left[Z_i^{(r)}\mid Z_{H(k)}^{(r)}=0\right] 
			+ |\theta_k-r^k\eta_k|\E\left[Z_k^{(r)}\mid Z_{H(k)}^{(r)}=0\right] 
			+ \sum_{j\in\calh}|\theta_j|\E\left[Z_j^{(r)}\mid Z_{H(k)}^{(r)}=0\right]\Big)\label{eq:conditional-ssc-analysis}\\
			&\quad\times \underbrace{\exp\Big(\sum_{i=1}^{k-1}\frac{|\theta_i|}{r^k} + \frac{|\theta_k-r^k\eta_k|}{r^{k+1}} + \sum_{j\in\calh}\frac{|\theta_j|}{r}\Big)}_{\leq\Gamma=O(1)\text{ following }\eqref{eq:bounded-comp}}\nonumber.
		\end{align}
		
		Therefore, we focus on understanding the three summation terms in \eqref{eq:conditional-ssc-analysis}. 
		
		Firstly, for $i\in\call$ and $i<k$, we make use of Lemma~\ref{lemma:cross-condition} and have
		\begin{equation*}
			|\theta_i|\E\left[Z_i^{(r)}\mid Z_{H(l)}^{(r)}=0\right]=\underbrace{\frac{|\theta_i|}{r^{i+1}}}_{=O(r^{k-(i+1)})}\underbrace{\E\left[r^{i+1}Z_i^{(r)}\mid Z_{H(l)}^{(r)}=0\right]}_{=o(1)}=o(1).
		\end{equation*}
		
		Next, it is clear that $\E[Z_k^{(r)}\mid Z_{H(k)}^{(r)}=0]=0$, since $k\in H(k)$ and the condition $Z_{H(k)}^{(r)}=0$ naturally implies that $Z_k^{(r)}=0$.
		
		Lastly, for $j\in\calh$, we have
		\begin{align*}
			&|\theta_j|\E\left[Z_j^{(r)}\mid Z_{H(j)}^{(r)}=0\right]=
			\frac{|\theta_j|}{\Prob(Z_{H(k)}^{(r)}=0)}\E\left[Z_j^{(r)}\cdot\mathbbm{1}\{Z_{H(k)}^{(r)}=0\}\right]\\
			&=\frac{|\theta_j|}{\Prob(Z_{H(k)}^{(r)}=0)}\Big(\E\left[Z_j^{(r)}\cdot\mathbbm{1}\{Z_{H(k)}^{(r)}=0, Z_j^{(r)}>r^{-1/2}\}\right]+\E\left[Z_j^{(r)}\cdot\mathbbm{1}\{Z_{H(k)}^{(r)}=0, Z_j^{(r)}\leq r^{-1/2}\}\right]\Big)\\
			&\leq\frac{|\theta_j|}{\Prob(Z_{H(k)}^{(r)}=0)}\Big(\E\left[Z_j^{(r)}\cdot\mathbbm{1}\{ Z_j^{(r)}>r^{-1/2}\}\right]+r^{-1/2}\E\left[\mathbbm{1}\{Z_{H(k)}^{(r)}=0\}\right]\Big)\\
			&\overset{\text{(i)}}{=}\underbrace{\frac{|\theta_j|}{\Prob(Z_{H(k)}^{(r)}=0)}}_{=O(1)}\underbrace{\E\left[Z_j^{(r)}\cdot\mathbbm{1}\{ Z_j^{(r)}>r^{-1/2}\}\right]}_{=o(1)}+\underbrace{r^{-1/2}|\theta_j|}_{=O(r^{1/2})}
			=o(1),
		\end{align*}
		where in (i) we use the order $\theta=O(r^k)$, $\Prob(Z_{H(k)}^{(r)}=0)=r^kb_k$ under the multi-scale heavy traffic assumption in Assumption~\ref{assumption:multi-scale-heavy-traffic}, and the uniform moment bound of high-priority queue lengths under Assumption~\ref{assumption:high-moment} to conclude the desired $o(1)$ order.
		
		As such, we have first demonstrated that
		\begin{equation*}
			\phi_k^{(r)}(0_{1:k-1},r^k\eta_k, r^{k+1}\eta_{k+1},\ldots, r^J\eta_J,0_H)-\varphi_k^{(r)}(\theta)=o(1).
		\end{equation*}
		
		Also, it is easy to see that
		\begin{equation*}
			\phi_k^{(r)}(0_{1:k-1},r^k\eta_k, r^{k+1}\eta_{k+1},\ldots, r^J\eta_J,0_H)=\phi_k^{(r)}(0_{1:k}, r^{k+1}\eta_{k+1},\ldots, r^J\eta_J,0_H),
		\end{equation*}
		for the condition $Z_{H(k)}^{(r)}=0$ implies $Z_k^{(r)}=0$, which nullifies the impact of $\theta_k$ in $\phi_k$.
		
		Hence, we have proven the second SSC result \eqref{eq:ssc-2-appendix}
		\begin{equation*}
			\phi_k^{(r)}(0_{1:k}, r^{k+1}\eta_{k+1},\ldots, r^J\eta_J,0_H)-\varphi^{(r)}(\theta)=o(1).
		\end{equation*}
	\end{proof}
	
	\begin{proof}[Proof of \eqref{eq:ssc-high-1-appendix}]
		It is easier to prove the SSC when conditioning on high-priority classes, as shown in Lemma~\ref{lem:idle-prob} that $\Prob(Z_{H(l)}^{(r)}=0)\to\beta_{l}>0$ for $l\in\calh$. Hence, we have
		\begin{align*}
			&\Big\vert\phi_{l}^{(r)}(0_{1:k-1},r^k\eta_k, \ldots, r^J\eta_J,0_H)-\varphi_{l}^{(r)}(\theta)\Big\vert\\
			&\leq \Big(\sum_{i=1}^{k-1}|\theta_i|\E\left[Z_i^{(r)}\mid Z_{H(l)}^{(r)}=0\right] 
			+ |\theta_k-r^k\eta_k|\E\left[Z_k^{(r)}\mid Z_{H(l)}^{(r)}=0\right] 
			+ \sum_{j\in\calh}|\theta_j|\E\left[Z_j^{(r)}\mid Z_{H(l)}^{(r)}=0\right]\Big)\\
			&\quad\times \underbrace{\exp\Big(\sum_{i=1}^{k-1}\frac{|\theta_i|}{r^k} + \frac{|\theta_k-r^k\eta_k|}{r^{k+1}} + \sum_{j\in\calh}\frac{|\theta_j|}{r}\Big)}_{\leq\Gamma=O(1)\text{ following }\eqref{eq:bounded-comp}}.
		\end{align*}
		We note that when subject to Assumption~\ref{assumption:low-moment}, for $i\in\call$ and $i<k$, 
		\begin{equation*}
			|\theta_i|\E\left[Z_i^{(r)}\mid Z_{H(l)}^{(r)}=0\right]=\frac{|\theta_i|}{\beta_{l}^{(r)}}\E\left[Z_i^{(r)}\cdot\mathbbm{1}\{Z_{H(l)}^{(r)}=0\}\right]
			\leq\underbrace{\frac{|\theta_i|}{r^i}}_{=O(r)}\frac{1}{\beta_{l}^{(r)}}\E\left[r^iZ_i^{(r)}\right]=o(1).
		\end{equation*}
		Similarly, we have
		\begin{align*}
			|\theta_k-r^k\eta_k|\E\left[Z_k^{(r)}\mid Z_{H(l)}^{(r)}=0\right]&=\frac{|\theta_k-r^k\eta_k|}{\beta_{l}^{(r)}}\E\left[Z_k^{(r)}\cdot\mathbbm{1}\{Z_{H(l)}^{(r)}=0\}\right]\\
			&\leq\underbrace{\frac{|\theta_k-r^k\eta_k|}{r^k}}_{=O(r)}\frac{1}{\beta_{l}^{(r)}}\E\left[r^kZ_k^{(r)}\right]=o(1).
		\end{align*}
		Lastly, for $j\in\calh$, under Assumption~\ref{assumption:high-moment},
		\begin{equation*}
			|\theta_j|\E\left[Z_j^{(r)}\mid Z_{H(l)}^{(r)}=0\right]=\frac{|\theta_j|}{\beta_{l}^{(r)}}\E\left[Z_j^{(r)}\cdot\mathbbm{1}\{Z_{H(l)}^{(r)}=0\}\right]\leq \underbrace{\frac{|\theta_j|}{\beta_{l}^{(r)}}}_{=O(r^k)}\E\left[Z_j^{(r)}\right]=o(1).
		\end{equation*}
		
		Thus, SSC holds for $l\in\calh$,
		\begin{equation*}
			\phi_{l}^{(r)}(0_{1:k-1},r^k\eta_k, r^{k+1}\eta_{k+1},\ldots, r^J\eta_J,0_H)-\varphi_{l}^{(r)}(\theta)=o(1).
		\end{equation*}
	\end{proof}

	\subsection{Proof of Lemma~\ref{lemma:cross-condition}}
	\label{sec:cross-condition-proof}
	In this section, we provide the proof to Lemma~\ref{lemma:cross-condition}. This lemma follows from the uniform moment bound condition of scaled low-priority queue lengths in Theorem~\ref{assumption:low-moment}.
	
	\begin{proof}
		We first note that for $k\in\call$ and $l<k$,
		\begin{align*}
			\E\Big[r^{l+1}Z_{l}^{(r)}\mid Z_{H(k)}^{(r)}=0\Big]&=\frac{1}{\Prob(Z_{H(k)}^{(r)}=0)}\E\Big[r^{l+1}Z_{l}^{(r)}\mathbbm1\{Z_{H(k)}^{(r)}=0\}\Big]\\
			&\overset{\text{(i)}}{\leq} \frac{1}{\Prob(Z_{H(k)}^{(r)}=0)}\E\Big[(r^{l+1}Z_{l}^{(r)})^p\Big]^{1/p}\E\Big[\mathbbm1\{Z_{H(k)}^{(r)}=0\}\Big]^{1/q}\\
			&\overset{\text{(ii)}}{=}r^{k/q-k+1}b_k^{1/q-1} \E\Big[(r^{l}Z_{l}^{(r)})^p\Big]^{1/p},
		\end{align*}
		where we make use of the H\"{o}lder's inequality to obtain the inequality in (i) and the multi-scale heavy traffic assumption in Theorem~\ref{assumption:multi-scale-heavy-traffic} to achieve (ii).
		
		Setting $p=k+\epsilon_0$, and hence $q=(k+\epsilon_0)/(k+\epsilon_0-1)$, we have
		\begin{equation*}
			\E\Big[r^{l+1}Z_{l}^{(r)}\mid Z_{H(k)}^{(r)}=0\Big]\leq r^{\epsilon_0/(k+\epsilon_0)}b_k^{-1/(k+\epsilon_0)}\E\Big[(r^lZ_{l}^{(r)})^{k+\epsilon_0}\Big]^{1/(k+\epsilon_0)}.
		\end{equation*}
		Hence, under Theorem~\ref{assumption:low-moment}, it is clear that 
		\begin{equation*}
			\E\Big[r^{l+1}Z_{l}^{(r)}\mid Z_{H(k)}^{(r)}=0\Big]=o(1),
		\end{equation*}
		and this completes the proof.
	\end{proof}

	\subsection{Proof of Lemma~\ref{lem:mgf-high}}
	\label{sec:mgf-high-proof}
	
	In this section, we present the proof of Lemma~\ref{lem:mgf-high}, which extends the proof of Lemma 7.9 of \cite{BravDaiMiya2023}.
	
	\begin{proof}
		Fix $\eta\in\R_-^L$ and index $k$. We first set $\theta_L$ following the rules in \eqref{eq:theta-rule-1}--\eqref{eq:theta-rule-3}. Next, we fix index $h\in\calh$, and consider the following choice of $\theta_H$
		\begin{equation*}
			\theta_H=-A_H^{-\top}A_{LH}^\top\theta_L + A_H^{-\top} e_H^{(h)}r^k,
		\end{equation*}
		where $e_H^{(h)}$ is a $H$-dimensional vector with all entries being $0$ except for the $h$-th element taking value $1$.
		Under this choice of $\theta_H$, we first observe that
		\begin{align*}
			&\mu_{l-}\bar\zeta_{l-}(\theta)-\mu_l\bar\zeta_l(\theta)=0,\quad l\in\calh\backslash\{h\},\\
			&\mu_{h-}\bar\zeta_{h-}(\theta)-\mu_h\bar\zeta_h(\theta)=r^k.
		\end{align*}
		Therefore, $\theta=(\theta_L,\theta_H)$ still satisfies the rules in \eqref{eq:theta-rule-1}--\eqref{eq:theta-rule-3} and is of order $|\theta|=O(r^k)$.
		
		Bringing this choice of $\theta$ back into the BAR~\eqref{eq:bar-taylor}, we obtain
		\begin{align}
			q^\ast(\theta)\psi^{(r)}(\theta)&-\sum_{l\in\call}\beta_l^{(r)}\mu_l^{(r)}\zeta_l^\ast(\theta)\Big(\psi_l^{(r)}(\theta)-\psi^{(r)}(\theta)\Big)\label{eq:mgf-high-proof-1}\\
			&+\beta_h^{(r)}\underbrace{\Big(\mu_{h-}^{(r)}\bar\zeta_{h-}^\ast(\theta) - \mu_h^{(r)}\zeta_h^\ast(\theta)\Big)}_{=r^k}\Big(\psi_h^{(r)}(\theta)-\psi^{(r)}(\theta)\Big)\label{eq:mgf-high-proof-2}\\
			&+\sum_{l\in\calh}\beta_l^{(r)}\Big(\mu_{l-}^{(r)}\tilde\zeta_{l-}(\theta) - \mu_l^{(r)}\tilde\zeta_l(\theta)\Big)\Big(\psi_l^{(r)}(\theta)-\psi^{(r)}(\theta)\Big)= o(|\theta|^2),\label{eq:mgf-high-proof-3}
		\end{align}
		where we observe that terms in \eqref{eq:mgf-high-proof-1} and \eqref{eq:mgf-high-proof-3} are at least of order $O(r^{k+1})$. Therefore, when we divide both sides by $r^k$, and then take $r\to0$, we have
		\begin{equation*}
			\psi_h^{(r)}(\theta)-\psi^{(r)}(\theta)=o(1).
		\end{equation*}
		Together with the SSC results~\eqref{eq:ssc-1} and \eqref{eq:ssc-high-1}, we have completed the proof and shown the desired SSC result, i.e.,
		\begin{equation*}
			\phi_{h}^{(r)}(0_{1:k-1},r^k\eta_k, r^{k+1}\eta_{k+1},\ldots, r^J\eta_J,0_H)-\phi^{(r)}(0_{1:k-1},r^k\eta_k, r^{k+1}\eta_{k+1},\ldots, r^J\eta_J,0_H)=o(1).
		\end{equation*}
	\end{proof}

	\subsection{Proof of Lemma~\ref{lem:independence-lem}}
	\label{sec:indep-lem-proof}
	
	In this section, we prove Lemma~\ref{lem:independence-lem}, which plays a crucial role in constructing the inductive argument presented later in Proposition~\ref{prop:independent-prop}. The essence of the proof for the two limits lies in using the asymptotic BAR in Lemma~\ref{prop:asymptotic-bar} as the foundation and leveraging the special structure of $\theta$ defined in~\eqref{eq:theta-def-first}--\eqref{eq:theta-def-last}.
	
	\begin{proof}[Proof of \eqref{eq:mgf-eq1}]
		We first note that $\theta$ as defined in \eqref{eq:theta-def-first}--\eqref{eq:theta-def-last} is carefully chosen to achieve the following three properties. Firstly, the order of $|\theta|$ is set at $r^k$, where $k$ is the fixed index, and hence we know that $q^\ast(\theta)=O(r^{2k})$ and $\bar\zeta_l(\theta)=O(r^k)$ for $l\in\calk$. Second, more specifically, when $l\in\call$ and $l<k$, we have $\Bar{\zeta}_l(\theta)=0$. Lastly, $\mu_{l-}^{(r)}\Bar{\zeta}_{l}(\theta)-\mu_l^{(r)}\Bar{\zeta}_l(\theta)=0$ for $l\in\calh$.
		
		Recall the asymptotic BAR in~\eqref{eq:bar-taylor}, which is repeat below for easier reference,
		\begin{align}
			&\Tilde{q}(\theta)\psi^{(r)}(\theta)-\sum_{l\in\call}\beta_l^{(r)}\mu_l^{(r)}\Bar{\zeta}_l(\theta)\Big(\psi_l^{(r)}(\theta)-\psi^{(r)}(\theta)\Big)\label{eq:bar-general-1}\\
			&
			+\sum_{l\in\calh}\beta_l^{(r)}\Big(\mu_{l-}^{(r)}\bar{\zeta}_{l-}(\theta) - \mu_l^{(r)}\Bar{\zeta}_l(\theta)\Big)\Big(\psi_l^{(r)}(\theta)-\psi^{(r)}(\theta)\Big) \label{eq:bar-general-2}\\
			&= \sum_{l\in\call}\beta_l^{(r)}\mu_l^{(r)}\tilde{\zeta}_l(\theta)\Big(\psi_l^{(r)}(\theta)-\psi^{(r)}(\theta)\Big) \\
			&- \sum_{l\in\calh}\beta_l^{(r)}\Big(\mu_{l-}^{(r)}\tilde{\zeta}_{l-}(\theta) - \mu_l^{(r)}\tilde{\zeta}_l(\theta)\Big)\Big(\psi_l^{(r)}(\theta)-\psi^{(r)}(\theta)\Big) + o(|\theta|^2).\label{eq:bar-general-3}
		\end{align}
		We next analyze the above asymptotic BAR term by term under the special choice of $\theta$.
		
		We start with the second term in~\eqref{eq:bar-general-1}.
		\begin{align*}
			&\sum_{l\in\call}\beta_l^{(r)}\mu_l^{(r)}\Bar{\zeta}_l(\theta)\Big(\psi_l^{(r)}(\theta)-\psi^{(r)}(\theta)\Big)
			\overset{\text{(i)}}{=}\sum_{l=k}^L\beta_l^{(r)}\mu_l^{(r)}\Bar{\zeta}_l(\theta)\Big(\psi_l^{(r)}(\theta)-\psi^{(r)}(\theta)\Big)\\
			&\overset{\text{(ii)}}{=}r^kb_k\mu_l^{(r)}\Bar{\zeta}_l(\theta)\Big(\psi_l^{(r)}(\theta)-\psi^{(r)}(\theta)\Big)
			+\underbrace{\sum_{l=k+1}^Lr^lb_l\mu_l^{(r)}\Bar{\zeta}_l(\theta)\Big(\psi_l^{(r)}(\theta)-\psi^{(r)}(\theta)\Big)}_{=o(r^{2k})}\\
			&\overset{\text{(iii)}}{=}r^kb_k\mu_l^{(r)}\Bar{\zeta}_l(\theta)\Big(\psi_l^{(r)}(\theta)-\psi^{(r)}(\theta)\Big) + o(r^{2k}),
		\end{align*}
		where in (i) we use the special property of $\theta$ that $\bar\zeta_l(\theta)=0$ when $l\in\call$ and $l<k$. We next make use of Lemma~\ref{lem:idle-prob} to obtain (ii). Consolidate the orders, we lastly obtain (iii).
		
		For the term in~\eqref{eq:bar-general-2}, we use the third property that $\theta_H=-A_H^{-\top}A_{LH}^\top\theta_L$, which is equivalent to
		\begin{equation*}
			\mu_{l-}^{(r)}\bar{\zeta}_{l-}(\theta) - \mu_l^{(r)}\Bar{\zeta}_l(\theta)=0,\quad\forall l\in\calh.
		\end{equation*}
		Hence, we know that the summation term in \eqref{eq:bar-general-2} equals zero, i.e.,
		\begin{equation*}
			\sum_{l\in\calh}\beta_l^{(r)}\Big(\mu_{l-}^{(r)}\bar{\zeta}_{l-}(\theta) - \mu_l^{(r)}\Bar{\zeta}_l(\theta)\Big)\Big(\psi_l^{(r)}(\theta)-\psi^{(r)}(\theta)\Big)=0.
		\end{equation*}
		
		Lastly, since $\Tilde{\zeta
		}_l$ is quadratic in $\theta$, it is easy to see that the remaining terms in~\eqref{eq:bar-general-3} all are of order $o(r^{2k})$,
		\begin{align*}
			&\sum_{l\in\call}\beta_l^{(r)}\mu_l^{(r)}\tilde{\zeta}_l(\theta)\Big(\psi_l^{(r)}(\theta)-\psi^{(r)}(\theta)\Big)=\sum_{l=1}^Lr^lb_l\mu_l^{(r)}\underbrace{\tilde{\zeta}_l(\theta)}_{=O(r^{2k})}\Big(\psi_l^{(r)}(\theta)-\psi^{(r)}(\theta)\Big)=o(r^{2k})\\
			&\sum_{l\in\calh}\beta_l^{(r)}\Big(\mu_{l-}^{(r)}\tilde{\zeta}_{l-}(\theta) - \mu_l^{(r)}\tilde{\zeta}_l(\theta)\Big)\Big(\psi_l^{(r)}(\theta)-\psi^{(r)}(\theta)\Big)= o(r^{2k}).
		\end{align*}
		
		Consolidating the above analyses, we obtain the following BAR,
		\begin{equation*}
			\tilde{q}(\theta)\psi^{(r)}(\theta) - r^kb_k\mu_l^{(r)}\Bar{\zeta}_l(\theta)\Big(\psi_l^{(r)}(\theta)-\psi^{(r)}(\theta)\Big) = o(r^{2k}).
		\end{equation*}
		Therefore, we divide $r^{2k}$ on both sides and take $r\to0$, and obtain
		\begin{equation*}
			\Big(\sigma_k^2/2\Big)\eta_k^2\psi^{(r)}(\theta)+b_k\mu_k\Big(1-w_{kk}\Big)\eta_k\Big(\psi_k^{(r)}(\theta)-\psi^{(r)}(\theta)\Big)=o(1),
		\end{equation*}
		with $\sigma_k^2$ defined in \eqref{eq:sig-def}.
		
		As such, we have shown the desired MGF relationship as given in \eqref{eq:mgf-eq1}
		\begin{equation*}
			\psi^{(r)}(\theta)-\frac{1}{1-d_k\eta_k}\psi_k^{(r)}(\theta)=o(1),
		\end{equation*}
		with $d_k=\frac{\sigma_k^2}{2(1-w_{kk})\mu_kb_k}$.
	\end{proof}
	
	\section{Proof of Uniform Moment Bound in Theorem~\ref{assumption:low-moment}}
	\label{sec:moment-bound-proof}
	In this section, we prove a general version of the uniform moment bound in Theorem~\ref{assumption:low-moment}. 
	
	\begin{assumption}
		\label{assumption:t-moment-general}
		There exist a small $\delta_0>0$ such that
		\begin{equation*}
			\E\left[ (T_{e,k}(1))^{N+1}\right] < \infty,\quad\text{for } k\in\cale\quad \text{and}\quad \E\left[ (T_{s,k}(1))^{N+1}\right] < \infty,\quad\text{for } k\in\calk.
		\end{equation*}
	\end{assumption}
	When setting $N=J+\delta_0$, the assumption above recovers Assumption~\ref{assumption:t-moment}.
	
	\begin{assumption}
		\label{assumption:high-moment-general} 
		For each $k\in\calh$, we assume a similar $N$-th moment uniform integrability for the collection of the steady-state job-count vectors $\{Z_k^{(r)}, r\in(0,1)\}$, namely,
		\begin{equation*}
			\lim_{\kappa\to\infty}\sup_{r\in(0,1)}\E\left[\Big(Z_k^{(r)}\Big)^{N}\cdot\mathbbm{1}\{Z_k^{(r)}>\kappa\}\right]=0.
		\end{equation*}
	\end{assumption}
	
	\begin{theorem}
		Given $N\geq J$, under Assumptions~\ref{assumption:t-moment-general},~\ref{assumption:high-moment-general},~\ref{assumption:ah-invertible} and~\ref{assumption:q-m-matrix},
		for each $k\in\call$, the collection of the appropriately scaled steady-state job-count vectors $\{r^kZ_k^{(r)}, r\in(0,1)\}$ is uniformly integrable up to its $N$-th moment for some $\epsilon_0>0$, namely
		\begin{equation*}
			\lim_{\kappa\to\infty}\sup_{r\in(0,1)}\E\Big[\Big(r^kZ_k^{(r)}\Big)^{N}\cdot\mathbbm{1}\{Z_k^{(r)}>\kappa\}\Big]=0.
		\end{equation*}
	\end{theorem}
	
	To prove the desired uniform integrability, it is equivalent to proving that $\exists r_0\in(0,1)$
	\begin{equation*}
		\sup_{r\in(0,r_0)}\E\Big[\Big(r^kZ_k^{(r)}\Big)^{N}\Big]<\infty,\quad\forall k\in\call.
	\end{equation*}
	
	The proof of moment bound follows an induction technique developed in~\cite{GuanChenDai2023}, for $l\in\call$ and $0\leq n\leq N$. 
	As we need to induct on two parameters $(n, k)$, the process is as follows: for each $n \leq N$, we induct on $k \in \call$, then proceed to $n+1$. By this method, we eventually show that the desired properties hold for all $(n, k)$ pairs.
	The induction hypotheses incorporate moment bounds for the scaled queue length, along with auxiliary results that bound the expectations of cross terms involving the queue length and the residual inter-arrival or service times.
	
	\begin{lemma}
		\label{lem:moment-bound-aux}
		Given $N\geq J$,  then for each pair of $1\leq k\leq J$ and $0\leq n\leq N$, there exists positive and finite constants $C_{l,n}$, $D_{l,n}$, $E_{l,n}$, and $F_{l,n}$ that are independent of $r$ such that the following statements hold for all $r\in(0,r_0)$.
		\begin{align}
			&\E_\pi\Big[\Big(r^kZ_k^{(r)}\Big)^n\Big]\leq C_{k,n}\label{eq:S1}\\
			&\sum_{l\in\cale}\E_{e,l}\Big[\Big(r^kZ_k^{(r)}\Big)^n\Big]+\sum_{l\in\calk}\E_{s,l}\Big[\Big(r^kZ_k^{(r)}\Big)^n\Big]\leq D_{k,n}\label{eq:S2}\\
			&\E_\pi\Big[\Big(r^kZ_k^{(r)}\Big)^n\Lambda_{N-n}\Big(X^{(r)}\Big)\Big]\leq E_{k,n}\label{eq:S3}\\
			&\sum_{l\in\cale}\E_{e,l}\Big[\Big(r^kZ_k^{(r)}\Big)^n\Lambda_{N-n}\Big(X^{(r)}\Big)\Big]+\sum_{l\in\calk}\E_{s,l}\Big[\Big(r^kZ_k^{(r)}\Big)^n\Lambda_{N-n}\Big(X^{(r)}\Big)\Big]\leq F_{k,n},\label{eq:S4}
		\end{align}
		where
		\begin{equation*}
			\Lambda_{n}(X^{(r)})=\sum_{l\in\cale}(R_{e,l}^{(r)})^n + \sum_{l\in\calk}(R_{s,l}^{(r)})^n\mathbbm1\{Z_{H_+(l)}^{(r)}=0\}.
		\end{equation*}
	\end{lemma}
	
	To employ the BAR technique and establish the four technical inequalities, we need to judiciously construct appropriate test functions for each case.
	Proving \eqref{eq:S1} involves employing a test function of the form
	\begin{equation}
		f_{k,n}(X^{(r)})=\frac{1}{n+1}r^{k(n-1)}\Big(\dotp{u^{(k)},Z^{(r)}}\Big)^{n+1} +r^{k(n-1)}\Big(\dotp{u^{(k)},Z^{(r)}}\Big)^nh_k\Big(X^{(r)}\Big)\label{eq:S1-test},
	\end{equation}
	where $u^{(k)}$ is as defined in~\eqref{eq:u-vec-def} and
	\begin{equation*}
		h_k(X^{(r)})=-\sum_{l\in\cale}\alpha_lu_l^{(k)}R_{e,l}^{(r)}+\sum_{l\in\calk}\Big(u_l^{(k)}-\sum_{l'\in\calk}P_{ll'}u_{l'}^{(k)}\Big)\mu_lR_{s,l}^{(r)}\mathbbm1\{Z_{H_+(l)}=0\}.
	\end{equation*}
	
	We briefly remark on the choice of $u^{(k)}$ vector. Despite that our goal is to prove the uniform moment bound for scaled queue lengths of low-priority classes, by this choice of $u^{(k)}$ vector, we are also effectively involving queue lengths of high-priority classes in the steady-state BAR evaluation. Nonetheless, the presence of these high-priority queue lengths should not raise any concerns, as it has been carefully controlled by Assumption~\ref{assumption:high-moment-general}. 
	This choice of $u$ vector is therefore unique to MCNs under SBP service policies, which is absent from the proof in~\cite{GuanChenDai2023} for GJNs.

	Proving \eqref{eq:S2}--\eqref{eq:S4} utilizes the following test functions respectively,
	\begin{align}
		&f_{k,n,D}(X^{(r)})=\Big(r^kZ_k^{(r)}\Big)^n\Lambda_1\Big(X^{(r)}\Big)\label{eq:S2-test}\\
		&f_{k,n,E}(X^{(r)})=\Big(r^kZ_k^{(r)}\Big)^n\Lambda_{N-n+1}\Big(X^{(r)}\Big)\label{eq:S3-test}\\
		&f_{k,n,F}(X^{(r)})=\Big(r^kZ_k^{(r)}\Big)^n\Lambda_{N-n}\Big(X^{(r)}\Big)\Lambda_1\Big(X^{(r)}\Big)\label{eq:S4-test}
	\end{align}
	It is important to note that for the BAR technique to be applicable, we require the test functions to be bounded. However, the four test functions shown in \eqref{eq:S1-test}--\eqref{eq:S4-test} do not meet this bounded criteria. Given the complexity of this proof, we temporarily set aside concerns about truncation and proceed to establish the desired bound. The discussion of truncation will follow a standard technique as employed in~\cite{GuanChenDai2023}, and thus is omitted here in this work.
	
	Additionally, we need to extend our induction proof to prove for non-integer valued moment bound condition. We emphasize that the extension technique is standard and has been thoroughly discussed in~\cite{GuanChenDai2023}. Therefore, we omit this extension here, and focus on analyzing the different choices of test functions and applying the induction argument to prove the desired moment bound.

	\subsection{Base Case}
	We begin the induction proof with the base case for $n=0$.
	We notice that in the base case $n=0$, \eqref{eq:S1} and \eqref{eq:S2} hold trivially. Hence, our focus is to prove \eqref{eq:S3} and \eqref{eq:S4} when $n=0$.

	\subsubsection{Proof of~\texorpdfstring{\eqref{eq:S3}}{(S3)}}
	
	In this case, we would like to show that for all $k\in\call$,
	\begin{equation*}
		\E_\pi\Big[\Lambda_N\Big(X^{(r)}\Big)\Big]\leq E_{0,k}<\infty.
	\end{equation*}
	We note that the LHS of the inequality above is independent of $k$. As such, we shall have some $E_0\equiv E_{k,0}$ for all $k\in\call$.

	When $n=0$, the test function is
	\begin{equation*}
		f_{k,0,E}(X^{(r)})=\Lambda_{N+1}\Big(X^{(r)}\Big)
	\end{equation*}
	We substitute the test function into BAR~\eqref{eq:bar-general}, and analyze both sides of the equation.
	For the LHS of BAR~\eqref{eq:bar-general}, we have
	\begin{equation*}
		-\E_\pi[\cala f(X^{(r)})]=(N+1)\E\Big[\sum_{l\in\cale}(R_{e,l}^{(r)})^N + \sum_{l\in\calk}(R_{s,l}^{(r)})^N\mathbbm{1}\{Z_l^{(r)}>0, Z_{H_+(l)}^{(r)}=0\}\Big].
	\end{equation*}
	On the other hand, the RHS of BAR~\eqref{eq:bar-general} consists of the summation of ``Palm'' expectations taken w.r.t.\ external arrivals and service completions. We analyze these ``Palm'' expectations respectively.
	\begin{align*}
		\E_{e,l}[f(X^{(r)}+\Delta_{e,l})-f(X^{(r)})]&=\E[T_{e,l}^{N+1}]/\alpha_l^{N+1}\\
		\E_{s,l}[f(X^{(r)}+\Delta_{s,l})-f(X^{(r)})]&=\E[T_{s,l}^{N+1}]/(\mu_l^{(r)})^{N+1}.
	\end{align*}
	Altogether, for the RHS of BAR~\eqref{eq:bar-general}, we have
	\begin{align*}
		&\E_\pi\Big[\sum_{l\in\cale}(R_{e,l}^{(r)})^N + \sum_{l\in\calk}(R_{s,l}^{(r)})^N\mathbbm{1}\{Z_l^{(r)}>0, Z_{H_+(l)}^{(r)}=0\}\Big]\\
		&\leq \frac{1}{N+1}\Big(
		\sum_{l\in\cale}\E[T_{e,l}^{N+1}]/\alpha_l^N + \sum_{l\in\calk}\lambda_l\E[T_{s,l}^{N+1}]/(\mu_l^{(r)})^{N+1}\Big)\\
		&\leq \frac{1}{N+1}\Big(
		\sum_{l\in\cale}\E[T_{e,l}^{N+1}]/\alpha_l^N + \sum_{l\in\calk}\E[T_{s,l}^{N+1}]/(\mu_l^{(r)})^N\Big),
	\end{align*}
	where we make use of $\lambda_l/\mu_l^{(r)}<1$ in the last inequality.
	
	Next, we note that
	\begin{equation*}
		\E_\pi\Big[(R_{s,l}^{(r)})^N\mathbbm{1}\{Z_{H(l)}^{(r)}=0\}\Big]=\E[T_{s,l}^N\mathbbm1\{Z_{H_l^{(r)}}=0\}]/(\mu_l^{(r)})^N\leq \E[T_{s,l}^N]/(\mu_l^{(r)})^N.
	\end{equation*}
	
	Combining the analyses above, we have proven the desired base case of \eqref{eq:S3},
	\begin{align*}
		&\E_\pi\Big[\Lambda_N(X^{(r)})\Big]\\
		&=\E_\pi\Big[\sum_{l\in\cale}(R_{e,l}^{(r)})^N + \sum_{l\in\calk}(R_{s,l}^{(r)})^N\mathbbm{1}\{Z_l^{(r)}>0, Z_{H_+(l)}^{(r)}=0\}\Big] + \sum_{l\in\calk}\E_\pi\Big[(R_{s,l}^{(r)})^N\mathbbm{1}\{Z_{H(l)}^{(r)}=0\}\Big]\\
		&\leq 
		\sum_{l\in\cale}\frac{\E[T_{e,l}^{N+1}]}{(N+1)\alpha_l^N }+ \sum_{l\in\calk}\frac{\E[T_{s,l}^{N+1}]}{(N+1)(\mu_l^{(r)})^N} + \frac{E[T_{s,l}^N]}{(\mu_l^{(r)})^N}\\
		&\equiv E_0<\infty.
	\end{align*}
	As such, we have completed the proof of the base case of~\eqref{eq:S3} for all $(k,0)$ with $k\in\call$.
	
	\subsubsection{Proof of~\texorpdfstring{\eqref{eq:S4}}{(S4)}}
	
	Next, we prove the base case of~\eqref{eq:S4}. Setting $n=0$, we aim to prove that for all $k\in\call$,
	\begin{equation*}
		\sum_{l\in\cale}\E_{e,l}\Big[\Lambda_{N}\Big(X^{(r)}\Big)\Big]+\sum_{l\in\calk}\E_{s,l}\Big[\Lambda_{N}\Big(X^{(r)}\Big)\Big]\leq F_{k,0}<\infty.
	\end{equation*}
	Similarly, we notice that the LHS of the inequality above is independent of $k$, and therefore, we shall have some $F_0\equiv F_{k,0}$ for all $k\in\call$.
	
	The test function~\eqref{eq:S4-test} with $n=0$ is
	\begin{equation*}
		f_{k,0,F}(X^{(r)})=\Lambda_N\Big(X^{(r)}\Big)\Lambda_1\Big(X^{(r)}\Big).
	\end{equation*}
	We now substitute our test function into BAR and analyze both sides of the equation.
	For the LHS of BAR~\eqref{eq:bar-general}, we have
	\begin{align*}
		-\E_\pi[\cala f(X^{(r)})]&=N\E\Big[\Big(\sum_{l\in\cale}(R_{e,l}^{(r)})^{N-1} + \sum_{l\in\calk}(R_{s,l}^{(r)})^{N-1}\mathbbm{1}\{Z_l^{(r)}>0, Z_{H_+(l)}^{(r)}=0\}\Big)\Lambda_1(X^{(r)})\Big]\\
		&+\E\Big[\Lambda_N(X^{(r)})\Big(E+\sum_{l\in\calk}\mathbbm1\{Z_l^{(r)}>0, Z_{H_+(l)}^{(r)}=0\}\Big)\Big]\\
		&\leq N\E\Big[\Lambda_{N-1}(X^{(r)})\Lambda_1(X^{(r)})\Big] + 2J\E[\Lambda_N(X^{(r)})].
	\end{align*}
	Next, we note that
	\begin{align*}
		&\Lambda_{N-1}(X^{(r)})\Lambda_1(X^{(r)})\\
		&=\Big(\sum_{l\in\cale}(R_{e,l}^{(r)})^{N-1} + \sum_{l\in\calk}(R_{s,l}^{(r)})^{N-1}\mathbbm1\{Z_{H_+(l)}^{(r)}=0\}\Big)\Big(\sum_{l\in\cale}R_{e,l}^{(r)} + \sum_{l\in\calk}R_{s,l}^{(r)}\mathbbm1\{Z_{H_+(l)}^{(r)}=0\}\Big)\\
		&\leq \Big(\sum_{l\in\cale}R_{e,l}^{(r)} + \sum_{l\in\calk}R_{s,l}^{(r)}\mathbbm1\{Z_{H_+(l)}^{(r)}=0\}\Big)^N = \Lambda_1^N(X^{(r)})\\
		&\leq (2J)^{N-1}\Big(\sum_{l\in\cale}(R_{e,l}^{(r)})^N + \sum_{l\in\calk}(R_{s,l}^{(r)})^N\mathbbm1\{Z_{H_+(l)}^{(r)}=0\}\Big)=(2J)^{N-1}\Lambda_N(X^{(r)}).
	\end{align*}
	Hence, we first obtain the following upper bound to the LHS of BAR,
	\begin{equation*}
		-\E_\pi[\cala f(X^{(r)})]\leq (N(2J)^{N-1}+2J)\E_\pi[\Lambda_N(X^{(r)})]\leq (N(2J)^{N-1}+2J)E_0.
	\end{equation*}
	
	Next, we analyze the terms on the RHS of BAR~\eqref{eq:bar-general}. We start by examining the ``Palm'' expectations individually and establish the following lower bounds.
	\begin{align*}
		&\E_{e,l}[f(X^{(r)}+\Delta_{e,l})-f(X^{(r)})]\\
		&=\E_{e,l}\Big[\Big(\Lambda_N(X^{(r)})+(T_{e,l}/\alpha_l)^N\Big)\Big(\Lambda_1(X^{(r)})+T_{e,l}/\alpha_l\Big) - \Lambda_N(X^{(r)})\Lambda_1(X^{(r)})\Big]\\
		&\geq \E_{e,l}\Big[\Lambda_N(X^{(r)})T_{e,l}/\alpha_l\Big]=\E_{e,l}[\Lambda_N(X^{(r)})]/\alpha_l.
	\end{align*}
	Similarly, we have
	\begin{align*}
		&\E_{s,l}[f(X^{(r)}+\Delta_{s,l})-f(X^{(r)})]\\
		&=\E_{s,l}\Big[\Big(\Lambda_N(X^{(r)})+(T_{s,l}/\mu_l^{(r)})^N\Big)\Big(\Lambda_1(X^{(r)})+T_{s,l}/\mu_l^{(r)}\Big) - \Lambda_N(X^{(r)})\Lambda_1(X^{(r)})\Big]\\
		&\geq \E_{s,l}\Big[\Lambda_N(X^{(r)})T_{s,l}/\mu_l^{(r)}\Big]=\E_{s,l}[\Lambda_N(X^{(r)})]/\mu_l^{(r)}.
	\end{align*}
	
	Combining the analyses above, we get
	\begin{align*}
		&\sum_{l\in\cale}\E_{e,l}[\Lambda_N(X^{(r)})]+\sum_{l\in\calk}\frac{\lambda_l}{\mu_l^{(r)}}\E_{s,l}[\Lambda_N(X^{(r)})]\\
		&\leq \sum_{l\in\cale}\alpha_l\E_{e,l}[\Delta f(X_+,X_-)]+\sum_{l\in\calk}\lambda_l\E_{s,l}[\Delta f(X_+,X_-)]=-\E_\pi[\cala f(X)]\\
		&\leq(N(2J)^{N-1}+2J)E_0. 
	\end{align*}
	
	Lastly, we take care of the coefficients on the LHS and conclude that
	\begin{equation*}
		\sum_{l\in\cale}\E_{e,l}[\Lambda_N(X^{(r)})]+\sum_{l\in\calk}\E_{s,l}[\Lambda_N(X^{(r)})]
		\leq\Big(\sup_{r\in(0,r_0)}\sup_{l\in\calk}\frac{\mu_l^{(r)}}{\lambda_l}\Big)(N(2J)^{N-1}+2J)E_0\equiv F_0.
	\end{equation*}
	As such, we have completed the proof of base case of~\eqref{eq:S4} for all $(k,0)$ with $k\in\call$.
	
	\subsection{Induction Step}
	Now that we have established the base case, we proceed to the inductive step. To prove that\eqref{eq:S1}--\eqref{eq:S4} hold for $(k,n)$, we assume that the four technical inequalities all hold for $(k',n')$ with $k'\in\call, n'\leq n$, and for $(k',n)$ for $k'<k$.

	\subsubsection{Proof of~\texorpdfstring{\eqref{eq:S1}}{(S1)}}
	\label{sec:proof-s1-induction}
	We now proceed to prove~\eqref{eq:S1} for general $(k,n)$. 
	In this case, we have the test function
	\begin{equation*}
		f_{k,n}(X^{(r)})=\frac{1}{n+1}r^{k(n-1)}\Big(\dotp{u^{(k)},Z^{(r)}}\Big)^{n+1} +r^{k(n-1)}\Big(\dotp{u^{(k)},Z^{(r)}}\Big)^nh_k\Big(X^{(r)}\Big),
	\end{equation*}
	where $u^{(k)}$ is as defined in~\eqref{eq:u-vec-def} and
	\begin{equation*}
		h_k(X^{(r)})=-\sum_{l\in\cale}\alpha_lu_l^{(k)}R_{e,l}^{(r)}+\sum_{l\in\calk}\Big(u_l^{(k)}-\sum_{l'\in\calk}P_{ll'}u_{l'}^{(k)}\Big)\mu_l^{(r)}R_{s,l}^{(r)}\mathbbm1\{Z_{H_+(l)}=0\}.
	\end{equation*}
	
	Before substituting the test function into the BAR for analysis, we first note the following special properties that result from both our choice of $u^{(k)}$ and the induction hypothesis.
	With $u^{(k)}$ as defined in~\eqref{eq:u-vec-def}, we have
	\begin{equation*}
		\dotp{u^{(k)},Z^{(r)}}=\sum_{l<k}w_{lk}Z_l^{(r)} + Z_k^{(r)} + \dotp{u_H^{(r)},Z_H^{(r)}}.
	\end{equation*}
	By induction hypotheses, we have
	\begin{equation}
		\label{eq:induction-u-property}
		\E_\pi\Big[r^{n(k-1)}\sum_{l<k}\big|w_{lk}Z_l^{(r)}\big|^n\Big]<\infty.
	\end{equation}
	By Assumption~\ref{assumption:high-moment-general}, 
	\begin{equation}
		\label{eq:induction-high-property}\E_\pi\Big[\big|Z_H^{(r)}\big|\Big]<\infty.
	\end{equation}

	Now, we substitute the test function into the BAR~\eqref{eq:bar-general}. For the LHS, we have
	\begin{equation*}
		-\cala f(X^{(r)})=r^{k(n-1)}\Big(\dotp{u^{(k)},Z^{(r)}}\Big)^n\Big(-\cala h_k(X^{(r)})\Big).
	\end{equation*}
	Note that
	\begin{align*}
		-\cala h_k(X^{(r)})&=-\sum_{l\in\cale}\alpha_lu_l^{(k)}+\sum_{l\in\calk}\Big(u_l^{(k)}-\sum_{l'\in\calk}P_{ll'}u_{l'}^{(k)}\Big)\mu_l^{(r)}\mathbbm1\{Z_l^{(r)}>0, Z_{H_+(l)}^{(r)}=0\}\\
		&=-\sum_{l\in\cale}\alpha_lu_l^{(k)}+\sum_{j\in\calj}\Big(u_j^{(k)}-\sum_{l\in\calk}P_{jl}u_l^{(k)}\Big)\mu_j^{(r)}\sum_{l\in\calc(j)}\mathbbm1\{Z_l^{(r)}>0, Z_{H_+(l)}^{(r)}=0\}\\
		&=-\sum_{l\in\cale}\alpha_lu_l^{(k)}+\sum_{j\in\calj}\Big(u_j^{(k)}-\sum_{l\in\calk}P_{jl}u_l^{(k)}\Big)\mu_j^{(r)}(1-\mathbbm1\{Z_{H(j)}^{(r)}=0\}),
	\end{align*}
	where we make use of Lemma~\ref{lem:station-property} to obtain the second equality. We further make use of the traffic equation, that $\alpha_l=\lambda_l-\sum_{l'\in\calk}P_{l'l}\lambda_{l'},$
	and we obtain
	\begin{align*}
		\sum_{l\in\cale}\alpha_lu_l^{(k)}&=\sum_{l\in\calk}u_l^{(k)}\Big(\lambda_l-\sum_{l'\in\calk}P_{l'l}\lambda_{l'}\Big)\\
		&=\sum_{l\in\calk}u_l^{(k)}\lambda_l-\sum_{l'\in\calk}\lambda_{l'}\sum_{l\in\calk}P_{l'l}u_l^{(k)}
		=\sum_{l\in\calk}\lambda_l\Big(u_l^{(k)}-\sum_{l'\in\calk}P_{ll'}u_{l'}^{(k)}\Big)\\
		&=\sum_{l\in\calk}\frac{\lambda_l}{\mu_l^{(r)}}\Big(u_l^{(k)}-\sum_{l'\in\calk}P_{ll'}u_{l'}^{(k)}\Big)\mu_l^{(r)}\\
		&\overset{\text{(i)}}{=}\sum_{j\in\calj}\Big(u_j^{(k)}-\sum_{l\in\calk}P_{jl}u_{l}^{(k)}\Big)\mu_j^{(r)}\sum_{l\in\calc(j)}\frac{\lambda_l}{\mu_l^{(r)}}\\
		&=\sum_{j\in\calj}\Big(u_j^{(k)}-\sum_{l\in\calk}P_{jl}u_{l}^{(k)}\Big)\mu_j^{(r)}\rho_j^{(r)},
	\end{align*}
	where in (i) we use Lemma~\ref{lem:station-property} again to obtain the equality.
	Therefore, the generator on the LHS BAR~\eqref{eq:bar-general} can be simplified to the following
	\begin{align*}
		&-\cala h_k(X^{(r)})=\sum_{j\in\calj}\Big(u_j^{(k)}-\sum_{l\in\calk}P_{jl}u_{l}^{(k)}\Big)\mu_j^{(r)}\Big(-\rho_j^{(r)}+1-\mathbbm1\{z_{H(j)}=0\}\Big)\\
		&=(1-w_{kk})\mu_k^{(r)}\Big(-\rho_k^{(r)}+1-\mathbbm1\{Z_{H(k)}^{(r)}=0\}\Big)-\sum_{j=k+1}^Jw_{jk}\mu_j^{(r)}\Big(-\rho_j^{(r)}+1-\mathbbm1\{Z_{H(j)}^{(r)}=0\}\Big)\\
		&=(1-w_{kk})\mu_k^{(r)}(r^k-\mathbbm1\{Z_{H(k)}^{(r)}=0\})-\sum_{j=k+1}^Jw_{jk}\mu_j^{(r)}\Big(r^j-\mathbbm1\{Z_{H(j)}^{(r)}=0\}\Big).
	\end{align*}
	Subsequently, together with expectation, we have
	\begin{align*}
		&-\E_\pi[\cala f(X^{(r)})]\\
		&=\E_\pi\Big[r^{k(n-1)}\Big(\dotp{u^{(k)},Z^{(r)}}\Big)^n\\
		&\qquad\qquad\quad~~\Big((1-w_{kk})\mu_k^{(r)}(r^k-\mathbbm1\{Z_{H(k)}^{(r)}=0\})-\sum_{j=k+1}^Jw_{jk}\mu_j^{(r)}\Big(r^j-\mathbbm1\{Z_{H(j)}^{(r)}=0\}\Big)\Big)\Big]\\
		&=(1-w_{kk})\mu_k^{(r)}\E_\pi\Big[\Big(\dotp{u^{(k)},r^kZ^{(r)}}\Big)^n\Big]\\
		&- (1-w_{kk})\mu_k^{(r)}\underbrace{\E_\pi\Big[r^{k(n-1)}\Big(\dotp{u^{(k)},Z^{(r)}}\Big)^n\mathbbm1\{Z_{H(k)}^{(r)}=0\}\Big]}_{T_1}\\
		&-\sum_{j=k+1}^Jw_{jk}\mu_j^{(r)}r^{j-k}\E_\pi\Big[\Big(\dotp{u^{(k)},r^kZ^{(r)}}\Big)^n\Big]\\
		&+ \underbrace{\sum_{j=k+1}^Jw_{jk}\mu_j^{(r)}\E_\pi\Big[r^{k(n-1)}\Big(\dotp{u^{(k)},Z^{(r)}}\Big)^n\mathbbm1\{Z_{H(j)}^{(r)}=0\}\Big]}_{T_2}
	\end{align*}
	
	It is important to analyze the two cross terms. Starting with $T_1$, we have
	\begin{equation*}
		\E_\pi\Big[r^{k(n-1)}\Big(\dotp{u^{(k)},Z^{(r)}}\Big)^n\mathbbm1\{Z_{H(k)}^{(r)}=0\}\Big]=r^{n-k}\E_\pi\Big[\Big(r^{k-1}\dotp{u^{(k)},Z^{(r)}}\Big)^n\mathbbm1\{Z_{H(k)}^{(r)}=0\}\Big].
	\end{equation*}
	If $n\geq k$, then it is easy to see that
	\begin{align*}
		&r^{n-k}\E_\pi\Big[\Big(r^{k-1}\dotp{u^{(k)},Z^{(r)}}\Big)^n\mathbbm1\{Z_{H(k)}^{(r)}=0\}\Big]\\
		&=r^{n-k}\E_\pi\Big[\Big(r^{k-1}(\sum_{l<k}u_l^{(k)}Z_l^{(r)}+\dotp{u_H^{(k)},Z_H^{(r)}})\Big)^n\mathbbm1\{Z_{H(k)}^{(r)}=0\}\Big]\\
		&\leq r^{n-k}\E_\pi\Big[\Big(r^{k-1}(\sum_{l<k}u_l^{(k)}Z_l^{(r)}+\dotp{u_H^{(k)},Z_H^{(r)}})\Big)^n\Big]
	\end{align*}
	By induction hypotheses, this term can be controlled and is of a higher order w.r.t.\ $r$
	Otherwise, if $n<k$, we apply the H\"{o}lder's inequality with $p=k/n$ and $q=k/(k-n)$, and obtain
	\begin{align*}
		&r^{n-k}\E_\pi\Big[\Big(r^{k-1}\dotp{u^{(k)},Z^{(r)}}\Big)^n\mathbbm1\{Z_{H(k)}^{(r)}=0\}\Big]\\
		&\leq r^{n-k}\Big(\E_\pi\Big[\Big(r^{k-1}(\sum_{l<k}u_l^{(k)}Z_l^{(r)}+\dotp{u_H^{(k)},Z_H^{(r)}})\Big)^{np}\Big]\Big)^{1/p}\Big(\E[\mathbbm1\{Z_{H(k)}^{(r)}=0\}]\Big)^{1/q}\\
		&=\Big(\E_\pi\Big[\Big(r^{k-1}(\sum_{l<k}u_l^{(k)}Z_l^{(r)}+\dotp{u_H^{(k)},Z_H^{(r)}})\Big)^{k}\Big]\Big)^{1/p},
	\end{align*}
	which is finite as we have previously demonstrated in~\eqref{eq:induction-u-property} and~\eqref{eq:induction-high-property}.

	The other cross-term $T_2$ can be analyzed in a similar fashion, so we omit it here.
	
	Now, we proceed to analyze the RHS of BAR~\eqref{eq:bar-general}. Starting with the ``Palm'' expectations w.r.t.\ external arrivals, we have
	\begin{align}
		&\E_{e,l}[f(X^{(r)}+\Delta_{e,l})-f(X^{(r)})]\nonumber\\
		&=r^{k(n-1)}\E_{e,l}\Big[\frac{1}{n+1}\Big(\dotp{u^{(k)},Z^{(r)}+e^{(l)}}^{n+1}-\dotp{u^{(k)},Z^{(r)}}^{n+1}\Big)-\dotp{u^{(k)},Z^{(r)}+e^{(l)}}^nu_lT_{e,l}\Big]\label{eq:external-t1}\\
		&+r^{k(n-1)}\E_{e,l}\Big[\Big(\dotp{u^{(k)},Z^{(r)}+e^{(l)}}^{n}-\dotp{u^{(k)},Z^{(r)}}^{n}\Big)h_k(X^{(r)})\Big]\label{eq:external-t2}
	\end{align}
	
	Starting with the analysis of~\eqref{eq:external-t1}, we make use of the mean value theorem, that $\exists \omega_{l,1}, \omega_{l,2}\in(0,1)$, such that
	\begin{align*}
		&r^{k(n-1)}\E_{e,l}\Big[\frac{1}{n+1}\Big(\dotp{u^{(k)},Z^{(r)}+e^{(l)}}^{n+1}-\dotp{u^{(k)},Z^{(r)}}^{n+1}\Big)-\dotp{u^{(k)},Z^{(r)}+e^{(l)}}^nu_lT_{e,l}\Big]\\
		&=r^{k(n-1)}\E_{e,l}\Big[u_l\Big(\dotp{u^{(k)},Z^{(r)}+\omega_{l,1}e^{(l)}}^n -\dotp{u^{(k)},Z^{(r)}+e^{(l)}}^n\Big)\Big]\\
		&=-\omega_{l,1}u_l^2n\E_{e,l}\Big[\dotp{u^{(k)},r^k(Z^{(r)}-\omega_{l,2}e^{(l)})}^{n-1}\Big],
	\end{align*}
	where the last term is bounded by induction hypotheses that $\E_{e,l}[(r^kZ_k^{(r)})^{n-1}]<\infty$
	
	For the second term~\eqref{eq:external-t2}, we again apply the mean value theorem, and obtain that $\exists\omega_{l,3}\in(0,1)$, such that
	\begin{align*}
		&r^{k(n-1)}\E_{e,l}\Big[\Big(\dotp{u^{(k)},Z^{(r)}+e^{(l)}}^{n}-\dotp{u^{(k)},Z^{(r)}}^{n}\Big)h_k(X^{(r)})\Big]\\
		&=nu_l\E_{e,l}\Big[\dotp{u^{(k)},r^k(Z^{(r)}+\omega_{l,3}e^{(l)})}^{n-1}h_k(X^{(r)})\Big],
	\end{align*}
	where the last term is bounded by induction hypotheses~\eqref{eq:S4}.
	
	The ``Palm'' expectations w.r.t.\ service completions can be analyzed similarly.
	\begin{align}
		&\E_{s,l}[f(X^{(r)}+\Delta_{s,l})-f(X^{(r)})]\nonumber\\
		&=r^{k(n-1)}\E_{s,l}\Big[\frac{1}{n+1}\Big(\dotp{u^{(k)},Z^{(r)}-e^{(l)}+\xi^{(l)}}^{n+1}-\dotp{u^{(k)},Z^{(r)}}^{n+1}\Big)\label{eq:service-t1a}\\
		&\qquad\qquad\qquad\qquad\qquad-\dotp{u^{(k)},Z^{(r)}-e^{(l)}+\xi^{(l)}}^n\Big(u_l^{(k)}-\sum_{l'\in\calk}P_{ll'}u_{l'}^{(k)}\Big)T_{s,l}\Big]\label{eq:service-t1b}\\
		&+r^{k(n-1)}\E_{s,l}\Big[\Big(\dotp{u^{(k)},Z^{(r)}-e^{(l)}+\xi^{(l)}}^{n}-\dotp{u^{(k)},Z^{(r)}}^{n}\Big)h_k(X^{(r)})\Big].\label{eq:service-t2}
	\end{align}
	
	Before term-by-term analysis, we first note that $-e^{(l)}+\xi^{(l)}$ is independent of $Z^{(r)}$ and
	\begin{equation}
		\Big|\dotp{u^{(k)},-e^{(l)}+\xi^{(l)}}\Big|\leq\|u^{(k)}\|_1, 
		\label{eq:u-delta-norm}
	\end{equation}
	and
	\begin{equation*}
		\E\Big[\dotp{u^{(k)},-e^{(l)}+\xi^{(l)}}\Big]=-u_l^{(k)}+\sum_{l'\in\calk}P_{ll'}u^{(k)}_{l'}.
	\end{equation*}
	
	Now, proceed to analyze the first term in~\eqref{eq:service-t1a}--\eqref{eq:service-t1b}. Making use of the mean value theorem, we obtain that $\exists\omega_{l,4},\omega_{l,5},\omega_{l,6}\in(0,1)$, such that
	\begin{align*}
		&r^{k(n-1)}\E_{s,l}\Big[\frac{1}{n+1}\Big(\dotp{u^{(k)},Z^{(r)}-e^{(l)}+\xi^{(l)}}^{n+1}-\dotp{u^{(k)},Z^{(r)}}^{n+1}\Big)\\
		&\qquad\qquad\qquad\qquad-\dotp{u^{(k)},Z^{(r)}-e^{(l)}+\xi^{(l)}}^n\Big(u_l^{(k)}-\sum_{l'\in\calk}P_{ll'}u_{l'}^{(k)}\Big)T_{s,l}\Big]\\
		&=r^{k(n-1)}\E_{s,l}\Big[\dotp{u^{(k)},-e^{(l)}+\xi^{(l)}}\dotp{u^{(k)},Z^{(r)}+\omega_{l,4}(-e^{(l)}+\xi^{(l)})}^{n}\\
		&\qquad\qquad\qquad\qquad-\dotp{u^{(k)},Z^{(r)}-e^{(l)}+\xi^{(l)}}^n\Big(u_l^{(k)}-\sum_{l'\in\calk}P_{ll'}u_{l'}^{(k)}\Big)\Big]\\
		&=r^{k(n-1)}\E_{s,l}\Big[\dotp{u^{(k)},-e^{(l)}+\xi^{(l)}}\Big(\dotp{u^{(k)},Z^{(r)}+\omega_{l,4}(-e^{(l)}+\xi^{(l)})}^{n}-\dotp{u^{(k)},Z^{(r)}}^n\Big)\Big]\\
		&+r^{k(n-1)}\E_{s,l}\Big[\dotp{u^{(k)},-e^{(l)}+\xi^{(l)}}\dotp{u^{(k)},Z^{(r)}}^n-\dotp{u^{(k)},Z^{(r)}-e^{(l)}+\xi^{(l)}}^n\Big(u_l^{(k)}-\sum_{l'\in\calk}P_{ll'}u_{l'}^{(k)}\Big)\Big]\\
		&=n\omega_{l,4}\E_{s,l}\Big[\dotp{u^{(k)},-e^{(l)}+\xi^{(l)}}^2\dotp{u^{(k)},r^k(Z^{(r)}+\omega_{l,5}(-e^{(l)}+\xi^{(l)}))}^{n-1}\Big]\\
		&+n\Big(u_l^{(k)}-\sum_{l'\in\calk}P_{ll'}u_{l'}^{(k)}\Big)\E_{s,l}\Big[\dotp{u^{(k)},-e^{(l)}+\xi^{(l)}}\dotp{u^{(k)},r^k(Z^{(r)}-\omega_{l,6}(-e^{(l)}+\xi^{(l)}))}^{n-1}\Big].
	\end{align*}
	Thus, making use of~\eqref{eq:u-delta-norm} and induction hypotheses $\E[(r^kZ_k^{(r)})^{n-1}]<\infty$, we can control the final two terms in the equation above.
	
	Consolidating the analyses above, we get
	\begin{align}
		&(1-w_{kk})\mu_k^{(r)}\E_\pi\Big[\dotp{u^{(k)},r^kZ^{(r)}}^n\Big]\label{eq:s1-final-lhs}\\
		&\leq \sum_{j=k+1}^Jw_{jk}\mu_j^{(r)}r^{j-k}\E_\pi\Big[\dotp{u^{(k)},r^kZ^{(r)}}^n\Big]\label{eq:s1-final-rhs-t1}\\
		&+(1-w_{kk})\mu_k^{(r)}\E_\pi\Big[r^{k(n-1)}\dotp{u^{(k)},Z^{(r)}}^n\mathbbm1\{Z_{H(k)}^{(r)}=0\}\Big]\nonumber \\
		&-\sum_{j=k+1}^Jw_{jk}\mu_j^{(r)}\E_\pi\Big[r^{k(n-1)}\dotp{u^{(k)},Z^{(r)}}^n\mathbbm1\{Z_{H(j)}^{(r)}=0\}\Big]\nonumber\\
		&+\sum_{l\in\cale}\E_{e,l}[f(X^{(r)}+\Delta_{e,l})-f(X^{(r)})]+\sum_{l\in\calk}\E_{s,l}[f(X^{(r)}+\Delta_{s,l})-f(X^{(r)})].\nonumber
	\end{align}
	Under Assumption~\ref{assumption:ah-invertible} and~\ref{assumption:q-m-matrix}, we have $1-w_{kk}>0$ in~\eqref{eq:s1-final-lhs}, which ensures the sign of the inequality would not flip. The term~\eqref{eq:s1-final-rhs-t1} is of a higher order of $r$, and thus can be easily controlled when $r$ is taken sufficiently small.
	
	As such, we have proven the desired upper bound,
	\begin{equation*}
		\E_\pi\Big[\Big(r^kZ_k^{(r)}\Big)^n\Big]\leq C_{k,n}.
	\end{equation*}

	\subsubsection{Proof of~\texorpdfstring{\eqref{eq:S2}}{(S2)}}
	\label{sec:proof-s2-induction}
	
	Now we proceed to prove~\eqref{eq:S2}, in which we employ the test function
	\begin{equation*}
		f_{k,n,D}(X^{(r)})=\Big(r^kZ_k^{(r)}\Big)^n\Lambda_1\Big(X^{(r)}\Big).
	\end{equation*}
	
	We substitute the test function into the BAR~\eqref{eq:bar-general} and analyze the two sides of the equation. Starting with the LHS of BAR~\eqref{eq:bar-general}, we have
	\begin{align*}
		-\E_\pi[\cala f(X^{(r)})]&=\E_\pi\Big[\Big(r^kZ_k^{(r)}\Big)^n\Big(-\cala\Lambda_1\Big(X^{(r)}\Big)\Big)\Big]\\
		&=\E_\pi\Big[\Big(r^kZ_k^{(r)}\Big)^n\Big(E+\sum_{l\in\calk}\mathbbm1\{Z_l^{(r)}>0, Z_{H_+(l)}^{(r)}=0\}\Big)\Big]\\
		&\leq 2J\E_\pi\Big[\Big(r^kZ_k^{(r)}\Big)^n\Big]\leq 2J C_{k,n},
	\end{align*}
	where the last inequality holds for what we have just shown above in Appendix~\ref{sec:proof-s1-induction}.
	
	Next, we analyze the ``Palm'' expectations on the RHS of BAR. For the ``Palm'' expectation w.r.t.\ external arrival, we have the following lower bound,
	\begin{align*}
		&\E_{e,l}[f(X^{(r)}+\Delta_{e,l})-f(X^{(r)})]\\
		&=\E_{e,l}\Big[\Big(r^k(Z_k^{(r)}+e_k^{(l)})\Big)^n\Big(\Lambda_1\Big(X^{(r)}\Big) +T_{e,l}/\alpha_l\Big) - \Big(r^kZ_k^{(r)}\Big)^n\Lambda_1\Big(X^{(r)}\Big)\Big]
		\geq\E_{e,l}\Big[\Big(r^kZ_k^{(r)}\Big)^n\Big]/\alpha_l.
	\end{align*}
	
	Similarly, for the ``Palm'' expectation w.r.t.\ service completion, we have
	\begin{align*}
		&\E_{s,l}[f(X^{(r)}+\Delta_{s,l})-f(X^{(r)})]\\
		&=\E_{s,l}\Big[\Big(r^k(Z_k^{(r)}-e_k^{(l)}+\xi^{(l)}_k)\Big)^n\Big(\Lambda_1\Big(X^{(r)}\Big) +T_{s,l}/\mu_l^{(r)}\Big) - \Big(r^kZ_k^{(r)}\Big)^n\Lambda_1\Big(X^{(r)}\Big)\Big]\\
		&=\E_{s,l}\Big[\Big(\Big(r^k(Z_k^{(r)}-e_k^{(l)}+\xi^{(l)}_k)\Big)^n-\Big(r^kZ_k^{(r)}\Big)^n\Big)\Big(\Lambda_1\Big(X^{(r)}\Big) +1/\mu_l^{(r)}\Big) + \Big(r^kZ_k^{(r)}\Big)^n/\mu_l^{(r)}\Big],
	\end{align*}
	where $\xi^{(l)}_k$ denotes the $k$-th coordinate of the departure random variable $\xi^{(l)}$.
	Applying the mean value theorem, we obtain that $\exists\omega_l\in(0,1)$, such that
	\begin{align}
		&\E_{s,l}\Big[\Big(\Big(r^k(Z_k^{(r)}-e_k^{(l)}+\xi^{(l)}_k)\Big)^n-\Big(r^kZ_k^{(r)}\Big)^n\Big)\Big(\Lambda_1\Big(X^{(r)}\Big) +1/\mu_l^{(r)}\Big)\Big]\\
		&=nr^k\E_{s,l}\Big[(-e^{(k)}_k+\xi^{(l)}_k)\Big(r^k(Z_k^{(r)}+\omega_l(-e_k^{(l)}+\xi^{(l)}_k))\Big)^{n-1}\Big(\Lambda_1\Big(X^{(r)}\Big) +1/\mu_l^{(r)}\Big)\Big],
	\end{align}
	which is finite by the induction hypothesis that $\E_{s,l}[(r^kZ_k^{(r)})^{n-1}]<\infty$.
	
	Therefore, from the analyses above, we have that
	\begin{align*}
		&\sum_{l\in\cale}\E_{e,l}\Big[\Big(r^kZ_k^{(r)}\Big)^n\Big]/\alpha_l + \sum_{l\in\calk}\E_{s,l}\Big[\Big(r^kZ_k^{(r)}\Big)^n\Big]/\mu_l^{(r)}\\
		&\leq 2JC_{k,n} -\sum_{l\in\calk}nr^k\E_{e,l}\Big[(-e^{(k)}_k+\xi^{(l)}_k)\Big(r^k(Z_k^{(r)}+\omega_l(-e_k^{(l)}+\xi^{(l)}_k))\Big)^{n-1}\Big(\Lambda_1\Big(X^{(r)}\Big) +1/\mu_l^{(r)}\Big)\Big].
	\end{align*}
	
	Lastly, taking care of the coefficient on the LHS of the inequality above, we are able to conclude the desired upper bound in~\eqref{eq:S2},
	\begin{align*}
		&\sum_{l\in\cale}\E_{e,l}\Big[\Big(r^kZ_k^{(r)}\Big)^n\Big] + \sum_{l\in\calk}\E_{s,l}\Big[\Big(r^kZ_k^{(r)}\Big)^n\Big]\leq \Big(\sup_{l\in\calk}\sup_{r\in(0,r_0)}\mu_l^{(r)}\Big)\\
		&\times\Big(2JC_{k,n}-\sum_{l\in\calk}nr^k\E_{e,l}\Big[(-e^{(k)}_k+\xi^{(l)}_k)\Big(r^k(Z_k^{(r)}+\omega_l(-e_k^{(l)}+\xi^{(l)}_k))\Big)^{n-1}\Big(\Lambda_1\Big(X^{(r)}\Big) +1/\mu_l^{(r)}\Big)\Big]\Big)\\
		&\leq D_{k,n}.
	\end{align*}
	
	\subsubsection{Proof of~\texorpdfstring{\eqref{eq:S3}}{(S3)}}
	\label{sec:proof-s3-induction}
	In this section, we prove~\eqref{eq:S3} for the pair $(k,n)$ of interest with test function
	\begin{equation*}
		f_{k,n,E}(X^{(r)})=\Big(r^kZ_k^{(r)}\Big)^n\Lambda_{N-n+1}\Big(X^{(r)}\Big).
	\end{equation*}
	
	Plugging in the test function into the BAR~\eqref{eq:bar-general}, for the LHS, we have
	\begin{align*}
		&-\E_\pi[\cala f(X^{(r)})]=\E_\pi\Big[\Big(r^kZ_k^{(r)}\Big)^n\Big(-\cala\Lambda_{N-n+1}\Big(X^{(r)}\Big)\Big)\Big]\\
		&=(N-n+1)\E_\pi\Big[\Big(r^kZ_k^{(r)}\Big)^n\Big(\sum_{l\in\cale}(R_{e,l}^{(r)})^{N-n} + \sum_{l\in\calk}(R_{s,l}^{(r)})^{N-n}\mathbbm1\{Z_l^{(r)}>0, Z_{H_+(l)}^{(r)}=0\}\Big)\Big].
	\end{align*}
	
	For the ``Palm'' expectations on the RHS, we have
	\begin{align*}
		&\E_{e,l}[f(X^{(r)}+\Delta_{e,l})-f(X^{(r)})]\\
		&=\E_{e,l}\Big[\Big(r^k(Z_k^{(r)}+e_k^{(l)})\Big)^n\Big(\Lambda_{N-n+1}\Big(X^{(r)}\Big) +(T_{e,l}/\alpha_l)^{N-n+1}\Big) - \Big(r^kZ_k^{(r)}\Big)^n\Lambda_{N-n+1}\Big(X^{(r)}\Big)\Big]\\
		&=\E_{e,l}\Big[\Big(\Big(r^k(Z_k^{(r)}+e_k^{(l)})\Big)^n-\Big(r^kZ_k^{(r)}\Big)^n\Big)\Big(\Lambda_{N-n+1}\Big(X^{(r)}\Big) +(T_{e,l}/\alpha_l)^{N-n+1}\Big) \\
		&\qquad\qquad+ \Big(r^kZ_k^{(r)}\Big)^n(T_{e,l}/\alpha_l)^{N-n+1}\Big]
	\end{align*}
	By the mean value theorem, we know that $\exists\omega_l\in(0,1)$, such that
	\begin{align*}
		&\E_{e,l}\Big[\Big(\Big(r^k(Z_k^{(r)}+e_k^{(l)})\Big)^n-\Big(r^kZ_k^{(r)}\Big)^n\Big)\Big(\Lambda_{N-n+1}\Big(X^{(r)}\Big) +(T_{e,l}/\alpha_l)^{N-n+1}\Big)\Big]\\
		&=r^{k}\E_{e,l}\Big[(e_k^{(l)})^n\Big(r^k(Z_k^{(r)}+\omega_le_k^{(l)})\Big)^{n-1}\Big(\Lambda_{N-n+1}\Big(X^{(r)}\Big) +\frac{\E[T_{e,l}^{N-n+1}]}{\alpha_l^{N-n+1}}\Big)\Big].
	\end{align*}
	Applying the induction hypotheses, that $\E[(r^kZ_k^{(r)})^{n-1}\Lambda_{N-n+1}(X^{(r)})]\leq E_{k,n-1}$, we can bound the above term.
	
	Furthermore, by independence, we have
	\begin{equation*}
		\E_{e,l}\Big[\Big(r^kZ_k^{(r)}\Big)^n(T_{e,l}/\alpha_l)^{N-n+1}\Big]=\E_{e,l}\Big[\Big(r^kZ_k^{(r)}\Big)^n\Big]\frac{\E[T_{e,l}^{N-n+1}]}{\alpha_l^{N-n+1}}\leq D_{k,n}\frac{\E[T_{e,l}^{N-n+1}]}{\alpha_l^{N-n+1}},
	\end{equation*}
	where the last inequality holds for we have just proven in Appendix~\ref{sec:proof-s2-induction}.
	
	Next, we examine the ``Palm'' expectation w.r.t.\ service completion.
	\begin{align*}
		&\E_{s,l}[f(X^{(r)}+\Delta_{s,l})-f(X^{(r)})]\\
		&=\E_{s,l}\Big[\Big(r^k(Z_k^{(r)}-e_k^{(l)}+\xi^{(l)}_k)\Big)^n\Big(\Lambda_{N-n+1}\Big(X^{(r)}\Big) +(T_{s,l}/\mu_l^{(r)})^{N-n+1}\Big) \\
		&\qquad\qquad- \Big(r^kZ_k^{(r)}\Big)^n\Lambda_{N-n+1}\Big(X^{(r)}\Big)\Big]\\
		&=\E_{s,l}\Big[\Big(\Big(r^k(Z_k^{(r)}-e_k^{(l)}+\xi^{(l)}_k)\Big)^n-\Big(r^kZ_k^{(r)}\Big)^n\Big)\Big(\Lambda_{N-n+1}\Big(X^{(r)}\Big) +(T_{s,l}/\mu_l^{(\alpha)})^{N-n+1}\Big) \\
		&\qquad\qquad+ \Big(r^kZ_k^{(r)}\Big)^n(T_{s,l}/\mu_l^{(\alpha)})^{N-n+1}\Big].
	\end{align*}
	Applying the mean value theorem and induction hypotheses, we can similarly bound the above term.
	
	Consolidating the analyses of both sides of BAR, we get
	\begin{align*}
		&\E_\pi\Big[\Big(r^kZ_k^{(r)}\Big)^n\Big(\sum_{l\in\cale}(R_{e,l}^{(r)})^{N-n} + \sum_{l\in\calk}(R_{s,l}^{(r)})^{N-n}\mathbbm1\{Z_l^{(r)}>0, Z_{H_+(l)}^{(r)}=0\}\Big)\Big]\\
		&\leq \frac{1}{N-n+1}\Big(\sum_{l\in\cale}\alpha_l\E_{e,l}[f(X^{(r)}+\Delta_{e,l})-f(X^{(r)})]+\sum_{l\in\calk}\lambda_l\E_{s,l}[f(X^{(r)}+\Delta_{s,l})-f(X^{(r)})]\Big).
	\end{align*}
	
	Note that
	\begin{align}
		&\E_\pi\Big[\Big(r^kZ_k^{(r)}\Big)^n\Lambda_{N-n}(X^{(r)})\Big]\nonumber\\
		&=\E_\pi\Big[\Big(r^kZ_k^{(r)}\Big)^n\Big(\sum_{l\in\cale}(R_{e,l}^{(r)})^{N-n} + \sum_{l\in\calk}(R_{s,l}^{(r)})^{N-n}\mathbbm1\{Z_l^{(r)}>0, Z_{H_+(l)}^{(r)}=0\}\Big)\Big]\nonumber\\
		&+\E_\pi\Big[\Big(r^kZ_k^{(r)}\Big)^n\sum_{l\in\calk}(R_{s,l}^{(r)})^{N-n}\mathbbm1\{Z_{H(l)}^{(r)}=0\}\Big]\label{eq:s3-residual},
	\end{align}
	Therefore, it remains to bound the residual term in~\eqref{eq:s3-residual}.
	Recall that $R_{s,l}=T_{s,l}/\mu_l^{(r)}$ when $Z_{H(l)}^{(r)}=0$, and thus we have
	\begin{align*}
		&\E_\pi\Big[\Big(r^kZ_k^{(r)}\Big)^n\sum_{l\in\calk}(R_{s,l}^{(r)})^{N-n}\mathbbm1\{Z_{H(l)}^{(r)}=0\}\Big]\\
		&=\E_\pi\Big[\Big(r^kZ_k^{(r)}\Big)^n\sum_{l\in\calk}(T_{s,l}/\mu_l^{(r)})^{N-n}\mathbbm1\{Z_{H(l)}^{(r)}=0\}\Big]\\
		&\leq\E_\pi\Big[\Big(r^kZ_k^{(r)}\Big)^n\Big]\sum_{l\in\calk}\frac{\E[T_{s,l}^{N-n}]}{(\mu_l^{(r)})^{N-n}},
	\end{align*}
	where the final inequality can be further bound by the induction hypothesis proved in Section~\ref{sec:proof-s1-induction}

	As such, we have proven the desired upper bound,
	\begin{equation*}
		\E_\pi\Big[\Big(r^kZ_k^{(r)}\Big)^n\Lambda_{N-n}\Big(X^{(r)}\Big)\Big]\leq E_{k,n}.
	\end{equation*}
	
	\subsubsection{Proof of~\texorpdfstring{\eqref{eq:S4}}{(S4)}}
	
	Lastly, we prove~\eqref{eq:S4} in this induction step. As discussed above, we substitute the following test function into the BAR~\eqref{eq:bar-general}.
	\begin{equation*}
		f_{k,n,F}(X^{(r)})=\Big(r^kZ_k^{(r)}\Big)^n\Lambda_{N-n}\Big(X^{(r)}\Big)\Lambda_1\Big(X^{(r)}\Big).
	\end{equation*}
	
	For the LHS of BAR~\eqref{eq:bar-general}, we have
	\begin{align*}
		&-\E_\pi[\cala f(X^{(r)})]\\
		&=\E_\pi\Big[(r^kZ_k^{(r)})^n\Lambda_{N-n}(X^{(r)})\Big(E+\sum_{l\in\calk}\mathbbm1\{Z_l^{(r)}>0,Z_{H_+(l)}^{(r)}=0\}\Big)\Big]\\
		&+(N-n)\E_\pi\Big[(r^kZ_k^{(r)})^n\Big(\sum_{l\in\cale}(R_{e,l}^{(r)})^{N-n-1}+\sum_{l\in\calk}(R_{s,l}^{(r)})^{N-n-1}\mathbbm1\{Z_l^{(r)}>0,Z_{H_+(l)}^{(r)}=0\}\Big)\Lambda_1(X^{(r)})\Big]\\
		&\leq 2J\E_\pi\Big[(r^kZ_k^{(r)})^n\Lambda_{N-n}(X^{(r)})\Big]+(N-n)\E_\pi\Big[(r^kZ_k^{(r)})^n\Lambda_{N-n-1}(X^{(r)})\Lambda_1(X^{(r)})\Big].
	\end{align*}
	Note that
	\begin{equation*}
		\Lambda_{N-n-1}(X^{(r)})\Lambda_1(X^{(r)})\leq\Lambda_1^{N-n}(X^{(r)})\leq (2J)^{N-n-1}\Lambda_{N-n}(X^{(r)}).
	\end{equation*}
	Hence, we have
	\begin{align*}
		&-\E_\pi[\cala f(X^{(r)})]\\
		&\leq 2J\E_\pi\Big[(r^kZ_k^{(r)})^n\Lambda_{N-n}(X^{(r)})\Big]+(N-n)(2J)^{N-n-1}\E_\pi\Big[(r^kZ_k^{(r)})^n\Lambda_{N-n}(X^{(r)})\Big]\\
		&\leq \Big(2J + (N-n)(2J)^{N-n-1}\Big)\E_\pi\Big[(r^kZ_k^{(r)})^n\Lambda_{N-n}(X^{(r)})\Big]\\
		&=\Big(2J + (N-n)(2J)^{N-n-1}\Big)E_{k,n},
	\end{align*}
	where the last inequality holds for the induction hypothesis proved in~\ref{sec:proof-s3-induction}.
	
	It remains for us to analyze the ``Palm'' expectations on the RHS of BAR~\eqref{eq:bar-general}. For the ``Palm'' expectation w.r.t.\ external arrivals, we derive the following lower bound,
	\begin{align*}
		&\E_{e,l}[f(X^{(r)}+\Delta_{e,l})-f(X^{(r)})]\\
		&=\E_{e,l}\Big[\Big(r^k(Z_k^{(r)}+e_k^{(l)})\Big)^n\Big(\Lambda_{N-n}\Big(X^{(r)}\Big) +\Big(\frac{T_{e,l}}{\alpha_l}\Big)^{N-n}\Big)\Big(\Lambda_1\Big(X^{(r)}\Big) +\frac{T_{e,l}}{\alpha_l}\Big) \\
		&\qquad\qquad- \Big(r^kZ_k^{(r)}\Big)^n\Lambda_{N-n}\Big(X^{(r)}\Big)\Lambda_1\Big(X^{(r)}\Big)\Big]\\
		&\geq\E_{e,l}\Big[\Big(\Big(r^k(Z_k^{(r)}+e_k^{(l)})\Big)^n-\Big(r^kZ_k^{(r)}\Big)^n\Big)\Lambda_{N-n}\Big(X^{(r)}\Big)\Big(\Lambda_1\Big(X^{(r)}\Big) +\frac{1}{\alpha_l}\Big)\\
		&\qquad\qquad+ \Big(r^kZ_k^{(r)}\Big)^n\Lambda_{N-n}\Big(X^{(r)}\Big)/{\alpha_l}\Big]\\
		&\geq \E_{e,l}\Big[\Big(r^kZ_k^{(r)}\Big)^n\Lambda_{N-n}\Big(X^{(r)}\Big)\Big]/{\alpha_l}.
	\end{align*}
	
	For the ``Palm'' expectation w.r.t.\ service completions, we have
	\begin{align*}
		&\E_{s,l}[f(X^{(r)}+\Delta_{s,l})-f(X^{(r)})]\\
		&=\E_{s,l}\Big[\Big(r^k(Z_k^{(r)}-e_k^{(l)}+\xi^{(l)}_k)\Big)^n\Big(\Lambda_{N-n}(X^{(r)}) +\Big(\frac{T_{s,l}}{\mu_l^{(r)}}\Big)^{N-n}\Big)\Big(\Lambda_1(X^{(r)}) +\frac{T_{s,l}}{\mu_l^{(r)}}\Big) \\
		&\qquad\qquad- \Big(r^kZ_k^{(r)}\Big)^n\Lambda_{N-n}(X^{(r)})\Lambda_1(X^{(r)})\Big]\\
		&\geq\E_{s,l}\Big[\Big(\Big(r^k(Z_k^{(r)}-e_k^{(l)}+\xi^{(l)}_k)\Big)^n-\Big(r^kZ_k^{(r)}\Big)^n\Big)\Lambda_{N-n}\Big(X^{(r)}\Big)\Big(\Lambda_1\Big(X^{(r)}\Big) +\frac{1}{\mu_l^{(r)}}\Big) \\
		&\qquad\qquad+ \Big(r^kZ_k^{(r)}\Big)^n\Lambda_{N-n}\Big(X^{(r)}\Big)/{\mu_l^{(r)}}\Big].
	\end{align*}
	We again apply the mean value theorem to bound the first term on the RHS lower bound, and obtain that $\exists\omega_l\in(0,1)$, such that
	\begin{align*}
		&\E_{s,l}\Big[\Big(\Big(r^k(Z_k^{(r)}-e_k^{(l)}+\xi^{(l)}_k)\Big)^n-\Big(r^kZ_k^{(r)}\Big)^n\Big)\Lambda_{N-n}\Big(X^{(r)}\Big)\Big(\Lambda_1\Big(X^{(r)}\Big) +\frac{1}{\mu_l^{(r)}}\Big)\Big]\\
		&=nr^{k}\E_{s,l}\Big[(-e_k^{(l)}+\xi^{(l)}_k)\Big(r^k(Z_k^{(r)}+\omega_l(-e_k^{(l)}+\xi^{(l)}_k))\Big)^{n-1}\Lambda_{N-n}\Big(X^{(r)}\Big)\Big(\Lambda_1\Big(X^{(r)}\Big) +\frac{1}{\mu_l^{(r)}}\Big)\Big].
	\end{align*}
	Recall the property that
	\begin{equation*}
		\Lambda_{N-n}(X^{(r)})\Lambda_1(X^{(r)})\leq\Lambda_1^{N-n+1}(X^{(r)})\leq (2J)^{N-n}\Lambda_{N-n+1}(X^{(r)})
	\end{equation*}
	Together with the induction hypothesis that $\E_{s,l}[(r^kZ_k^{(r)})^{n-1}\Lambda_{N-n+1}(X^{(r)})]\leq F_{k,n-1}$, we are able to control the first term.
	
	For the remaining term $\E_{s,l}[(r^kZ_k)^{n-1}\Lambda_{N-n}(X^{(r)})]$, this should be able to be controlled by the induction hypothesis $F_{k,n-1}$.
	
	Combining, we get
	\begin{align*}
		&\sum_{l\in\cale}\E_{e,l}\Big[\Big(r^kZ_k^{(r)}\Big)^n\Lambda_{N-n}\Big(X^{(r)}\Big)\Big]/{\alpha_l}
		+\sum_{l\in\calk}\E_{s,l}\Big[\Big(r^kZ_k^{(r)}\Big)^n\Lambda_{N-n}\Big(X^{(r)}\Big)\Big]/{\mu_l^{(r)}}\\
		&\leq \Big(2J + (N-n)(2J)^{N-n-1}\Big)E_{k,n} \\
		&+ \sum_{l\in\calk}nr^{k}\E_{s,l}\Big[(-e_k^{(l)}+\xi^{(l)}_k)\Big(r^k(Z_k^{(r)}+\omega_l(-e_k^{(l)}+\xi^{(l)}_k))\Big)^{n-1}\Lambda_{N-n}\Big(X^{(r)}\Big)\Big(\Lambda_1\Big(X^{(r)}\Big) +\frac{1}{\mu_l^{(r)}}\Big)\Big].
	\end{align*}
	
	As argued, this upper bound should be finite, and hence adjusting the coefficients on the LHS, we prove the desired inequality
	\begin{equation*}
		\sum_{l\in\cale}\E_{e,l}\Big[\Big(r^kZ_k^{(r)}\Big)^n\Lambda_{N-n}\Big(X^{(r)}\Big)\Big]+\sum_{l\in\calk}\E_{s,l}\Big[\Big(r^kZ_k^{(r)}\Big)^n\Lambda_{N-n}\Big(X^{(r)}\Big)\Big]\leq F_{k,n}.
	\end{equation*}
	
	\section{Additional Properties of Helper Quantities}
	
	The helper quantities include $u^{(k)}$ and $R$ matrix. In this section, we discuss useful properties of these quantities that may help with further analyses.

	\subsection{Properties of \texorpdfstring{$u^{(k)}$}{u(k)} vectors}
	
	Recall that the $u^{(k)}$ vectors are defined as follows, 
	\begin{equation*}
		u^{(k)}=\begin{bmatrix}
			u_L^{(k)}\\u_H^{(k)}
		\end{bmatrix},\quad\text{where}\quad
		u_L^{(k)}=\begin{bmatrix}
			w_{1:k-1,k}\\1\\0
		\end{bmatrix}\quad\text{and}\quad
		u_H^{(k)}=-A_H^{-\top}A_{LH}^\top u_L^{(k)},
	\end{equation*}
	where $A_H$ and $A_{LH}$ are submatrices of matrix $A$ defined as
	\begin{equation*}
		A=(I-P^\top)\diag(\mu)(I-B).
	\end{equation*}
	Hence, we could alternatively write $u^{(k)}$ as
	\begin{equation*}
		u^{(k)}=\begin{bmatrix}
			I\\-A_H^{-\top}A_{LH}^\top
		\end{bmatrix}u_L^{(k)}.
	\end{equation*}
	
	Next, we present a lemma that describes an important property of the $u^{(k)}$ vectors, which is crucial for proving the uniform moment bound in Appendix~\ref{sec:moment-bound-proof}.
	\begin{lemma}
		\label{lem:station-property}
		Given $k\in\call$, for any $l\in\calk$, we have
		\begin{equation}
			\label{eq:u-mu-relationship}
			(u_l^{(k)}-\sum_{l'\in\calk}P_{ll'}u_{l'}^{(k)})\mu_l = (u_{s(l)}^{(k)}-\sum_{l'\in\calk}P_{s(l)l'}u_{l'}^{(k)})\mu_{s(l)},
		\end{equation}
		where $s(l)$ denotes the station number of class $l$ and also corresponds to the class of the lowest priority at that station by our indexing convention. 
	\end{lemma}
	
	\begin{proof}
		The desired property follows directly from the computation below.
		\begin{align}
			(u^{(k)})^\top(I-P^\top)\diag(\mu)(I-B)&=(u^{(k)})^\top A\nonumber\\
			&=(u_L^{(k)})^\top\begin{bmatrix}
				I & -A_{LH}A_H^{-1}
			\end{bmatrix}\begin{bmatrix}
				A_L & A_{LH}\\A_{HL} & A_H
			\end{bmatrix}\nonumber\\
			&=(u_L^{(k)})^\top\begin{bmatrix}
				A_L-A_{LH}A_H^{-1}A_{HL} & 0
			\end{bmatrix}\nonumber\\
			&=(u_L^{(k)})^\top \begin{bmatrix}
				(I-Q)^\top\diag(\mu_L) & 0
			\end{bmatrix}.
			\label{eq:u-relationship}
		\end{align}
		Recall the construction of $B$ matrix, \eqref{eq:u-relationship} thus implies that
		\begin{equation*}
			(u_l^{(k)}-\sum_{l'\in\calk}P_{ll'}u_{l'}^{(k)})\mu_l = (u_{l-}^{(k)}-\sum_{l'\in\calk}P_{l-l'}u_{l'}^{(k)})\mu_{l-},
		\end{equation*}
		which subsequently gives the desired observation, that for $l\in\calk$,
		\begin{equation}
			(u_l^{(k)}-\sum_{l'\in\calk}P_{ll'}u_{l'}^{(k)} )\mu_l= (u_{s(l)}^{(k)}-\sum_{l'\in\calk}P_{s(l)l'}u_{l'}^{(k)})\mu_{s(l)}.
		\end{equation}
		As such, we have completed the proof.
	\end{proof}
	
	Therefore, re-organizing the key observation in~\eqref{eq:u-mu-relationship}, we have
	\begin{equation}
		\diag(\mu)(I-P)u^{(k)}=C^\top\diag(\mu_L)\begin{bmatrix}
			0_{1:k-1}\\1-w_{kk}\\-{w_{k+1:L,k}}
		\end{bmatrix},
	\end{equation}
	and hence we obtain the following alternative expression of $u^{(k)}$ vector.
	\begin{equation}
		\label{eq:u-alternative formulation}
		u^{(k)}=(I-P)^{-1}MC^\top\diag(\mu_L)\begin{bmatrix}
			0_{1:k-1}\\1-w_{kk}\\-{w_{k+1:L,k}}
		\end{bmatrix}.
	\end{equation}
	
	Additionally, following the construction of $w_{1:k-1,k}$ and definition of $u_L^{(k)}$ vector, we note that
	\begin{align*}
		(I-Q)u_L^{(k)}&=\begin{bmatrix}
			I-Q_{1:k-1,1:k-1} & -Q_{1:k-1,k} & Q_{1:k-1,k+1:L}\\
			-Q_{k,1:k-1} & 1-Q_{k,k} & -Q_{k,k+1:L}\\
			-Q_{k+1:L,1:k-1} & -Q_{k+1:L,k} & I-Q_{k+1:L,k+1:L}
		\end{bmatrix}\begin{bmatrix}
			(I-Q_{1:k-1,1:k-1})^{-1}Q_{1:k-1,k}\\1\\0_{k+1:L}
		\end{bmatrix}\\
		&=\begin{bmatrix}
			0_{1:k-1}\\1-w_{kk}\\-w_{k+1:L,k}
		\end{bmatrix}.
	\end{align*}
	Hence, we obtain another expression of $u^{(k)}$ vector, 
	\begin{equation}
		\label{eq:u-formula}
		u^{(k)}=(I-P)^{-1}MC^\top\diag(\mu_L)(I-Q)u_L^{(k)}.
	\end{equation}
	
	\subsection{Equivalence of Reflection Matrices \texorpdfstring{$R$, $\tilde{R}$ and $\bar{R}$}{R, R-tilde and R-bar}}
	\label{sec:equivalent-reflection}
	
	In this section, we discuss the properties of the matrix $R$, which often is the reflection matrix of the corresponding SRBM limit. In the literature, several alternative definitions of the reflection matrix exist, including $\bar{R}$ as defined in~\cite{DaiYehZhou1997}
	\begin{equation}
		\bar{R}=(I+G)^{-1},\quad
		G=CM(I-P^\top)^{-1}P^\top\begin{bmatrix}
			\diag(\mu_L)\\0
		\end{bmatrix},
		\label{eq:r-bar1997}
	\end{equation}   
	and $\tilde{R}$ as defined in~\cite{BravDaiMiya2023}
	\begin{equation}
		\label{eq:r-tilde-2024}
		\tilde{R}=A_L-A_{LH}A_H^{-1}A_{HL}.
	\end{equation}
	In this section, we will connect these definitions and demonstrate the equivalence of these matrices up to multiples of positive diagonal matrices.
	
	\begin{lemma}
		$R$, $\tilde{R}$ and $\bar{R}$ are equivalent up to multiplies of positive diagonal matrices. Specifically, we have
		\begin{align}
			\tilde{R}&=R\diag(\mu_L),\label{eq:r-to-rtilde}\\
			\bar{R}&=\diag(m_L)\tilde{R}=\diag(m_L)R\diag(\mu_L).\label{eq:r-to-rbar}
		\end{align}
		
	\end{lemma}
	\begin{proof}
		
		The equivalence between $R$ and $\tilde{R}$ in~\eqref{eq:r-to-rtilde} is straightforward and can be obtained directly from the definition of the two matrices. Therefore, we focus on showing the equivalence relationship in~\eqref{eq:r-to-rbar}.
		
		Proving~\eqref{eq:r-to-rbar} is equivalent to showing that $diag(\mu_L)=\Tilde{R} \bar{R}^{-1},$
		which is easier to approach, for we have a closed form expression of $\bar{R}^{-1}$,
		\begin{align*}
			\bar{R}^{-1}&=I+CM(I-P^\top)^{-1}P^\top\begin{bmatrix}
				\diag(\mu_L)\\0
			\end{bmatrix}\\
			&=I+CM\Big((I-P^\top)^{-1}-I\Big)\begin{bmatrix}
				\diag(\mu_L)\\0
			\end{bmatrix}\\
			&=CM(I-P^\top)^{-1}\begin{bmatrix}
				\diag(\mu_L)\\0
			\end{bmatrix}.
		\end{align*}
		
		Making use of~\eqref{eq:u-formula}, we see that
		\begin{align*}
			\bar{R}^{-\top}\diag(\mu_L)(I-Q)u_L^{(k)}&=\begin{bmatrix}
				\diag(\mu_L)&0
			\end{bmatrix}(I-P)^{-1}MC^\top\diag(\mu_L)(I-Q)u_L^{(k)}\\
			&=\begin{bmatrix}
				\diag(\mu_L)&0
			\end{bmatrix}u^{(k)}\\
			&=\diag(\mu_L)u_L^{(k)}.
		\end{align*}
		Alternatively, the above matrix-vector multiplication can be written as
		\begin{equation}
			\label{eq:reflection-product-properties}
			\sum_{j=1}^Ju_j^{(k)}\Big[\bar{R}^{-\top}\diag(\mu_L)(I-Q)\Big]_{\cdot,j}=\sum_{j=1}^Ju_j^{(k)}\mu_je^{(j)},
		\end{equation}
		where $\Big[\bar{R}^{-\top}\diag(\mu_L)(I-Q)\Big]_{\cdot,j}$ denotes the $j$-th column of the matrix $\bar{R}^{-\top}\diag(\mu_L)(I-Q)$, and $e^{(j)}\in\R^J$ denotes the basis vector where all entries are $0$ except for the $j$-th entry being 1.

		Since~\eqref{eq:reflection-product-properties} holds for all $k=1,\ldots,J$, we iteratively substitute $u_L^{(k)}$ as defined in~\eqref{eq:u-vec-def} and examine its properties.

		When $k=1$, we have
		\begin{equation*}
			u_L^{(k)}=e^{(1)}=\begin{bmatrix}
				1 & 0 &\cdots &0
			\end{bmatrix}^\top.
		\end{equation*}
		Therefore, 
		\begin{equation*}
			\Big[\bar{R}^{-\top}\diag(\mu_L)(I-Q)\Big]_{\cdot,1}=\begin{bmatrix}
				\mu_1 & 0 & \cdots & 0
			\end{bmatrix}^\top.
		\end{equation*}
		
		When $k=2$, we have
		\begin{equation*}
			u_L^{(k)}=\begin{bmatrix}
				w_{12} & 1 & 0 &\cdots &0
			\end{bmatrix}^\top.
		\end{equation*}
		Hence,
		\begin{align*}
			&\sum_{j=1}^Ju_j^{(k)}\Big[\bar{R}^{-\top}\diag(\mu_L)(I-Q)\Big]_{\cdot,j}\\
			&=w_{12}\Big[\bar{R}^{-\top}\diag(\mu_L)(I-Q)\Big]_{\cdot,1}+\Big[\bar{R}^{-\top}\diag(\mu_L)(I-Q)\Big]_{\cdot,2}\\
			&=w_{12}\begin{bmatrix}
				\mu_1 & 0 & \cdots & 0
			\end{bmatrix}^\top+\Big[\bar{R}^{-\top}\diag(\mu_L)(I-Q)\Big]_{\cdot,2}\\
			&=\sum_{j=1}^Ju_j^{(k)}e^{(j)}\\
			&=w_{12}\begin{bmatrix}
				\mu_1 & 0 & \cdots & 0
			\end{bmatrix}^\top+\begin{bmatrix}
				0 & \mu_2 & 0 & \cdots & 0
			\end{bmatrix}^\top.
		\end{align*}
		Therefore, we conclude that
		\begin{equation*}
			\Big[\bar{R}^{-\top}\diag(\mu_L)(I-Q)\Big]_{\cdot,2}=\begin{bmatrix}
				0 & \mu_2 & 0 & \cdots & 0
			\end{bmatrix}^\top.
		\end{equation*}
		Recursively, we obtain that
		\begin{equation*}
			\bar{R}^{-\top}\diag(\mu_L)(I-Q)=\diag(\mu_L).
		\end{equation*}
		Substituting~\eqref{eq:r-to-rtilde} into the above equation and taking transpose, we have shown that
		\begin{equation*}
			\diag(\mu_L)=\tilde{R}\bar{R}^{-1}.
		\end{equation*}
	\end{proof}
	
	Since these reflection matrices are equivalent up to multiples of positive diagonal matrices, the properties of $\mathcal{P}$-matrices apply to all of them. In other words, if any of $R$, $\tilde{R}$, or $\bar{R}$ is a $\mathcal{P}$-matrix, then they all are $\mathcal{P}$-matrices. As proven by the authors in~\cite{DaiYehZhou1997}, $\bar{R}$ for re-entrant lines under LBFS or FBFS policies is a $\mathcal{P}$-matrix. Consequently, we can conclude that our definition of the reflection matrix $R$ is also a $\mathcal{P}$-matrix.

	\newpage
	\bibliography{citation}

\end{document}